\documentclass{article}

\usepackage[english]{babel}
\usepackage{amssymb,amsmath,amsthm,dsfont}
\usepackage{textcomp, gensymb}
\usepackage[utf8]{inputenc}
\usepackage[title]{appendix}
\usepackage{color}
\usepackage{graphicx}
\usepackage{cite}
\usepackage{pdfpages}
\usepackage{hyperref}
\usepackage[capitalise]{cleveref}
\usepackage{mathtools}
\mathtoolsset{multlined-width=0.9\displaywidth}
\usepackage{enumitem}
\usepackage{empheq}
\usepackage{nicefrac}
\usepackage{stmaryrd}
\usepackage{authblk}
\usepackage{float}
\usepackage{esint}
\usepackage{multicol}
\usepackage{subcaption}
\usepackage{tikz}
\usepackage[ruled,vlined]{algorithm2e}
\usepackage[pagewise]{lineno}

\usepackage[letterpaper,top=2cm,bottom=2cm,left=3cm,right=3cm,marginparwidth=1.75cm]{geometry}

\usepackage{csquotes}

\newcommand{\eps}[1]{{#1}_{\varepsilon}}
\newcommand{\R}{\mathbb{R}}
\newcommand{\N}{\mathbb{N}}

\newcommand{\ezr}{{\e}}

\newcommand{\e}{\varepsilon}
\newcommand{\di}[1]{\,\mathrm{d}#1}
\newcommand{\dive}{\operatorname{div}}
\newcommand{\ID}{\operatorname{id}}
\newcommand{\dist}{\operatorname{dist}}
\newcommand{\indexI}{\mathcal{I}_\e}

\newcommand{\ddt}{\frac{\operatorname{d}}{\operatorname{d}t}}

\newcommand{\leslam}{\lesssim_{\lambda}}

\newcommand{\cext}{0}
\newcommand{\cequ}{c^{\mathrm{eq}}}
\newcommand{\dmin}{{\underline d}}
\newcommand{\dmax}{{\overline d}}

\newcommand{\LtLx}[3]{%
  L_t^{#1}L_x^{#2}(#3)%
}
\renewcommand{\LtLx}[3]{%
  L^{#1}(S;L^{#2}(#3))%
}

\newcommand{\dd}{\, \mathrm{d}}
\renewcommand{\d}{\mathrm{d}}
\newcommand{\Id}{\mathrm{Id}}
\newcommand{\1}{\mathbf1}
\DeclareMathOperator{\dv}{div}
\DeclareMathOperator{\Bog}{Bog}
\DeclareMathOperator{\supp}{supp}
\DeclareMathOperator{\curl}{curl}

\newtheorem{theorem}{Theorem}[section]
\newtheorem{lemma}[theorem]{Lemma}
\newtheorem{proposition}[theorem]{Proposition}
\newtheorem{corollary}[theorem]{Corollary}

\theoremstyle{definition}
\newtheorem{definition}[theorem]{Definition}
\newtheorem{example}[theorem]{Example}

\theoremstyle{remark}
\newtheorem{remark}[theorem]{Remark}

\crefname{theorem}{Theorem}{Theorems}
\crefname{lemma}{Lemma}{Lemmas}
\crefname{proposition}{Proposition}{Propositions}
\crefname{corollary}{Corollary}{Corollaries}
\crefname{definition}{Definition}{Definitions}
\crefname{example}{Example}{Examples}
\crefname{remark}{Remark}{Remarks}
\AddToHook{env/theorem/begin}{\crefalias{theorem}{theorem}}
\AddToHook{env/lemma/begin}{\crefalias{theorem}{lemma}}
\AddToHook{env/proposition/begin}{\crefalias{theorem}{proposition}}
\AddToHook{env/corollary/begin}{\crefalias{theorem}{corollary}}
\AddToHook{env/definition/begin}{\crefalias{theorem}{definition}}
\AddToHook{env/example/begin}{\crefalias{theorem}{example}}
\AddToHook{env/remark/begin}{\crefalias{theorem}{remark}}

\definecolor{darkblue}{rgb}{0,0,0.7} 
\definecolor{darkred}{rgb}{0.9,0.1,0.1}
\definecolor{darkgreen}{rgb}{0,0.5,0}

\makeatletter
\renewenvironment{proof}[1][\proofname]{%
  \par
  \pushQED{\qed}%
  \normalfont
  \topsep6\p@\@plus6\p@\relax
  \trivlist
  \item[\hskip\labelsep
        \bfseries\itshape #1\@addpunct{.}]
  \ignorespaces
}{%
  \popQED
  \endtrivlist
  \@endpefalse
}
\makeatother

\title{Reactive Flow around Spherically Evolving Particles in Critical and Supercritical Dilute Regimes}
\author{Michael Eden\thanks{michael.eden@ur.de} }
\author{Richard M.~Höfer\thanks{richard.hoefer@ur.de}}

\affil{Department of Mathematics, University of Regensburg, Germany}

\begin{document}
\maketitle

\begin{abstract}    
    We study a coupled Stokes-reaction-diffusion-advection system in a three dimensional domain perforated by a large number of evolving spherical inclusions with radii of order $\e^\alpha$ for $\alpha\in(1,3]$.
    Their centers are separated on the order of $\e$ but are not assumed to form a periodic lattice.
    The inclusions grow or shrink through an interfacial adsorption-desorption mechanism, so that the evolution of the geometry is coupled to the Stokes-reaction-diffusion-advection system.
    
    We first establish well-posedness on arbitrary large time intervals for sufficiently small $\e$.
    We then derive sharp quantitative homogenization limits in terms of the limiting empirical measure describing the spatial distribution and size of the inclusions.
    
    Its evolution is governed by a continuity equation in the radius variable, and the convergence of the empirical measures is controlled in the $2$-Wasserstein distance.
    In the critical case, $\alpha=3$, the Stokes system converges to a Brinkman system, while non-vanishing concentration boundary layers modify the microscopic exchange law for the reaction-diffusion-advection equation.
    For $\alpha\in(1,3)$, the limit is of Darcy type and the original exchange law is retained.
\end{abstract}



\tableofcontents

\section{Introduction}
Reactive transport and flows in complex geometries is often accompanied by dynamic changes of the underlying geometry.
Such changes may be caused by processes such as adsorption, precipitation and dissolution, colloidal depositions, swelling, or carbonation \cite{noorden_crystal_2009,eden_multiscale_2022,gahn_homogenization_2023,BizmarkEtAl2020,BaqerChen2022}.
Similar setups are also found in the context of phase transitions or biofilm growth \cite{eden_thermo-elasticity_2026,SchulzKnabner2017}.
In such scenarios, the interaction between transport processes and the changing geometry naturally leads to strongly coupled multiscale problems with moving boundaries.

Apart from porous media, dilute particle suspensions are of particular interest. Here, the characteristic size of individual precipitates or deposits is small when compared with their mutual distance \cite{niethammer_local_2020}.
Such scenarios may arise, for instance, during nucleation and the early stages of precipitation \cite{jaho_experimental_2016,nooraiepour_probabilistic_2021}.
Although their total volume fraction is small, a large number of such inclusions may still have a substantial cumulative effect on both the fluid flow and interfacial mass exchange \cite{DesvillettesGolseRicci08,HillairetMoussaSueur19,jager_homogenization_2011}.

In the present paper, we study such a coupled system in a three-dimensional domain containing a large number of small spherical obstacles.
These obstacles are allowed to evolve radially due to adsorption and desorption processes.
More precisely, we consider $N_\e=\e^{-3}$ balls with fixed  centers $z_{\e,i}\in\Omega$ and time-dependent radii $\e^\alpha r_{\e,i}(t)$, where $\alpha\in(1,3]$.
The corresponding time-dependent fluid domain is
\[
    \Omega_\e(t)
    :=
    \Omega\setminus
    \bigcup_{i=1}^{N_\e}
    B_{\e^\alpha r_{\e,i}(t)}(z_{\e,i}),
    \qquad
    N_\e=\e^{-3},
    \qquad
    \alpha\in(1,3]
\]
for a bounded domain $\Omega\subset\mathbb R^3$.
The centers $(z_{\e,i})$ are not required to form a periodic lattice, but are assumed to be uniformly separated from each other and $\partial\Omega$ on the $\e$-scale.
The inclusions have radii of order $\e^\alpha$, where $\alpha\in(1,3]$, and are therefore asymptotically smaller than their mutual distances.
This is the characteristic geometry of a dilute perforated medium.

The model we are considering here, which is described in full detail in Section \ref{sec:setting_main_results}, couples the fluid flow, the transport of a dissolved species, and the evolution of the obstacles due to adsorption and desorption.
More precisely, the fluid velocity $u_\e$ and pressure $p_\e$ are governed by a quasi-steady incompressible Stokes system,
\[
    -\e^{3-\alpha}\Delta u_\e+\nabla p_\e=h_\e,
    \qquad
    \dive u_\e=0
    \qquad\text{in }\Omega_\e(t),
\]
where $h_\e$ is a given source term, together with a no-slip condition on the moving obstacle boundaries $\Gamma_{\e,i}(t) := \partial B_{\e^\alpha r_{\e,i}(t)}(z_{\e,i})$.
The concentration $c_\e$ satisfies an advection-diffusion-reaction equation,
\[
    \partial_t c_\e-\dive(\nabla c_\e-u_\e c_\e)=F(c_\e)
    \qquad\text{in }\Omega_\e(t).
\]
At the obstacle surfaces $\Gamma_{\e,i}(t)$, the dissolved species is exchanged with the solid phase through the adsorption--desorption law
\[
    -\nabla c_\e\cdot\nu_{\e,i}
    =\e^{3-2\alpha}G(c_\e,r_{\e,i})\qquad\text{on }\Gamma_{\e,i}(t).
\]
Conservation of mass at each obstacle couples this interfacial exchange back to the geometry via the radius evolution
\[
    \partial_t r_{\e,i}
    =
    -\fint_{\Gamma_{\e,i}(t)}
    G(c_\e,r_{\e,i})\,\di{\gamma} .
\]
Hence, the concentration drives the evolution of the obstacles, while the changing geometry in turn affects both the transport process and the fluid flow.
Our main objective is to determine the effective behavior of the solution $(u_\e,c_\e, (r_{\e,i})_{i =1}^{N_\e})$ to this fully coupled system as $\e\to0$ and, in particular, to understand how the macroscopic dynamics depend on the scaling exponent $\alpha$.
As we shall see in our analysis, the same exponent $\alpha=3$ is critical for two distinct mechanisms: the hydrodynamic resistance of the obstacles and the local concentration perturbation generated by the interfacial exchange.

To describe the homogenization limit, we encode the spatial distribution
of the obstacles and their radii by the empirical measure
\[
    f_\e(t)
    :=
    \e^3\sum_{i=1}^{N_\e}
    \delta_{z_{\e,i}}\otimes\delta_{r_{\e,i}(t)},
\]
which is, for every time $t$, a probability measure on $\Omega \times [0,\infty)$.
Its limit measure $f(t,x,\d\sigma)$ describes the local distribution of obstacle radii.
In the critical case $\alpha=3$, the collective hydrodynamic resistance of the obstacles remains of order one which leads to a Brinkman-type equation,
\[
    -\Delta u+\nabla p
    +6\pi u\int_0^\infty \sigma f(\cdot,\d\sigma)
    =h,
    \qquad
    \dive u=0.
\]
At the same time, $\alpha=3$ is also critical for the interfacial exchange: 
Substantial concentration boundary layers form around the obstacles and modify the microscopic exchange law $G$ into an effective law $H$, characterized implicitly by the algebraic equation
\begin{equation}\label{int:H}
    H(c,\sigma)=\sigma G(c+H(c,\sigma),\sigma).
\end{equation}
Consequently, the corresponding macroscopic reaction diffusion problem takes the form
\[
\partial_tc-\dive(\nabla c-uc)=F(c)+4 \pi\int_0^\infty \sigma H(c,\sigma) f(\cdot,\d\sigma)
\]
while the particle-size distribution evolves according to the continuity equation
\[
    \partial_t f
    -
    \partial_\sigma
    \left(
        \frac{H(c,\sigma)}{\sigma}f
    \right)
    =0.
\]
For $\alpha\in(1,3)$, the collective hydrodynamic resistance of the obstacles diverges such that the fluid equation converges to the Darcy-type law
\[
    \nabla p
    +6\pi u\int_0^\infty \sigma f(\cdot,\d\sigma)
    =h,
    \qquad
    \dive u=0.
\]
In contrast to the critical case, the concentration boundary layers vanish
in this regime, so that the original exchange law $G$ enters the
macroscopic model directly. Accordingly, the effective exchange and radius
evolution are governed by
\[
    \partial_t c-\dive(\nabla c-u c)
    =F(c)+
    \int_0^\infty
        4\pi\sigma^2 G(c,\sigma)\,
        f(\cdot,\cdot,\d\sigma),
    \qquad
    \partial_t f
    -
    \partial_\sigma
    \left(
        G(c,\sigma)f
    \right)
    =0.
\]

\subsection{Previous results}

Our problem is at the intersection of two different directions in homogenization theory.
The first deals with critical effects generated by small inclusions in fixed perforated domains, both through their hydrodynamic resistance and through interactions at their boundaries.
The second concerns homogenization in evolving perforated domains, where transport processes are coupled to changes of the microscopic geometry.
We briefly discuss the relevant literature in these two directions and place the present work in this context.


\paragraph{Homogenization in perforated domains with fixed microstructure.}
When the perforated domain is not evolving but fixed,  critical-size effects are classical, both for scalar problems and for viscous flow, with a vast body of literature.
For the Poisson problem, Cioranescu and Murat \cite{CioranescuMurat82,CioranescuMurat83,CioranescuMurat97} showed that the collective capacity of many small holes has a non-trivial effect on the  homogenization limit in the form of an additional zero-order term, the celebrated ``strange term coming from nowhere''.
For the Stokes problem, three regimes have been identified depending on the parameter $\alpha$ in the seminal works by Tartar and Allaire  \cite{Tartar, allaire_homogenization_1991,Allaire1991-fp}:
\footnote{We state the regimes for the 3D case. 
More generally, for $d\ge3$, the critical hole size established in those works is of order $\e^{d/(d-2)}$.}
\begin{itemize}
    \item In the subcritical regime, $\alpha>3$, the homogenization limit is the unchanged Stokes equation.
    \item In the critical regime, $\alpha=3$, the homogenization limit is governed by the Brinkman equation.
    \item In the supercritical case, $\alpha \in [1,3)$, the homogenization limit is given by Darcy's law.
\end{itemize}
These results have been extended in various directions over the last decades, including non-homogeneous Dirichlet boundary conditions \cite{DesvillettesGolseRicci08,MecherbetHillairet20, HillairetMoussaSueur19}, non-periodic (random) polydispersed configurations of particles \cite{GiuntiHoefer19,  CarrapatosoHillairet20, HoferJansen24, Giunti21}, convergence rates \cite{JingLuPrange} and (compressible) Navier-Stokes equations \cite{Feireisl2016,  Masmoudi02,  Hoefer23, HoferNecasovaOschmann}, which is by no means a complete list of these results. 

A related class of critical-size effects occurs for diffusion problems with flux conditions on the boundaries of small inclusions.
Early works by Kaizu \cite{Kaizu1985Robin,Kaizu1991,Kaizu1989} showed that boundary interactions on many vanishing holes may survive in the homogenization limit as an effective bulk contribution.
In this setting, the relevant balance generally depends on both the size of the perforations and the scaling of the boundary condition.
For nonlinear flux conditions on dilute perforations, it is known that the interplay between the geometric and boundary scales may affect not only the strength but also the form of the effective reaction term \cite{Jager_neuss_shaposhnikova2011,Zubova_shaposhnikova13,PerezZubovaShaposhnikova2014}.
This phenomenon is closely related to the replacement of the microscopic exchange law $G$ by the effective law $H$ in our critical regime, see \cref{int:H}.

\paragraph{Homogenization in evolving perforated domains.}
The second relevant topic are homogenization problems in which the microscopic geometry itself is allowed to evolve.
The literature in this area typically consider periodic pore-scale microstructures, corresponding to the non-dilute case $\alpha=1$ and thus to a non-vanishing solid volume fraction, rather than the dilute regime investigated in this work.
Early mathematical contributions by van Noorden and coauthors considered crystal dissolution and precipitation models and derived effective models by formal asymptotic expansions \cite{noorden_crystal_2009,van_noorden_crystal_2009}.
A related line of work concerns homogenization in evolving microstructures whose motion is prescribed rather than coupled to the transported quantities \cite{Eden_Muntean2017,eden_homogenization_2019,GahnNeuss-RaduPop21,peter_homogenisation_2007}.
Rigorous homogenization results for reaction--diffusion and reaction--diffusion--advection problems coupled to an evolving microstructure were subsequently obtained in \cite{wiedemann_homogenisation_2023,gahn_homogenization_2023}.
Similar to these results,  a model for evolving microstructures inside thermo-elastic materials was considered in\cite{eden_thermo-elasticity_2026}.
The general strategy in these papers consists of the following steps: transform the evolving domains to a fixed reference geometry, establish well-posedness via fixed-point arguments, derive uniform estimates, and pass to the limit by two-scale convergence.
Additionally, a hard stopping criterion on the growth  of the inclusions is encoded to prevent collapse and touching of the balls.
An alternative modeling approach is provided by diffuse-interface or phase-field models, where the evolving fluid--solid interface is represented by an additional phase variable rather than tracked explicitly; see \cite{Eck2004,vanNoordenEck2011,BringedalVonWolffPop2020}.

The homogenization of Stokes flow in prescribed evolving pore geometries has been studied in \cite{wiedemann_homogenisation_2024,WiedemannPeter25}, leading to Darcy-type effective models with geometry-dependent coefficients.
Most closely related to the present work is \cite{gahn_rigorous_2024}, where the authors rigorously homogenize the coupled dynamics of a Stokes flow and a reaction-diffusion-advection equation in a domain perforated by evolving spheres.
The main difference to the present setting is that as in the aforementioned papers, the spheres in \cite{gahn_rigorous_2024} are arranged periodically and have radii of order $\e$ (corresponding to $\alpha = 1$).
The resulting macroscopic model is therefore described through effective coefficients determined by local cell problems, whereas the dilute microstructure considered here is encoded by the limiting distribution of particle radii.

\paragraph{Main novelties of the present paper.}
Against this background, the present work extends the analysis of evolving microstructures to the dilute regime.
More precisely, we consider inclusions of characteristic size $\e^\alpha$, $\alpha\in(1,3]$, whose mutual distance remains of order $\e$.
In contrast to the periodic pore-scale settings discussed above, the solid volume fraction therefore vanishes as $\e\to0$, and the inclusions are not required to form a periodic lattice.
As a consequence and in contrast to models described by a single macroscopic radius function $R(t,x)$, the limiting measure $f(t,x,\d\sigma) $ may encode genuinely polydisperse configurations, including continuous distributions of particle radii.
The classical single-radius description is recovered as the monodisperse special case in which
\[
f(t,x,\d\sigma)=\delta_{R(t,x)}(\d\sigma).
\]

The dilute scaling has several important consequences for both the microscopic problem and its homogenization.
The separation between the size of the obstacles and their mutual distances makes it possible to establish a global-in-time existence result in the following asymptotic sense: for every $T>0$, there exists $\e_T>0$ such that for all $\e\le\e_T$ the microscopic problem admits a unique solution over $(0,T)$.
Crucially, this is achieved without a hard stopping criterion as is required in several earlier works.
To prevent unctrolled growth and collapse of individual inclusions (i.e., their radii reaching zero) we only require some milder structural assumptions.
Moreover,  since the solid volume fraction vanishes as $\e\to0$, the evolution of the inclusions does not generate an additional volumetric source term in the effective fluid equation, and the limiting velocity remains divergence free.
Similarly, the diffusion coefficient of the concentration equation remains effectively unchanged.

\paragraph{Difficulties in the critical case $\alpha = 3$ and some ingredients of the proof.}
The cumulative hydrodynamic resistance and interfacial exchange are of order one in the critical regime $\alpha = 3$ and therefore survive in the homogenization limit.
The criticality with respect to the  interfacial exchange term is reflected as follows:
For the supercritical\footnote{We use the terminology \enquote{supercritical} because $\alpha\in(1,3)$ is supercritical for the Stokes equation, even though it is actually \emph{subcritical} for the exchange term.} case $\alpha\in(1,3)$, this term can be effectively treated as a compact perturbation of the 
advection–diffusion–reaction equation.
Indeed, through a trace estimate in the perforated domain (see \cref{lemma:trace_estimates}), we can essentially control the exchange term in terms of the $L^2$-norm of the concentration in the perforated domain.
For $\alpha = 3$, however, the interfacial term can only be controlled in by the full $H^1$-norm of the concentration in the perforated domain. 
In other words,  for $\alpha = 3$, the interfacial mass flux is concentrated at such small scales that the diffusion is not strong enough to prevent the formation of boundary layers in a small neighborhood of the obstacles.
This critical scaling also makes structural assumptions on the exchange law necessary for uniform control of the microscopic problem: without a dissipative mechanism, the interfacial flux may induce uncontrolled growth of the concentration and hence of the particles in very short times. In particular, there there would be no hope for a uniform a priori existence time as $\eps \to 0$.

These local perturbations generate the modified effective law $H$ in the critical regime (\cref{int:H}), whereas for $\alpha\in(1,3)$ the microscopic law $G$ passes directly to the limit.
They also lead to a loss of strong $H^1$-compactness for the uncorrected concentration ($\|\nabla c_\e-\nabla c\|_{L^2(\Omega_\e(t))}\nrightarrow0$ in general) and make a direct passage to the limit in the weak formulation considerably more delicate.
We therefore use a relative-energy argument that combines appropriately constructed local correctors with quantitative control of the evolving empirical measures in the 2-Wasserstein distance.
We emphasize that this approach also yields optimal quantitative convergence rates.
In the critical Brinkman regime, we obtain the optimal rate $\mathcal O(\e)$, whereas in the Darcy regime the rate deteriorates due to three effects:
First, if $\alpha $ is close to $3$, the viscous term vanishes only very slowly. Secondly, if $\alpha$ is close to $1$, the solid volume fraction vanishes only very slowly. 
Finally, in all cases, an additional  boundary-layer at $\partial \Omega$  appears that account for the loss of the microscopic no-slip condition in the Darcy limit.
A further benefit of the relative energy argument is that it allows us to work directly in the moving geometry when passing to the limit instead of the reference geometry.
This is an important distinction from \cite{gahn_rigorous_2024} and all the aforementioned homogenization results in evolving domains. 

\paragraph{Limitations and possible generalizations.}
As in \cite{gahn_rigorous_2024}, we restrict ourselves to spherical inclusions undergoing radial evolutions.
The present approach can be quite easily extended to more general fixed reference shapes as long as the evolution of each inclusion is still described by a scalar rescaling as, e.g., in \cite{eden_thermo-elasticity_2026}.
A natural next step would be to allow additional finite-dimensional shape parameters, for instance the three semi-axes of ellipsoidal inclusions, see for instance \cite{BurteaGerardVaret26} for such an approach in a related context.
In such settings, the limiting empirical measure would have to evolve also in the corresponding shape variables.
Beyond finite-dimensional shape parameterizations, one could also incorporate regularizing shape dynamics, for example through a volume-preserving curvature term.
Much more challenging, however, would be a fully coupled free-boundary evolution in which the local normal velocity is determined directly by the interfacial reaction (without averaging).
This general problem seems currently out of reach.

In view of applications to colloids, one is interested in allowing the particle centers to move.
One would then aim to derive a Vlasov equation for the effective particle evolution. 
However, even in the absence of growth and coupling to concentration dynamics, there are only preliminary results regarding the derivation of the Vlasov-Stokes equation from such microscopic descriptions, see \cite{HoferSchubert26}.

Two further natural extensions concern changes in the particle population and the presence of multiple particle-size scales.
Allowing inclusions to dissolve completely would require treating the loss of particles through the boundary $\sigma=0$ in the limiting size-distribution equation, while nucleation could be incorporated through additional source terms for $f$.
However, this is extremely delicate for the analysis of the microscopic system, where vanishing or newly created inclusions lead to local singularities because of topological changes of the fluid domain.
Moreover, one could consider particle populations with several characteristic sizes $\e^{\alpha_i}$, potentially leading to the simultaneous appearance of different effective hydrodynamic regimes, such as Darcy- and Brinkman-type contributions.

\subsection{Outline of the paper}

In \cref{sec:setting_main_results}, we introduce the microscopic model and its weak formulation, state the assumptions used throughout the paper, and formulate the main results.
These are:
\begin{itemize}
	\item \Cref{theorem:well_posedness}: Well-posedness of the microscopic problem on every fixed time interval for sufficiently small \(\varepsilon\). 
	\item \Cref{th:hom.critical}:  a quantitative homogenization result in the critical regime \(\alpha=3\), leading to a Brinkman-type limit system with a modified effective exchange law.  
	\item \Cref{thm:hom.supercritical}: a quantitative homogenization result for \(1<\alpha<3\), leading to a Darcy-type limit system with the original exchange law.
\end{itemize}

In \cref{ssec:maximum} we prove an a priori estimate of maximum principle type which is crucial for the long-time well-posedness of the microscopic system. This is easiest proven in the original moving geometry.

\Cref{sec:analysis} is devoted to the proof of \Cref{theorem:well_posedness}. 
We transform the moving geometry back to a fixed reference geometry, and apply the Banach fixed point theorem combined wih suitable a priori estimates for the coupled system for a \emph{prescribed} evolution of the geometry.

\Cref{th:hom.critical} is proved in \cref{sec:hom_crit} via a relative energy argument. The relative energy consists of two parts. The first part corresponds to the squared $L^2$ difference of the chemical concentration for which we need to introduce suitable correctors. The second part controls the squared $2$-Wasserstein distance between the empirical density of the obstacles and the limit density.
The fluid velocity does not enter the relative energy because it satisfies a quasistatic equation. To close a Gronwall argument on the relative energy we therefore show that the squared $H^1$ difference of suitably corrected fluid velocities is controlled by the squared Wasserstein distance.

 \Cref{thm:hom.supercritical} is proved in \cref{sec:hom_supercrit}. 
 The proof is largely analogous to he proof of \Cref{th:hom.critical}. 
 It is simpler in the sense that the interfacial mass flux is subcritical. 
 However,  it is complicated by the loss of the viscous term in the fluid velocity, which leads to the appeareance of boundary layers at $\partial \Omega$. 
 This requires the addition of a suitable boundary layer corrector. 

The appendices collect the technical results concerning the moving-domain transformation, uniform inequalities on the perforated domains, and the local Stokes correctors.

\section{Setting and main results}
\label{sec:setting_main_results}
Let $T>0$ and $S= (0,T)$ a fixed time horizon, $\Omega\subset\mathbb{R}^3$ be a bounded Lipschitz domain, $\alpha\in(1,3]$, and assume $0<\e\ll1$ with
$\e^{-1}\in\mathbb{N}$.\footnote{We just take $\e=\e_n:=\frac{1}{n}$ and suppress $n$. This is mainly done to simplify the notation in the limit procedure. More generally, we can have $\e^3N_\e\to\sigma$ with $\sigma>0$ which would lead to additional scaling factors in the limit problem.} We set
\[
    N_\e := \e^{-3},
    \qquad
    \mathcal{I}_\e := \{1,\ldots,N_\e\}.
\]
Furthermore, let $(z_{\e,i})_{\mathcal{I}_\e}\subset\Omega$ be a point cloud and, for all $t\in\overline{S}=[0,T]$, the radii $(r_{\e,i}(t))_{\mathcal{I}_\e}\subset(0,\infty)$ such that the closed balls $\overline{B}_{\e^\alpha r_{\e,i}(t)}(z_{\e,i})$ are pairwise disjoint and contained in $\Omega$ (the precise assumptions are given in \cref{ssec:assumpstions}).
Then, the fluid domain is defined as 
\begin{align} \label{eq:changing_domains}
    \Omega_\e(t)=\Omega\setminus\bigcup_{i\in\mathcal{I}_\e}\overline B_{\e^\alpha r_{\e,i}(t)}(z_{\e,i})
\end{align} 
with the convention $\Omega_\e:=\Omega_\e(0)$.
We also set $$\Gamma_{\e,i}(t)=\partial B_{\e^\alpha r_{\e,i}(t)}(z_{\e,i}), \qquad \Gamma_\e(t)=\bigcup_{\mathcal{I}_\e}\Gamma_{\e,i}(t)$$ and introduce the unit normal vector $\nu_{\e,i}(\gamma)=\frac{\gamma-z_{\e,i}}{|\gamma-z_{\e,i}|}$ for $\gamma\in\Gamma_{\e,i}$.
Finally, let \(\Sigma_1,\Sigma_2\subset\partial\Omega\) be disjoint, relatively open subsets with positive surface measure such that
\[
\partial\Omega=\overline{\Sigma_1}\cup\overline{\Sigma_2}\quad\text{and}\ |\overline{\Sigma_1}\cap\overline{\Sigma_2}|=0
\]
where $|\overline{\Sigma_1}\cap\overline{\Sigma_2}|$ denotes the surface measure.
We denote by \(n\) the outer unit normal to \(\Omega\).

\begin{figure}[ht] 
\centering
        \includegraphics[angle=90,width=.75\linewidth]{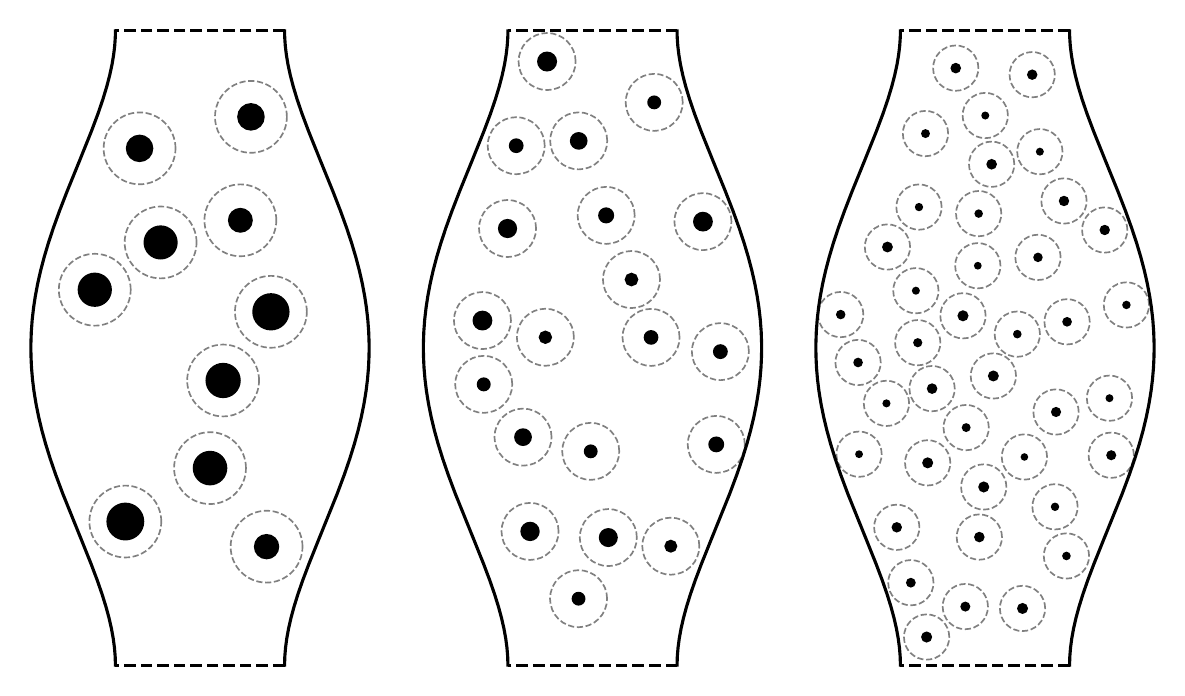}
 \caption{Representation of the geometry with randomly distributed, small holes of different sizes to the order of $\e^3$ for three choices of the parameter $\e$.
 Note that these images only show a cross-section of the three-dimensional setup.
 The dashed circles refer to the $\e$-sized cells separating the balls.
 For the Stokes system, the dashed part of the outer boundary is where inflow/outflow of fluid is possible with no-slip at the rest of the boundary.
 }
 \label{geometry}
\end{figure}

We model the fluid dynamics inside  $\Omega_\e(t)$ via the incompressible, quasi-steady Stokes system.
To that end, let $u_\e$ denote the fluid velocity and $p_\e$ the fluid pressure.

\begin{subequations}\label{system:stokes_eps}
\begin{empheq}[left=\empheqlbrace]{alignat=2}
    -\e^{3-\alpha}\Delta u_{\e}+\nabla p_{\e}&=h_\e&\quad&\text{in}\ \Omega_\e(t),\ t\in S,\label{eq:Stokes_equation}\\
    \dive u_{\e}&=0&\quad&\text{in}\ \Omega_\e(t),\ t\in S,\\
    u_{\e}& =\e^\alpha\partial_t r_{\e,i}\nu_{\e,i}&\quad&\text{on}\ \Gamma_{\e,i}(t),\ t\in S,\ i\in\mathcal{I}_\e, \label{inh.Dirichlet}\\
    \sigma_\e[u_\e,p_\e]n&=0&\quad&\text{on}\  \Sigma_1,\ t\in S,\\
    u_\e&=0&\quad&\text{on}\  \Sigma_2,\ t\in S.
\end{empheq}
\end{subequations}
Here, $h_\e$ is a given source term,  $\sigma_\e[u_\e,p_\e]=2\e^{3-\alpha}e(u_\e)-p_\e I$ is the Cauchy stress tensor and $e(u)=\frac12(\nabla u+\nabla u^T)$ the symmetric gradient.
Note that, using $\dive u_\e=0$, we can write \cref{eq:Stokes_equation} as
\[
-\dive\sigma_\e[u_\e,p_\e]=h_\e.
\]
The kinematic conditions $u_{\e} =\e^\alpha\partial_t r_{\e,i}\nu_{\e,i}$ model the no-slip condition where we emphasize that the unit normal vector $\nu_\e$ points outwards of the ball into the domain $\Omega_\e(t)$. 
Due to the incompressibility of the fluid, and the (inhomogeneous) no-slip condition at the interfaces, it is not possible to impose (homogeneous) no-slip boundary conditions on all of $\partial \Omega$. This is the reason why we split the boundary $\partial \Omega$ into two parts $\Sigma_1$ and $\Sigma_2$ and require a no-slip boundary condition on $\Sigma_2$ and a do-nothing condition on $\Sigma_1$, which allows for fluid in- and outflow through $\Sigma_1$.  

\medskip

Let $c_\e$ denote the bulk  concentration inside the fluid domain.
We describe its dynamics via the following advection--diffusion--reaction equation with interfacial adsorption kinetics on the evolving inclusions:
\begin{subequations}\label{system:diffusion_eps}
\begin{empheq}[left=\empheqlbrace]{alignat=2}
    \partial_tc_\e-\dive(\nabla c_\e-u_\e c_\e)&=F(c_\e)&\quad&\text{in}\ \Omega_\e(t),\ t\in S,\\
    -(\nabla c_\e-u_\e c_\e)\cdot n&=0&\quad&\text{on}\ \Sigma_2,\ t\in S,\\
    c_\e&=\cext&\quad&\text{on}\ \Sigma_1,\ t\in S,\\
    -(\nabla c_\e-u_\e c_\e)\cdot\nu_{\e,i}&=\e^\alpha\partial_tr_{\e,i} c_\e +\e^{3-2\alpha} G(c_\e,r_{\e,i})&\quad&\text{on}\ \Gamma_{\e,i}(t),\ t\in S,\ i\in\mathcal{I}_\e\\
    c_\e(0)&=c_{\e,0}&\quad&\text{on}\ \Omega_\e.
\end{empheq}
\end{subequations}
Here, $F$ and $G$ are given functions.
The former is a reaction term and the term $G(c_\e,r_{\e,i})$ describes the mass flux due to adsorption/desorption of substance at the balls where $G<0$ corresponds to adsorption (loss of bulk substance) and $G>0$ to desorption (gain of bulk substance).
For simplicity, we have set the diffusivity to 1.
The outer boundary conditions describe that  there is no mass flux through $\Sigma_2$ and that the concentration vanishes at the interface $\Sigma_1$ through which fluid can enter and exit.
We introduce the overall surface mass $M_\e(t)=\int_{\Omega_\e(t)}c_\e(t,x)\di{x}$.
Using Reynold's transport theorem and the model equations, we (formally) obtain the mass balance
\[
\ddt M_\e(t)=\int_{\Omega_\e(t)}F(c_\e)\di{x}+\int_{\Sigma_1}\nabla c_\e\cdot n+\e^{3-2\alpha}\sum_{\mathcal{I}_\e}\int_{\Gamma_{\e,i}(t)}G(c_\e(t,\gamma),r_{\e,i}(t))\di{\gamma}.
\]
Note that the scaling $e^{3-2\alpha}$ is chosen in such a way that the last term is of order one. 
If we assume that the total mass is conserved in the adsorption/desorption process, the stored surface mass for any individual ball $M_{\e,i}^\Gamma(t)$ should change according to
\begin{equation}\label{eq:mass_depos}
\ddt M_{\e,i}^\Gamma(t)=-\e^{3-2\alpha}\int_{\Gamma_{\e,i}(t)}G(c_\e(t,\gamma),r_{\e,i}(t))\di{\gamma}.
\end{equation}
Taken together, we have
\[
\ddt\left(M_\e(t)+\sum_{\mathcal{I_\e}} M_{\e,i}^\Gamma(t)\right)
=\int_{\Omega_\e(t)}F(c_\e)\di{x}+\int_{\Sigma_1}\nabla c_\e\cdot n.
\]%
Moreover, if the deposited material fills a shell with constant density $\rho_\e$, the overall stored mass is proportional to the shell volume:
\[
M_{\e,i}^\Gamma(t)=\frac{4\pi}3 \rho_\e\e^{3\alpha} \left(r_{\e,i}(t)^3-r_{c}^3\right) 
\]
where $r_c\geq0$ is the solid core of the ball ($r_c=0$ corresponds to the case where the whole ball is made up by the chemical substance).
We assume the scaling $\rho_\e=\e^{3-3\alpha}$ in order to obtain that the rate of  change of $r_{\e,i}$ is of order one.
In this case, \cref{eq:mass_depos} can be written as an ODE for the radial evolution:
\begin{equation}
 \ddt r_{\e,i}(t)=-\frac{1}{|\Gamma_{\e,i}(t)|}\int_{\Gamma_{\e,i}(t)}G(c_\e(t,\gamma),r_{\e,i}(t))\di{\gamma}.
\end{equation}
We get the radius evolution problem:
\begin{subequations}\label{system:radius_eps}
\begin{empheq}[left=\empheqlbrace]{alignat=2}
    \partial_tr_{\e,i}&=-\fint_{\Gamma_{\e,i}(t)}G\left(c_\e,r_{\e,i}\right)\di\gamma&\quad&\text{for}\ \ t\in S,\ i\in\mathcal{I}_\e,\label{evol.surface.concentration}\\
    r_{\e,i}(0)&=r_{\e,i,0},&\quad&\text{for}\ i\in\mathcal{I}_\e.
    \end{empheq}
\end{subequations}
A particularly  simple model for the exchange function $G\colon\R\times\R\to\R$ is given via
\[
G(c,r)=\tilde g_{1}(r)-cg(r)
\]
Here, the desorption rate $\tilde g_1(r)$ only depends on the stored mass (modeling \textit{spontaneous detachment} of substance) and the adsorption rate $cg(r)$ is modeled as a chemical reaction (with a simple mass-action principle).
This can also be written as an exchange law via
\[
G(c,r)=g(r)\left(\cequ(r)-c\right)
\]
where $\cequ(r)=\frac{\tilde g(r)}{g(r)}$ is the equilibrium bulk concentration corresponding to radius $r$.
This model corresponds to the Langmuir adsorption kinetics, see \cite{augner_analysis_2024,noorden_crystal_2009}.
For simplicity of the presentation, we will restrict our model to this exchange law in the following.

\subsection{Assumptions and notation}\label{ssec:assumpstions}
We start with a few short comments regarding our notation: In the following, we use $C>0$ to denote generic constants, whose specific value might change even from line to line, and which (if not explicitly indicated otherwise) are independent of the parameter $\e$ and time $t\in S$.
For two nonnegative quantities \(A,B\), we write $A\lesssim B$ if there exists such a constant \(C>0\) with \(A\le CB\). Likewise,
\[
A\lesssim_{\lambda} B
\]
means that \(A\le C B\) with a constant \(C=C(\lambda)>0\) which may not decrease with \(\lambda\) (but still is independent of \(\e\) and \(t\)).
When we write $A\lesssim_t t B$ as for example in \cref{lemma:apriori_prescribed_radius}, we emphasize that $A\le C(t)B$ with $\lim_{t\to0} C(t)=0$.
In general, we only indicate the constants which are necessary for the analysis; in particular, we assume all our data in the assumptions below to be fixed and we do not track their influence.
For the index set $\mathcal{I}_\e=\{1,...,N_\e\}$, we introduce the finite Banach spaces $\ell^p(\mathcal{I}_\e) := \ell^p(\R^{N_\e}) $ and note that since $N_\e=\e^{-3}$
\[
\|r\|_{\ell^\infty(\mathcal{I}_\e)}\leq \|r\|_{\ell^p(\mathcal{I}_\e)}\lesssim \e^{-\frac3p}\|r\|_{\ell^\infty(\mathcal{I}_\e)}
\]
for all $p\in[1,\infty)$ and $r\in \ell^p(\mathcal{I}_\e)$.
To unburden the notation slightly, we drop the $\e$-index whenever we chose a  fixed function; for example, when we consider a prescribed radial evolution, we write $r$ instead of $r_\e$ even though it technically depends on $\e$ via $\indexI$ and the initial value.


We make the following assumptions on our geometry and our data:
\begin{enumerate}[label=\textbf{(A\arabic*)}, ref=(A\arabic*)]
    \item\label[assumption]{A1} Assumptions on the microstructures:
        \begin{itemize}
            \item The center points of the balls are given by $z_{\e,i}\in\Omega$ for $i\in\mathcal{I}_\e=\{1,...,N_\e\}$ where $N_\e=\e^{-3}$.
            Both the minimal distance between two center points and the distance of the balls to the outer boundary is of order $\e$, i.e., there is $\dmin>0$ such that
            \begin{align} \label{ass:strong.separation}
            |z_{\e,i}-z_{\e,j}|\geq \dmin\e,\quad \dist(z_{\e,i},\partial\Omega)\geq \dmin\e\quad(i,j\in\mathcal{I}_\e,\ i\neq j).
            \end{align}
            \item The initial radii $r_{\e,i,0}$ are uniformly bounded from below and above via
            \[
            0<\frac{1}{\overline r}\le\inf_{\e>0}\min_{i\in\mathcal{I}_\e}r_{\e,i,0}\le\sup_{\e>0}\max_{i\in\mathcal{I}_\e}r_{\e,i,0}\le\overline{r}<\infty.
            \]
            for some $\overline r>1$.
            \item If $\alpha < 3$, we additionally assume that there exists $\dmax>0$ such that
            \begin{align}\label{ass:uniform.covering}
                \Omega \subset \bigcup_{i \in \mathcal I_\e} B_{\dmax \e}(z_{\e,i}).
            \end{align}
        \end{itemize}
    \item\label[assumption]{A2} Regularity and structure of the right hand sides and coefficients:
    \begin{itemize}
        \item Let $F\in C^{0,1}_{\mathrm{loc}}(\R)$ with $F=0$ on $(-\infty,0]$ and such that there is a constant $C_F>0$ with 
        \[
        C_F(1+|b|)\quad (b\in\R).
        \]
        \item There are functions $g,c^{\mathrm{eq}}\in C_{\mathrm{loc}}^{0,1}(\R;\R_{\geq0})$ such that $G(b,\sigma)=g(\sigma)\left(c^{\mathrm{eq}}(\sigma)-b\right)$.
        In addition, $c^{\mathrm{eq}}$ is nondecreasing, $\cequ(0)=0$, and
        \begin{equation}\label{eq:non_blowup}
        \int_r^{1}\frac{1}{g(z)\cequ(z)}\di{z}\stackrel{r\to 0+}{\to}\infty,\quad \int_{1}^r\frac{1}{g(z)}\di{z}\stackrel{r\to\infty}{\to}\infty.
        \end{equation}
        \item $h_\e\in L^2(S;H^1(\Omega)^3)\cap L^\infty(S;L^2(\Omega)^3)$ with 
        \[
            \sup_{\e>0}
            \left(
            \e^\alpha\|\nabla h_\e\|_{L^\infty(S;L^2( \Omega))}
            +
            \|h_\e\|_{L^\infty(S;L^2(\Omega))}
            \right)<\infty.
        \]
    \end{itemize}
    \item\label[assumption]{A3} Initial condition:
    $c_{\e,0}\in L^\infty(\Omega_\e)$ with $0\le c_{\e,0}\le \overline{c}_0$ for some $\overline{c}_0\ge0$ almost everywhere in $\Omega_\e$.

\end{enumerate}
\begin{remark}
\begin{enumerate}
    \item Assumption \ref{A1} ensures that the balls are initially well-separated. In particular, we have (for sufficiently small $\e$ and $B_{\e,i}:=B_{\e^\alpha r_{\e,i,0}}(z_{\e,i})$)
    \[
    \dist(B_{\e,i},B_{\e,j})
    =|z_{\e,i}-z_{\e,j}|-\e^\alpha\left(r_{\e,i,0}+r_{\e,j,0}\right)
    \geq |z_{\e,i}-z_{\e,j}|-2\e^{\alpha}\overline{r}
    \geq \frac23|z_{\e,i}-z_{\e,j}|
    \]
    since $\alpha-1>0$ and $|z_{\e,i}-z_{\e,j}|\geq \dmin\e$.
    For $\alpha<3$, the uniform covering assumption \cref{ass:uniform.covering} is made to to control $u_\e$ through the scaling in the Poincaré inequality for $\Omega_\e$ (see \cref{lemma:poincare}).
    \item 
    The blow-up conditions \eqref{eq:non_blowup} are necessary for long time existence (in particular they prevent the balls to collapse to a point in finite time).
    This prohibits the physically reasonable choice $g(r)\cequ(r)=const.$ which models a constant desorption rate.
    In classical precipitation--dissolution models, such a constant rate is typically switched off once the deposited  layer is exhausted, for instance by means of a Heaviside-type law $g(r) \cequ(r) = a \1_{r> r_c} $ where $r_c\ge0$ denotes the radius of a non-dissolvable core, see e.g.~\cite{vanDuijnPop2004,KumarNeussRaduPop2016}.
    A smooth and therefore admissible realization of this mechanism is the saturation law given by $g(r)\cequ(r)=a\frac{(r-r_c)_+}{(r-r_c)_++\gamma}$ for some $a,\gamma>0$.
    Other admissible choices include $g(r)\cequ(r)=a (r-r_c)_+^\gamma$ for some $a\ge0$ and $\gamma\ge1$.
    \item We emphasize that the dissipative structure of the interfacial exchange, i.e., 
    \[
        G(b,\sigma)(\cequ(\sigma)-b)\ge0,
    \]
    is important for obtaining estimates that are uniform in $\e$ in the critical case $\alpha=3$.
    While Assumption~\textup{(A2)} is not sharp and can be relaxed, general non-dissipative linear interface production may lead to growth rates that become unbounded as \(\e\to0\); see the discussion at the beginning of \cref{ssec:maximum}.
    \end{enumerate}
\end{remark}

For $S=(0,T)$, we introduce the space of admissible radial evolutions
\begin{align} \label{mathcalR}
\mathcal R_\e(S):=\{r\in W^{1,\infty}(S)^\indexI\ : \ r(0)=r_{\e,0},\ \inf_{t\in S}\min_{i_\in\indexI}r_i(t)>0\}.
\end{align}
The minimal distance of any two different balls can be estimated by
\[
\dist(B_{\e,i}(t),B_{\e,j}(t))\le|z_{\e,i}-z_{\e,j}|-\e^\alpha\left(r_{\e,i}(t)+r_{\e,j}(t)\right)\ge\e\left(\dmin-2\e^{\alpha-1}\|r\|_{L^\infty(S;\ell^\infty(\indexI))}\right)\quad (i\neq j).
\]
This ensures $\dist(B_{\e,i}(t),B_{\e,j}(t))\geq \e\frac{\dmin}2$ as long as 
\begin{equation}\label{eq:eps_condition}
\e^{\alpha-1}\le \frac{\dmin}{4\|r\|_{L^\infty(S;\ell^\infty(\indexI))}}.
\end{equation}
In this case, 
the perforations defined in \eqref{eq:changing_domains} do not collapse to points and remain uniformly separated from each other and from $\partial\Omega$.
Then, the family $(\Omega_\e(t))_{t\in S}$ is related, for all $t\in\overline S$, by smooth diffeomorphisms
\[
\Psi_\e(t)\colon \overline{\Omega_\e}\to\overline{\Omega_\e(t)},
\]
where $\Omega_\e=\Omega_\e(0)$
(see \cref{ssec:reference_config,appen:trafo} for details on these transformations).
For a family of Banach spaces $X(t)$ of functions defined on $\Omega_\e(t)$
(e.g.,
$X(t)=L^q(\Omega_\e(t))$
or
$X(t)=H^1(\Omega_\e(t))$),
we use the shorthand
\[
L^p(S;X(t))
\]
for the set of functions $u$ such that the pullback
\[
\hat u(t,\cdot)
:=
u(t,\Psi_\e(t,\cdot))
\]
belongs to the classical Bochner space
\(
L^p(S;X(0)).
\)
The corresponding dual spaces (e.g., $L^2(S;H^1(\Omega_\e(t))^*)$) are understood via the induced dual pullback operators (we point to \cite{alphonse_abstract_2015} for details).
We also introduce the Hilbert space
\[
W(S;\Omega_\e(t))=\{u\in L^2(S;H^1(\Omega_\e(t)))\ : \ \partial_tu\in L^2(S;H^1(\Omega_\e(t))^*)\}
\]
and, for a domain $A$ and $\Sigma\subset\partial A$ with positive surface measure, we write
\[
H^1_\Sigma(A)=\{u\in H^1(A)\ : \ u=0 \ \text{on $\Sigma$}\}.
\]

\begin{definition}[Weak form (moving geometry)]
    \label{definition:weak_solution_moving}
    Let $T>0$ and $S=(0,T)$.
    We say that 
    \[
    (c_\e,u_\e,p_\e,r_\e)\in W(S;\Omega_\e(t))\times  L^\infty(S;H^1(\Omega_\e(t)))^3\times \LtLx{\infty}{2}{\Omega_\e(t)}\times W^{1,\infty}(S)^{\indexI}
    \]
    is a local-in-time, weak solution to the moving boundary problem (given by Systems~\eqref{system:stokes_eps} and~\eqref{system:diffusion_eps} as well as ODE~\eqref{system:radius_eps}) over $S$ if the following conditions are satisfied:
    \begin{itemize}
    \item[$(i)$] The radius function satisfies $r_\e(0)=r_{\e,0}$, $r_{\e,i}(t)>0$ for all $t\in \overline S$, and the ODEs
    \begin{subequations}\label{eq:micro_system}
    \begin{equation}\label{eq:radius_evolution_ODE}
    \partial_t r_{\e,i}=-\fint_{\Gamma_{\e,i}(t)}G(c_\e,r_{\e,i})\di\gamma \quad (i\in\mathcal{I}_\e).
    \end{equation}
    The radius function characterizes the changing geometry via \cref{eq:changing_domains} such that the balls remain separated.
    \item[$(ii)$] For almost all $t\in S$, it is $u_\e=\e^\alpha\partial_tr_\e\nu_\e$ on $\Gamma_\e(t)$, $u_\e=0$ on $\Sigma_2$, and
    \begin{align}
        2\e^{3-\alpha}(e (u_\e),e(\psi))_{L^2(\Omega_\e(t))}-(p_\e,\dive\psi)_{L^2(\Omega_\e(t))}
        &=(h_\e,\psi)_{L^2(\Omega_{\e}(t))},\label{eq:micro_system:b}\\
        (\dive u_\e,\phi)_{L^2(\Omega_\e(t))}&=0\label{eq:micro_system:c}
    \end{align}
    holds for all $(\psi,\phi)\in H^1_{\Sigma_2\cup\Gamma_\e}(\Omega_\e(t))^3\times L^2(\Omega_\e(t))$.
    \item[$(iii)$] It is $c_\e=0$ on $\Sigma_1$ and, for all $\varphi\in C^\infty([0,T]\times\overline\Omega)$ with $\varphi=0$ on $\Sigma_1$, it holds
    \begin{multline}\label{eq:micro_system_moving}
    (c_{\e}(T),\varphi(T))_{L^2(\Omega_{\e}(T))}
    +\int_S-(c_\e,\partial_t\varphi)_{L^2(\Omega_{\e}(t))}
    +(\nabla c_\e-c_\e u_\e,\nabla\varphi)_{L^2(\Omega_{\e}(t))}\di{t}\\
    -\e^{3-2\alpha}\int_S(G(c_\e,r_{\e}),\varphi)_{L^2(\Gamma_{\e}(t))}\di{t}
    =\int_{S}(F(c_\e),\varphi)_{L^2(\Omega_{\e}(t))}\di{t}+(c_{\e,0},\varphi(0))_{L^2(\Omega_{\e})}.
    \end{multline}
    \end{subequations}
    \end{itemize}
\end{definition}

\begin{remark}
    \begin{itemize}
        \item 
        We can naturally extend \cref{eq:micro_system_moving} by density to functions $\varphi\in W(S;\Omega_\e(t))$ with $\varphi=0$ on $\Sigma_1$ via the form 
        \begin{multline}\label{eq:micro_system_moving_weak}
        (c_{\e}(T),\varphi(T))_{L^2(\Omega_{\e}(T))}+\int_S-\langle\partial_t\varphi,c_\e\rangle_{H^1(\Omega_{\e}(t))^*}+(\nabla c_\e-c_\e u_\e,\nabla\varphi)_{L^2(\Omega_{\e}(t))}\di{t}\\
        -\e^{3-2\alpha}\int_S(G(c_\e,r_{\e}),\varphi)_{L^2(\Gamma_{\e}(t))}\di{t}
        =\int_{S}(F(c_\e),\varphi)_{L^2(\Omega_{\e}(t))}\di{t}+(c_{\e,0},\varphi(0))_{L^2(\Omega_{\e})}.
        \end{multline}
        This allows us to chose $\varphi=c_\e$.
        \item We can also combine the ODE for the radial evolution (\cref{eq:radius_evolution_ODE}) and the diffusion equation (\cref{eq:micro_system_moving}) into the coupled problem formulation
        \begin{multline}\label{eq:weak_form_alternative_derivative}
        (c_{\e}(T),\varphi(T))_{L^2(\Omega_{\e}(T))}+
        \int_S-\langle\partial_t\varphi,c_\e\rangle_{H^1(\Omega_{\e}(t))^*}+(\nabla c_\e-c_\e u_\e,\nabla\varphi)_{L^2(\Omega_{\e}(t))}\di{t}\\
        \frac{4\pi\e^3}{3}(r_{\e}^3(T),\eta(T))_{\ell^2(\indexI)}-\frac{4\pi\e^3}{3}\int_S(r_{\e}^3,\partial_t\eta)_{\ell^2(\indexI)}\di{t}
        +\e^{3-2\alpha}\int_S(G(c_\e,r_{\e}),\eta-\varphi)_{L^2(\Gamma_{\e}(t))}\di{t}\\
        =\int_{S}(F(c_\e),\varphi)_{L^2(\Omega_{\e}(t))}\di{t}
        +(c_{\e,0},\varphi(0))_{L^2(\Omega_{\e})}
        +\frac{4\pi\e^3}{3}
        (r_{\e,0}^3,\eta(0))_{\ell^2(\indexI)}
        \end{multline}
        for all test functions $(\varphi,\eta)\in W(S;\Omega_\e(t))\times W^{1,\infty}(S)^\indexI$ with $\varphi=0$ on $\Sigma_1$.
        Here, we identify a vector $\eta\in \ell^p(\mathcal{I}_\e)$ with the locally constant $\eta\in L^p(\Gamma_\e)$ via $\eta=\eta_i$ on $\Gamma_{\e,i}(t)$.
        This version is particularly convenient for estimates as it allows us to directly exploit the dissipative structure of $G$ via
        \[
        G(b,\sigma)(\cequ(\sigma)-b)\ge0
        \]
        for all $b,\sigma\ge0$ (see Assumption (A2)).
        In \cref{ssec:maximum}, we will mainly work with \cref{eq:weak_form_alternative_derivative}. 
    \end{itemize}
\end{remark}



\subsection{Well-posedness of the microscopic problem}
\label{sssec:well_posedness}

\begin{theorem}[Well-posedness of the microscopic problem]\label{theorem:well_posedness}
Let Assumptions~\textup{(A1)--(A3)} hold.
For every \(T>0\), there is \(\varepsilon_T>0\) such that, for every
\(0<\varepsilon\leq\varepsilon_T\), the microscopic problem
\eqref{system:stokes_eps}--\eqref{system:radius_eps} admits a unique weak solution
\[
    (c_\varepsilon,u_\varepsilon,p_\varepsilon,r_\varepsilon)
\]
on \([0,T]\) in the sense of \cref{definition:weak_solution_moving}.

In addition, it holds $0\le c_\e(t,x)\le\max\{\overline c_0,\cequ(\overline r)\}e^{2C_Ft}$ and there are $\e$-independent functions $\underline{\xi},\overline\xi\colon[0,\infty)\to(0,\infty)$ such that
\[
0<\underline{\xi}(t)\le r_{\e,i}(t)\le\overline\xi(t)<\infty\quad (t\in S,\ i\in\indexI).
\]
\end{theorem}
The a priori estimates are shown in \cref{ssec:maximum} and the existence in \cref{sec:analysis}.
\begin{remark}
    \begin{enumerate}
        \item While this result is local-in-time for every fixed $\e$ (because the separation condition for the moving particles may be violated after finite time), it is asymptotically global in time in that we reach any arbitrary large time horizon for all sufficiently small $\e$.
        In addition, no collapse of particles in finite time is possible because $\underline\xi$ is positive for all times.
        \item The theorem only asserts non-negativity of \(c_\e\).
        Because of the homogeneous Dirichlet condition on \(\Sigma_1\), no strictly positive lower bound up to the outer boundary can hold.
        If the homogeneous boundary condition is replaced by
        \[
            c_\e=c_{\rm ext}>0
            \qquad\text{on }\Sigma_1,
        \]
        strict positivity in the interior for positive times follows via the corresponding assumptions by the parabolic strong maximum principle.
    \end{enumerate}
\end{remark}

\subsection{Convergence result in the critical case \texorpdfstring{$\alpha = 3$}{alpha equals 3}.}
\label{sssec:critical_case}

In order to describe the asymptotic behavior of the microscopic solutions as \(\e\to0\),  we denote the set of probability measures on $\Omega\times[0,\infty)$ by $\mathcal P\bigl(\Omega\times[0,\infty)\bigr)$ and introduce the empirical measure, which encodes the spatial distribution of the particles and their radii,
\begin{equation}\label{eq:f_vareps}
    f_\varepsilon(t)
    :=
    \frac{1}{N_\varepsilon}
    \sum_{i\in\mathcal I_\varepsilon}
    \delta_{z_{\varepsilon,i}}
    \otimes
    \delta_{r_{\varepsilon,i}(t)}
    \in
    \mathcal P\bigl(\Omega\times[0,\infty)\bigr),
\end{equation}
and the spatially regularized empirical measure,
\begin{equation}\label{f^app}
    f_\varepsilon^{\rm app}(t)
    :=
    \frac{1}{N_\varepsilon}
    \sum_{i\in\mathcal I_\varepsilon}
    \delta_{B_{\frac{\dmin \varepsilon}{4}}(z_{\varepsilon,i})}
    \otimes
    \delta_{r_{\varepsilon,i}(t)},
\end{equation}
where
\(
    \delta_{B_{\frac{\dmin\varepsilon}{4}}(z_{\varepsilon,i})}
\)
denotes the uniform probability measure on
\(B_{\frac{\dmin\varepsilon}{4}}(z_{\varepsilon,i})\).
We measure the convergence of the empirical measures with respect to the
\(2\)-Wasserstein distance \(\mathcal W_2\) on
\(\mathcal P(\Omega\times[0,\infty))\).
We recall that the $2$-Wasserstein distance minimizes the quadratic cost among all transport plans $\pi \in \Pi(f_\e,f)$, where $\Pi(f_\e,f)$ denotes the set of probability measures $\pi \in \mathcal{P}((\Omega \times (0,\infty))^2)$ having $f_\e$ (resp. $f$) as a first (resp. second) marginal, which means that for all $\varphi \in C_c^\infty(\Omega \times [0,\infty))$
\begin{align*}
    \int_{(\Omega \times (0,\infty))^2} \varphi(x_1,\sigma_1) \pi (\d x_1, \d \sigma_1, \d x_2, \d \sigma_2) =  \int_{\Omega \times (0,\infty)} \varphi(x,\sigma) f_\e (\d x, \d \sigma), \\
     \int_{(\Omega \times (0,\infty))^2} \varphi(x_2,\sigma_2) \pi (\d x_1, \d \sigma_1, \d x_2, \d \sigma_2) =  \int_{\Omega \times (0,\infty)} \varphi(x,\sigma) f (\d x, \d \sigma).
\end{align*}
Then, 
\begin{align} \label{W_2}
\mathcal{W}_2(f_\e,f)^2:=\underset{\pi \in \Pi(f_\e,f)}{\inf}\left\{ \int\left(|x_1-x_2|^2+|\sigma_1-\sigma_2|^2\right)   \pi(\d x_1, \d \sigma_1, \d x_2, \d \sigma_2) \right\}.
\end{align}
For the spatially regularized measures, we additionally use the dual norm associated with
\[
    H_x^1W_\sigma^{1,\infty}
    :=
    H^1\bigl(\mathbb R^3;W^{1,\infty}([0,\infty))\bigr).
\]

\medskip

We introduce the function 
\begin{align}\label{H}
 H\colon\R_{\ge0}^2\to\R \quad H(b,\sigma)=\frac{\sigma g(\sigma)}{1+\sigma g(\sigma)}\left(\cequ(\sigma)-b\right),
\end{align}
which describes the correction of the macroscopic concentration at the particle surface.
We note that $H$ satisfies the algebraic equation
\begin{align*}
    H(b,\sigma) =\sigma G(b+ H(b,\sigma),\sigma).
\end{align*}
When $G$ is nonlinear in $c$, the effective reaction term is in general characterized only implicitly through this equation, we refer to \cite{Jager_neuss_shaposhnikova2011}, where such an implicit nonlinear boundary-layer correction appears in the homogenized reaction term.

The macroscopic limit is described by a coupled system for the bulk
concentration \(c\), the  particle-size distribution \(f\), the fluid velocity \(u\), and the pressure \(p\).
We consider the limit concentration that solves
\begin{subequations}\label{limit.chemisty}
\begin{empheq}[left=\empheqlbrace]{alignat=2}
   \partial_tc-\dive(\nabla c-uc)&=F(c)+\int_0^\infty 4 \pi \sigma H(c,\sigma) f(\cdot,\d\sigma)&\quad&\text{in}\ S \times \Omega,\\
    -(\nabla c-u c)\cdot n&=0&\quad&\text{on}\ S \times \Sigma_2,\\
    c&=\cext&\quad&\text{on}\ S \times \Sigma_1,\\
    c(0)&=c_{0}&\quad&\text{in}\ \Omega.
\end{empheq}
\end{subequations}
The evolution of the particle-size distribution is governed by the continuity equation
\begin{subequations}\label{evolution.f}
\begin{empheq}[left=\empheqlbrace]{alignat=2} 
    \partial_t f - \partial_\sigma\left( \frac{H(c,\sigma)}{\sigma} f \right) &= 0 & \quad \text{in} \ S \times \Omega \times [0,\infty), \\
    f(0) &= f_0 & \text{in} \ \Omega \times [0,\infty),
\end{empheq}
\end{subequations}
Note that $x\in\Omega$ is just a parameter entering via the concentration $c=c(t,x)$ and the initial value $f_0$.
Finally, the fluid velocity and pressure solve the Brinkman system
\begin{subequations}\label{Brinkman}
\begin{empheq}[left=\empheqlbrace]{alignat=2}
    - \Delta u + \nabla p + u \int_0^\infty 6 \pi \sigma f(\cdot,\d\sigma) &= h&\quad&\text{in}\ S \times \Omega\\
    \dive u&=0&\quad&\text{in}\ S \times \Omega,\\
\sigma[u,p]n&=0&\quad&\text{on}\  S \times \Sigma_1,\\
    u&=0&\quad&\text{on}\  S\times \Sigma_2.
\end{empheq}
\end{subequations}

Here, the diffusion-reaction system \eqref{limit.chemisty} and the Brinkmann problem \eqref{Brinkman} are understood in the usual weak sense.
The continuity equation \eqref{evolution.f} is understood in the distributional sense.
We require \(f\) to be absolutely continuous with respect to the spatial variable, i.e.,
\[
    f(t,\d x,\d\sigma)=f(t,x,\d\sigma)\di x,
\]
where \(f(t,x,\cdot)\in\mathcal M_+([0,\infty))\), with uniformly bounded mass and uniformly compact support in \(\sigma\).
This allows \(f\) to be atomic in the radius variable, while also guaranteeing that the expressions
\[
    \int_0^\infty 4\pi\sigma H(c,\sigma)\,f(\cdot,\cdot,\d\sigma)
    \quad\text{and}\quad
    \int_0^\infty 6\pi\sigma\,f(\cdot,\cdot,\d\sigma)
\]
are well defined.
Moreover, we denote by $C_w(S;\mathcal P(\Omega \times [0,\infty))$ the space of weakly-$\ast$ continuous maps, i.e.  $S \ni t \mapsto f(t,\cdot) \in \mathcal P(\Omega \times [0,\infty))$ is an element of $
C_w(S;\mathcal P(\Omega \times [0,\infty))$ if for all $t \in S$
$$
\lim_{s\to t} \int_{\Omega \times [0,\infty)} \varphi \dd f(s,\cdot) = \int_{\Omega \times [0,\infty)} \varphi \dd f(t,\cdot)\qquad \text{for all } \varphi\in C_b(\Omega\times[0,\infty)).
$$

\begin{theorem}\label{th:hom.critical}
Let   $T>0$, $h \in L^\infty(S \times \Omega)$, $c_0 \in W^{1,\infty}(\Omega)$ and $f_0 \in \mathcal P(\Omega \times [0,\infty))$ such that $\supp f_0$ is compact in $\overline \Omega \times [0,\infty)$
\[
(c,u,p,f) \in W^{1,\infty}(S\times\Omega) \times L^{\infty}(S;W^{2,\infty}(\Omega)^3) \times L^{\infty}(S;W^{1,\infty}(\Omega)) \times C_w(S;\mathcal P(\Omega \times [0,\infty)))
\]
be a solution to \eqref{limit.chemisty}--\eqref{Brinkman} such that $\int_0^\infty f(\cdot,\cdot,\d \sigma)  \in L^\infty(S\times \Omega)$.
Then, there exists $\e_0$ such that the unique weak solution to \eqref{system:stokes_eps}--\eqref{system:radius_eps}  satisfies, for $\e\le\e_0$,
\begin{multline} \label{quantitative.est}
    \|u_\e - u\|_{\LtLx{2}{2}{\Omega_\e(t)}} + \|c_\e - c\|_{\LtLx{\infty}{2}{\Omega_\e(t)}}+ \sup_{S} \mathcal W_2(f_\e,f) \\
    \lesssim \e + \|f_\e^{app}(0) - f_0\|_{(H^1_xW_\sigma^{1,\infty})^\ast}+  \|h_\e - h\|_{L^2(S;(H^{1}(\Omega))^*)} + \|c_{\e,0} - c_0\|_{L^2(\Omega_\e)} + \mathcal W_2(f_\e(0),f_0).
\end{multline}
The constants \(\e_0>0\) and the implicit constant in
\eqref{quantitative.est} depend only on $T$, the constants in Assumptions~\textup{(A1)--(A3)}, and the indicated bounds on \((c,u,p,f)\).
\end{theorem}
\begin{remark}\label{remark:critical_theorem}
    \begin{enumerate}
        \item Such regularity of $(c,u,p)$ can be obtained for sufficiently smooth data and geometries for which the mixed boundary-value problems do not generate singularities at the interface between $\Sigma_1$ and $\Sigma_2$.
        This may be the case, for instance, when \(\Sigma_1\) and \(\Sigma_2\) belong to different connected components of \(\partial\Omega\), or, under suitable compatibility conditions, in particular configurations in which they meet orthogonally as in Figure \ref{geometry}; see, e.g., \cite{BenesKucera2016,Grisvard1985,MazyaRossmann2007}. 
        Alternatively, we could assume $\Sigma_2 = \emptyset$ to avoid any difficulties due to the mixed type of the boundary conditions. 
        In this case, there is no longer a Poincar\'e inequality on $H^1_{\Sigma_2}(\Omega)$ but one might assume \eqref{ass:uniform.covering} also in the critical case $\alpha = 3$ in order to apply the Poincar\'e inequality in $H^1_{\Gamma_\e}(\Omega_\e)$ instead, provided by  Lemma \ref{lemma:poincare}.

        The compact-support assumption on $f_0$ is equivalent to assuming that the solution $f$ is compactly supported in the radius variable, i.e., that the maximal radius in the support of $f_0$ is bounded.
        Indeed, the continuity equation for $f$ propagates such a uniform bound thanks to the boundedness of $c$.
        
        \item The regularity assumptions on the limit problem can probably be weakened.
        They are chosen to avoid further technicalities in the proof.  As usual in such homogenization results, a certain amount of regularity assumptions is probably unavoidable to obtain quantitative results.
        \item The constant $\e$ in \eqref{quantitative.est}  can be omitted, because $\mathcal W_2(f_\e(0),f_0) \gtrsim \e$. This estimate holds because the spatial marginal $ \int_0^\infty f_\e(0,\cdot,\d\sigma)$ is discrete (the empirical measure of $\e^{-3}$ points)  but the spatial marginal $\int_0^\infty f_0(\cdot,\d\sigma)$ is continuous. See e.g. \cite[Eq. (1.14)]{HoferSchubert25}.

       We emphasize that the error $\e$ is sharp for the homogenization of the Stokes equation in a fixed perforated domain, see \cite{allaire_homogenization_1990, JingLuPrange} (for the periodic setting).
        \item We get the following additional corrector estimates, see \eqref{corr.est.c} and \eqref{corr.est.u}
        \begin{multline*}
            \|\nabla (c_{\e} - c_{\e}^{app})\|_{L^2(S;L^2(\Omega_\e(t)))} + \|\nabla (u_\e - w_\e u)\|_{\LtLx{2}{2}{\Omega_\e(t)}} \\\leq C (\e + \|f_\e^{app}(0) - f_0\|_{(H^1_xW_\sigma^{1,\infty})^\ast}+\mathcal W_2(f_\e(0),f_0)\\
            +  \|h_\e - h\|_{L^2(S;(H^{1}(\Omega))^*)} + \|c_{\e,0} - c_0\|_{L^2(\Omega_\e)}).
        \end{multline*}
        where $c_\e^{app}$ and $w_\e$ are defined in \eqref{c^app} and \eqref{w^eps}, respectively. In particular, one can show that strong convergence $u_\e \to u$ and $c_\e \to c$ in $L^2(S;H^{1}(\Omega))$ and $L^2(S ; H^1(\Omega))$, respectively, cannot hold (for any extension of $c_\e$ and $u_\e$ to $\Omega$). This failure of strong convergence is typical for such homogenization problems and reflects the singular nature of the limit.
        \item The norm $ \|f_\e^{app}(0) - f_0\|_{(H^1_xW_\sigma^{1,\infty})^\ast}$ can be made quite small. 
       Indeed, since we do not require any regularity of $f$ with respect to the second variable, we can take $f_0$ to have \enquote{the same} radii as $f_\e^{app}(0)$ (in the sense that the optimal transport plan for $(f_\e^{app}(0),f_0)$ is supported on the diagonal with respect to the radii such that $\mathcal W_2(f_\e(0), f_0) =\mathcal W_2(\rho^{app}_\e , \rho) $). Then, $ \|f_\e^{app}(0) - f_0\|_{(H^1_xW_\sigma^{1,\infty})^\ast}$  essentially becomes $\|\rho^{app}_\e -  \rho\|_{H^{-1}(\R^3)}$, where $\rho^{app}_\e$,$\rho$ are the spatial marginals of $f_\e^{app}(0),f_0$. It is well known (see e.g. \cite[Lemma 5.33]{Santambrogio}), that 
       \begin{align*}\|(\rho^{app}_\e -  \rho)\|_{H^{-1}(\R^3)} &\leq \mathcal W_2(\rho^{app}_\e , \rho) \max\{\|\rho^{app}_\e\|_{L^\infty(\R^3)},  \|\rho\|_{L^\infty(\R^3)} \} \\&=   \mathcal W_2(f_\e(0), f_0) \max\{\|\rho^{app}_\e\|_{L^\infty(\R^3)},  \|\rho\|_{L^\infty(\R^3)} \}.
       \end{align*}
    \end{enumerate}
\end{remark}

\subsection{Convergence result in the supercritical case \texorpdfstring{$\alpha \in(1,3)$}{alpha less than 3}.}\label{sssec:supercritical_case}
In the supercritical case, the abovementioned local concentration correction at the particle surface is of order \(\e^{3-\alpha}\) and therefore vanishes as \(\e\to0\).
Consequently, the microscopic exchange law \(G\) enters the macroscopic equations directly, instead of its correction \(H\).

As in the critical case, the macroscopic limit is described by a coupled
system for the bulk concentration \(c\), the particle-size distribution
\(f\), the fluid velocity \(u\), and the pressure \(p\).
The concentration satisfies
\begin{subequations}\label{limit.chemisty.supercritical}
\begin{empheq}[left=\empheqlbrace]{alignat=2}
    \partial_t c-\dive(\nabla c-u c)
    &=
    F(c)
    +
    \int_0^\infty
        4\pi\sigma^2 G(c,\sigma)\,
        f(\cdot,\cdot,\d\sigma)
    &\quad&\text{in }S\times\Omega,
    \\
    -(\nabla c-u c)\cdot n
    &=0
    &&\text{on }S\times\Sigma_2,
    \\
    c
    &=\cext
    &&\text{on }S\times\Sigma_1,
    \\
    c(0)
    &=c_0
    &&\text{in }\Omega.
\end{empheq}
\end{subequations}
The evolution of the particle-size distribution is governed by
\begin{subequations}\label{evolution.f.supercritical}
\begin{empheq}[left=\empheqlbrace]{alignat=2}
    \partial_t f
    -
    \partial_\sigma\left(G(c,\sigma)f\right)
    &=0
    &\quad&\text{in }S\times\Omega\times(0,\infty),
    \\
    f(0)
    &=f_0
    &&\text{in }\Omega\times[0,\infty).
\end{empheq}
\end{subequations}
Finally, the fluid velocity and pressure solve the Darcy-type system
\begin{subequations}\label{Darcy}
\begin{empheq}[left=\empheqlbrace]{alignat=2}
    \nabla p
    +
    u
    \int_0^\infty
        6\pi\sigma\,
        f(\cdot,\cdot,\d\sigma)
    &=h
    &\quad&\text{in }S\times\Omega,
    \\
    \dive u
    &=0
    &&\text{in }S\times\Omega,
    \\
    p
    &=0
    &&\text{on }S\times\Sigma_1,
    \\
    u\cdot n
    &=0
    &&\text{on }S\times\Sigma_2.
\end{empheq}
\end{subequations}

\begin{theorem}[Homogenization result for $\alpha\in(1,3)$]
\label{thm:hom.supercritical}
Let $T>0$, $h \in L^\infty(S \times \Omega)$, $c_0 \in W^{1,\infty}(\Omega)$ and $f_0 \in \mathcal P(\Omega \times [0,\infty))$ such that $f_0 $ is compactly supported in $\overline \Omega \times (0,\infty)$ and $\int_0^\infty f_0(\cdot,\d \sigma) \in L^\infty(\Omega)$ is uniformly bounded from below.
Let 
\[
(c,u,p,f) \in W^{1,\infty}(S\times\Omega) \times L^{\infty}(S;W^{2,\infty}(\Omega)^3) \times L^{\infty}(S;W^{1,\infty}(\Omega)) \times C_w(S;\mathcal P(\Omega \times [0,\infty)))
\]
be a solution to \eqref{limit.chemisty.supercritical}--\eqref{Darcy} such that $\int_0^\infty f(\cdot,\cdot,\d \sigma)  \in L^\infty(S\times \Omega)$.
In addition, assume that there exists $\mu > 0$ and a $C^{3,\mu}$-domain $U\supset\Omega$  with $\Sigma_2\subset\partial U$ and such that $u$ admits an extension $\bar u\in L^\infty(S;C^{1,\mu}(\overline U)^3)$ satisfying $\bar u\cdot n=0$ on $S\times\partial U$.

Then, there exists $\e_0$ such that the unique weak solution to \eqref{system:stokes_eps}--\eqref{system:radius_eps} satisfies, for $\e\le\e_0$,
\begin{multline} \label{quantitative.est.supercritical}
    \|u_\e - u\|_{\LtLx{2}{2}{\Omega_\e(t)}} + \|c_\e - c\|_{\LtLx{\infty}{2}{\Omega_\e(t)}}+ \|\nabla (c_{\e} - c)\|_{L^2(S;L^2(\Omega_\e(t)))} + \sup_{S} \mathcal W_2(f_\e,f) \\
    \lesssim\e^{\frac{3-\alpha}{4}} +  \e^{\frac{\alpha - 1}{2}}    + \mathcal W_2(f_\e(0), f_0) +    \e^{\frac{\alpha-3}2}\|f_\e^{app}(0) - f_0\|_{(H^1_xW_\sigma^{1,\infty})^\ast}\\
    +  \|c_{\e,0} - c_0\|_{L^2(\Omega_\e)}  + \|h_\e - h\|_{L^2(S\times\Omega)}.
\end{multline}
The constants \(\e_0>0\) and the implicit constant in
\eqref{quantitative.est.supercritical} depend only on $T$, the constants in Assumptions~(A1)--(A3), and the indicated bounds on \((c,u,p,f)\).
\end{theorem}
\begin{remark}
    \begin{enumerate}
    \item 
    Note that in contrast to the critical regime, we assume through the assumption on the support that the distribution of radii according to $f_0$ is uniformly bounded from below. One can show that this is propagated by the equations similarly as for the microscopic dynamics. This assumption together with the uniform bound from below on $\int f(t,x,\d \sigma) =\int f_0(x,\d \sigma) $ guarantees that 
       $$\inf_{t \leq T} \inf_{x \in \Omega} \int_0^\infty 6 \pi \sigma f(t,x,\d \sigma) >0.$$
    Hence, equation \eqref{Darcy} is well-posed.

    These assumptions are actually not very restrictive since they are satisfied by any weak accumulation point of $f_\e(0)$. This holds firstly thanks to the uniform bound from above and  below on the microscopic radii, and secondly thanks to the assumption \eqref{ass:uniform.covering} which guarantees that for any cube $Q \subset \Omega$
    $$ \exists \e_Q > 0 \quad \forall \e< \e_Q \qquad f_\e(0,Q \times \R_+) \gtrsim |Q|.$$

    \item The additional regularity on the outer boundary is only needed for the construction of the boundary-layer corrector for the fluid flow to account for the loss of the no-slip condition on $\Sigma_2$ in the limit.
    Specifically, we require a correction near $\partial\Omega$ and the assumed extension property allows us to represent an extension of $u$ by a sufficiently regular vector potential, see~\cite[Lemma 4.10]{BalaziAllaireOmnes25}. Note that under the regularity assumptions of $u$, such $U$ and $\bar u$ exist in particular for domains as in Figure \ref{geometry}.
    
    \item The comments on the regularity assumptions in \cref{remark:critical_theorem} apply analogously here.
    The main difference is that the fluid variables are now governed by the Darcy system rather than the Brinkman system.
    Accordingly, the required regularity of $(u,p)$ has to be understood in terms of the corresponding mixed elliptic problem for the pressure, subject to the same caveats concerning the junction of $\Sigma_1$ and $\Sigma_2$.
    We again impose the stated regularity directly and do not pursue optimal sufficient conditions.

    \item The error terms $\e^{\frac{3-\alpha}{4}} +  \e^{\frac{\alpha - 1}{2}}$ are worse than in the critical regime. These errors are precisely the sharp errors  for the homogenization of the Stokes equations in a fixed perforated domain, see \cite{JingLuPrange, HoferLuOschmann26} (for the periodic setting).\footnote{Note that the correct exponents for the quantitative convergence result are written in the updated arxiv version of \cite{JingLuPrange}.}
    \end{enumerate}
\end{remark}

\subsection{Interpretation and examples}
\Cref{th:hom.critical,thm:hom.supercritical} ensure convergence whenever the data and the initial geometries converge in the sense
\[
\mathcal W_2(f_\e(0), f_0) +    \e^{\frac{\alpha-3}2}\|f_\e^{app}(0) - f_0\|_{(H^1_xW_\sigma^{1,\infty})^\ast}+  \|c_{\e,0} - c_0\|_{L^2(\Omega_\e)}  + \|h_\e - h\|_{L^2(S\times\Omega)}\to 0.
\]
In the critical case \(\alpha=3\), convergence rates of order \(\e\) are obtained, provided that the errors in the  data and  initial geometry converge at least at the same rate.
In the supercritical case \(1<\alpha<3\), the shown convergence rate is
\[
        \e^{\beta_\alpha},\qquad \beta_\alpha:=\min\left\{
            \frac{3-\alpha}{4},
            \frac{\alpha-1}{2}
        \right\},
\]
again provided that the preparation errors are of at least this order.
The exponent is maximal for $\alpha=\frac 53$ with $\beta_\alpha=\frac 13$ and it degenerates with $\beta_\alpha\to0$ for $\alpha\to 1$ and $\alpha\to3$.
Such convergence rates can usually be imposed directly on the initial concentration and the source term, but it is less clear for the initial geometry with the two separate errors
\[
\mathcal W_2(f_\e(0), f_0) +    \e^{\frac{\alpha-3}2}\|f_\e^{app}(0) - f_0\|_{(H^1_xW_\sigma^{1,\infty})^\ast}.
\]
These two errors are used at different steps in the proof.
The Wasserstein distance compares the initial discrete particle configuration with the limiting measure which is then propagated through the radius evolution (\cref{evolution.f,evolution.f.supercritical}).
Since the discrete measure \(f_\e\) is atomic in the spatial
variable and since point evaluation is not continuous on \(H^1(\Omega)\) in three dimensions, it cannot be used directly in the estimates involving spatial \(H^1\)-test functions.
To overcome this, we regularize via \(f_\e^{app}\) by spreading each point mass uniformly over a ball of radius $\sim \e$.
The mixed dual norm then measures the difference between this spatially regularized measure and \(f_0\).

We close this section with two examples of geometries that are admissible in our setup. 

\begin{example}[Uniform, periodic configuration]
Let $\Omega=(0,1)^3$ (such that $|\Omega|=1$ and the standard partition into $\e$-cubes (with $\e^{-1}\in\N$) consists of exactly $N_\e=\e^{-3}$ cells). 
Let \(z_{\e,i}\) be the centers of the \(\e\)-cubes of the standard cubic partition, and let
\[
    r_{\e,i,0}=r_*
\]
for all \(i\in\mathcal I_\e\) and some $r_*>0$.
Then
\[
    f_{\e,0}
    =
    \e^3\sum_{i\in\mathcal I_\e}
        \delta_{z_{\e,i}}\otimes\delta_{r_*},\qquad
    f_{\e,0}^{app}
    =
    \e^3\sum_{i\in\mathcal I_\e}
        \delta_{B_{\frac{\dmin\e}{4}}(z_{\e,i})}
        \otimes\delta_{r_*}.
\]
Both measures converge to
\[
    f_0(\d x,\d\sigma)
    =
    \d x\,\delta_{r_*}(\d\sigma).
\]
Moreover, for this configuration we have the natural approximation rates
\[
    \mathcal W_2(f_{\e,0},f_0)\lesssim\e,
    \qquad
    \|f_{\e,0}^{app}-f_0\|_
        {(H_x^1W_\sigma^{1,\infty})^*}
    \lesssim\e
\]
which is enough to ensure convergence of 
\[
\mathcal W_2(f_\e(0), f_0) +    \e^{\frac{\alpha-3}2}\|f_\e^{app}(0) - f_0\|_{(H^1_xW_\sigma^{1,\infty})^\ast}
\]
for all $\alpha\in(1,3]$.
In this case, the solution $f$ of the continuity equation remains mono-disperse in the radius variable. 
More precisely, $f(t,x,\sigma) = \delta_{r(t,x)}(\sigma)$, where $r$ solves (in the critical case)
\begin{subequations}
\begin{empheq}[left=\empheqlbrace]{alignat=2} 
    \partial_t r &= - \frac{H(c, r)}{ r} & \quad \text{in} \ S \times \Omega, \\
    r(0,x) &= r_\ast & \text{in} \ \Omega.
\end{empheq}
\end{subequations}
In particular, the coupling terms reduce to (in the critical case)
\[
    \int_0^\infty
        4\pi\sigma H(c,\sigma)\,
        f(t,x,\d\sigma)
    =    4\pi r(t,x)H(c(t,x),r(t,x)),\qquad
    \int_0^\infty
        6\pi\sigma\,
        f(t,x,\d\sigma)
    =
    6\pi r(t,x).
\]
\end{example}

\begin{example}[General radius distributions]
We note that, with our assumptions, the limit measure $f_0$ need not be atomic:
Take $\Omega=(0,1)^3$ and $f_0=\d x\otimes p(\d\sigma)$ a target limit measure with $p\in\mathcal{P}([\nicefrac1{\overline r},\overline r])$.
We let again \(z_{\e,i}\) be the centers of the \(\e\)-cubes of the standard cubic partition and choose the initial radii \(r_{\e,i,0}\) independently according to the probability measure \(p\).
Then, almost surely, both the empirical measure and its spatially regularized counterpart converge to
\[
    f_0(\d x,\d\sigma)
    =   \d x\otimes p(\d\sigma)
\]
where $p$ can, e.g., be atomic or absolutely continuous.

This construction can be extended to spatially non-homogeneous measures of the form
\[
f_0(\d x,\d\sigma)=\rho(x)\d x\, p_x(\d\sigma)
\]
for suitable spatial distributions $\rho$ and spatially dependent probability measures $p_x$, by combining non-periodic particle configurations approximating $\rho(x)\d x$ with radii sampled by $p_{z_{\e,i}}$.
\end{example}

\section{A-priori estimates and maximum principle}\label{ssec:maximum}
The analysis of the moving boundary problem is done in \cref{sec:analysis} via a contraction mapping argument which relies on transforming the problem to a fixed reference geometry.

There are, however, some very useful a-priori properties that can already be established directly on the moving geometry.
In particular, the model admits a maximum principle guaranteeing non-negativity and at most exponential growth of the bulk concentration.
More crucially, the evolution law for the radii yields uniform non-degeneracy bounds, ensuring that the perforations neither collapse nor grow unboundedly in finite time (\cref{theorem:maximum_principle}).
These estimates play a central role in \cref{sec:analysis}, specifically when showing long-time existence by preventing blow-up.
Before deriving these estimates, we briefly discuss the role of the dissipative structure of the interfacial exchange, which becomes particularly important in the critical regime.

\paragraph{The role of dissipativity in the critical regime.}

The dissipative structure of the interfacial exchange in Assumption~\textup{(A2)} is important in obtaining uniform $\e$ estimates.
Although the condition is not sharp and can be relaxed, there is a genuine obstruction in the critical regime $\alpha=3$ where non dissipative interface production may lead to growth rates that become unbounded as $\e\to0$.

This can already be seen via the following simplified problem in the fixed initial configuration for some $k>0$
    \begin{equation}\label{eq:simplified-problem}
    \left\{
    \begin{aligned}
    \partial_t c_\e-\Delta c_\e&=0
    &&\quad \text{in } S\times\Omega_\e,\\
    -\nabla c_\e\cdot \nu_\e&=k\e^{3-2\alpha}c_\e
    &&\quad \text{on } S\times\Gamma_\e,\\
    -\nabla c_\e\cdot \nu_\e&=0
    &&\quad \text{on } S\times\partial\Omega,\\
    c_\e(0)&=c_{\e,0}
    &&\quad \text{in } \Omega_\e.
    \end{aligned}
    \right.
    \end{equation}
In this problem, we have the linear and non-dissipative interface rate $k\e^{3-2\alpha}c_\e$.
For every fixed $\e>0$, this is a standard linear parabolic problem which admits a unique weak solution. 
Moroever, since $|\Gamma_\e|\sim\e^{2\alpha-3}$, the factor $\e^{3-2\alpha}$ precisely compensates for the total surface measure.

Testing with $c_\e$ gives
    \[
    \frac12\ddt\|c_\e\|^2_{L^2(\Omega_\e)}
    =k\e^{3-2\alpha}\|c_\e\|_{L^2(\Gamma_\e)}^2-\|\nabla c_\e\|^2_{L^2(\Omega_\e)}
    \]
and the trace estimate in \cref{lemma:trace_estimates} yields
\[
k\e^{3-2\alpha}\|c_\e\|_{L^2(\Gamma_\e)}^2-\|\nabla c_\e\|_{L^2(\Omega_\e)}^2
\le Ck\|c_\e\|_{L^2(\Omega_\e)}^2
-\bigl(1-Ck\e^{3-\alpha}\bigr)\|\nabla c_\e\|_{L^2(\Omega_\e)}^2.
\]
Hence, for every fixed \(k\ge0\), $\e$-uniform a priori estimates are readily available for sufficiently small $\e$ for all $\alpha\in(1,3)$.
In the critical case $\alpha=3$, however, the smallness of the factor \(\e^{3-\alpha}\) is lost, and the trace estimate provides uniform control only for sufficiently small $k$. 

We note that this is not merely a limitation of the trace estimate but a genuine critical instability:
The maximal exponential growth rate of the $L^2$-norm of solutions to \cref{eq:simplified-problem} is characterized by (cf.~\cite{Freitas2015,KrejcirikLotoreichik2020})
    \[
    \sigma_\e=\sup_{v\in H^1(\Omega_\e)\setminus\{0\}}\frac{k\e^{3-2\alpha}\|v\|_{L^2(\Gamma_\e)}^2-\|\nabla v\|^2_{L^2(\Omega_\e)}}{\|v\|^2_{L^2(\Omega_\e)}}.
    \]

Consider a particle of physical radius \(\e^\alpha r_{\e,i}\) and choose a localized radial test function.
To be more precise, let
\[
a_{\e}=\e^\alpha r_{\e,i},\qquad \lambda_\e=\e^{-\alpha}\left(\e^{3-\alpha}k-\frac{1}{r_{\e,i}}\right)
\]
and set
\[
v_\e=\frac{a_\e}{|x-z_{\e,i}|}\exp\left(-\lambda_\e(|x-z_{\e,i}|-a_\e)\right).
\]
After smoothly truncating $v_\e$ before reaching the neighboring particles, which only produces a lower-order error due to the separation of the scales, we can calculate
\[
\sigma_\e\gtrsim\e^{-2\alpha}\left(k\e^{3-\alpha}-c\frac1{r_{\e,i}}\right)
\]
for some universal $c$ (for sufficiently separated particles).
By the spectral theory for the Robin Laplacian (see, e.g., \cite[Section~2]{Freitas2015}), \(\sigma_\e\) is the largest eigenvalue of the self-adjoint generator \(A_\e\) associated with \cref{eq:simplified-problem}.
Let $\varphi_\e$ be a normalized eigenfunction associated with $\sigma_\e$ and chose $c_{\e,0}=\varphi_\e$.
The solution is then given by $c_\e(t)=e^{t\sigma_\e}\varphi_\e$.
Consequently, for $\alpha=3$ and $k$ sufficiently large
\[
\|c_\e(t)\|_{L^2(\Omega_\e)}=e^{t\sigma_\e}\to\infty
\]
for all $t>0$.
Thus, no $\e$-uniform $L^2$-estimate can hold for general non-dissipative interface production in the critical regime. 
Note that strictly sublinear production terms can still be controlled by the energy method, even in the critical case.

\paragraph{A priori estimates.}

\begin{theorem}\label{theorem:maximum_principle}
    Let $(c_\e,u_\e,p_\e,r_\e)$ be a solution in the sense of \cref{definition:weak_solution_moving} over the time interval $S=(0,T)$.
    Then, for almost all $t\in S$ and $x\in\Omega_\e(t)$,
    \[
    0\le c_\e(t,x)\le \max\left\{\overline c_0,c^{\mathrm{eq}}(\overline r)\right\}e^{2C_F t}.
    \]
    In addition, there are $\e$-independent functions $\underline\xi(t)$ and $\overline\xi(t)$ such that
    \[
    0<\underline\xi(t)\le r_{\e,i}(t)\le\overline\xi(t)<\infty
    \]
    for all $t\in S$ and $i\in\mathcal{I}_\e$.
    Also, there are $k,\gamma>0$ such that
    \[
    -\sup_{s\in[\underline{\xi}(t),\overline{\xi}(t)]}g(s)\cequ(s)\le \partial_tr_{\e,i}(t)\le ke^{\gamma t}\sup_{s\in[\underline{\xi}(t),\overline{\xi}(t)]}g(s).
    \]
\end{theorem}
\begin{proof}
The proof is done in two steps.
First, we establish a maximum principle via a truncation argument with which we establish the $L^\infty$-bounds for $c_\e$.
In a second step, the estimates for the radial evolution are derived by means of a comparison principle. \\[2mm]
\textit{Step 1: Maximum principle for the concentration.}
Let $\gamma>0$ and chose $k\in \N$ such that
\[
k\geq\max\left\{\overline c_0,c^{\mathrm{eq}}(\overline r)\right\}.
\]
To use the dissipative structure of the interfacial exchange, we test the concentration and radius equations simultaneously with the corresponding truncations.
We set
\[
    a(t)=ke^{\gamma t},\qquad
    c_k=(c_\e-a(t))_+,\qquad
    \cequ_k=(\cequ(r_\e)-a(t))_+,
\]
and note that $\varphi=c_k$ and $\eta=\cequ_k$ are valid test functions for the the weak form \eqref{eq:weak_form_alternative_derivative}.
With $c_k(0)=0$ and $\cequ(0)=0$, this means
\begin{multline*}
    (c_{\e}(t),c_k(t))_{L^2(\Omega_\e(t))}
    +\frac{4\pi\e^3}{3}\sum_{\mathcal{I}_\e}r_{\e,i}^3(t)c^{\mathrm{eq}}_k(t)\\
    -\int_0^t\langle \partial_tc_k,c_\e\rangle_{H^1(\Omega_{\e}(s))^*}
    -\frac{4\pi\e^3}{3}\sum_{\mathcal{I}_\e}r_{\e,i}^3\partial_tc^{\mathrm{eq}}_k
    +(\nabla c_\e-c_\e u_\e,\nabla c_k)_{L^2(\Omega_{\e}(s))}\di{s}\\
    +\e^{3-2\alpha}\sum_{\mathcal{I}_\e}\int_S(G(c_\e,r_{\e,i}),c^{\mathrm{eq}}_k-c_k)_{L^2(\Gamma_{\e,i}(s))}\di{s}
    =\int_0^t(F(c_\e),c_k)_{L^2(\Omega_{\e}(s))}\di{s}.
\end{multline*}
With \cref{lemma:times_derivative} and $c_\e c_k=c_k^2+ac_k$, we get
\begin{multline*}
(c_{\e}(t),c_k(t))_{L^2(\Omega_\e(t))}
-\int_0^t\langle \partial_tc_k,c_\e\rangle_{H^1(\Omega_{\e}(s))^*}\\
=\frac{1}{2}\|c_k(t)\|_{L^2(\Omega_\e(t))}^2
-\e^\alpha\int_0^t\int_{\Gamma_\e(s)}\left(\frac12c_k^2+a(s)c_k\right)\partial_tr_\e\di\gamma\di{s}
+\gamma\int_0^t\int_{\Omega_\e(s)}a(s)c_k\di{x}\di{s}
\end{multline*}
For the advection term, we find that
\[
\int_0^t(c_\e u_\e,\nabla c_k)_{L^2(\Omega_{\e}(s))}\di{s}
=\int_0^t\int_{\Omega_\e(s)}(c_k+a) u_\e\cdot \nabla c_k\di{x}\di{s}
=\int_0^t\int_{\Omega_\e(s)}\frac12u_\e\cdot \nabla|c_k|^2+au_\e\cdot\nabla c_k\di{x}\di{s}.
\]
Since $u_\e$ is divergence free, $u_\e=0$ on $\Sigma_2$, $c_\e=0$ on $\Sigma_1$, and $a$ is constant in space, 
\[
\int_0^t\int_{\Omega_\e(s)}\frac12u_\e\cdot \nabla|c_k|^2+au_\e\cdot\nabla c_k\di{x}\di{s}
=\int_0^t\int_{\Gamma_\e(s)}\left(\frac12c_k^2+a(s)c_k\right)u_\e\cdot\nu_\e\di\gamma\di{s}
\]
As a consequence (with $u_\e\cdot\nu_\e= \e^\alpha\partial_tr_\e$ on $\Gamma_\e(s)$),
\begin{multline*}
(c_{\e}(t),c_k(t))_{L^2(\Omega_\e(t))}
-\int_0^t\langle \partial_tc_k,c_\e\rangle_{H^1(\Omega_{\e}(s))^*}
-\int_0^t(c_\e u_\e,\nabla c_k)_{L^2(\Omega_{\e}(s))}\di{s}\\
=
\frac{1}{2}\|c_k(t)\|_{L^2(\Omega_\e(t))}^2
+\lambda\int_0^t\int_{\Omega_\e(s)}a(s)c_k\di{x}\di{s}.
\end{multline*}
For the diffusion term, we simply estimate
\[
\int_0^t(\nabla c_\e,\nabla c_k)_{L^2(\Omega_{\e}(s))}\di{s}
=\int_0^t\|\nabla c_k\|_{L^2(\Omega_{\e}(s))}^2\di{s}.
\]
For the interfacial exchange term, we use non-negativity of $g$ and monotonicity ot $T_a(z)=(z-a)_+$ to
estimate for each $i\in\mathcal{I}_\e$ (with Assumption $(A2)$):
\begin{align}\label{eq:max_interface}
    \int_0^t(G(c_\e,r_{\e,i}),c^{\mathrm{eq}}_k-c_k)_{L^2(\Gamma_{\e,i}(s))}\di{s}
    &=\int_0^t\int_{\Gamma_{\e,i}(s)}g(r_{\e,i})\left[c^{\mathrm{eq}}(r_{\e,i})-c_\e\right]\left(c^{\mathrm{eq}}_k-c_k\right)\di\gamma\di{s}\ge0.
\end{align}
The right hand side grows at most linearly (note that $c_\e(t)=c_k(t)+a(t)>0$ on the support of $c_k$):
\begin{align*}
\int_0^t(F(c_\e),c_k)_{L^2(\Omega_{\e}(s))}\di{s}
&\leq C_F\int_0^t\int_{\Omega_\e(s)}|c_k|^2+(1+a(s))c_k\di{x}\di{s}.
\end{align*}
Putting everything so far together, we have
\begin{multline}\label{eq:maximum_middle}
    \frac{1}{2}\|c_k(t)\|_{L^2(\Omega_\e(t))}^2
    +\frac{4\pi\e^3}{3}\sum_{\mathcal{I}_\e}(r_{\e,i}^3(t)-r_{c}^3)c^{\mathrm{eq}}_k(t)
    -\int_0^t\frac{4\pi\e^3}{3}\sum_{\mathcal{I}_\e}(r_{\e,i}^3-r_{c}^3)\partial_tc^{\mathrm{eq}}_k\di{s}\\
    +\int_0^t\left((\gamma-C_F)a(s)-C_F\right)\int_{\Omega_\e(s)}c_k\di{x}\di{s}
    +\int_0^t\|\nabla c_k\|_{L^2(\Omega_{\e}(s))}^2\di{s}
    \le C_F \int_0^t\int_{\Omega_\e(s)}|c_k|^2\di{x}\di{s}.
\end{multline}
Finally, we have to deal with the the dynamics at the interface. 
We have, for any $i\in\mathcal{I}_\e$,
\[
r_{\e,i}^3(t)c^{\mathrm{eq}}_k(t)
-\int_0^tr_{\e,i}^3(s)\partial_sc^{\mathrm{eq}}_k(s)\di{s}
=\int_0^t\partial_s(r_{\e,i}^3(s))c^{\mathrm{eq}}_k(s)\di{s}
\]
We take $M\colon S\to\R_{\ge0}$ via $M(t)=r^3_{\e,i}(t)$ and $Q\colon\R_{\geq0}\to\R_{\geq0}$ via $Q(m)=c^{\mathrm{eq}}(\sqrt[3]{m})$:
\[
\int_0^t\partial_s(r_{\e,i}^3(s))c^{\mathrm{eq}}_k(s)\di{s}
=
\int_0^t\partial_sM(s)(Q(M(s))-a(s))_+\di{s}.
\]
We introduce the primitive $\Psi\colon\R_{\ge0}\times\R_{\ge0}\to\R_{\ge0}$ via
\[
\Psi(M,a)=\int_0^M(Q(z)-a)_+\di{z}
\]
Since $Q$ is continuous and $M$ is absolutely continuous (because $r_{\e,i}$ is absolutely continuous), $t\mapsto \Psi(M(t),a(t))$ is absolutely continuous and it holds
\[
\ddt\left(\Psi(M(t),a(t))\right)=\partial_tM(t)\left(Q(M(t))-a(t)\right)_+ 
- \gamma a(t)\mathcal L^1\left(\{z\in[0,M(t)]\ : \ Q(z)>a(t)\}\right)
\]
for almost all $t\in S$.
Consequently,
\begin{align*}
    \int_0^t\partial_s(r_{\e,i}^3(s))c^{\mathrm{eq}}_k(s)\di{s}
    &=\int_0^t\ddt\left(\Psi(M(s),a(s))\right)\di{s}+\gamma\int_0^t\mathcal L^1\left(\{z\in[0,M(s)]\ : \ Q(z)>a(s)\}\right)a(s)\di{s}\\
    &\ge\Psi(M(t),a(t))-\Psi(M(0),a(0))=\Psi(M(t),a(t)).
\end{align*}
Here, $\Psi(M(0),a(0))=0$ since $Q(M(0))=\cequ(r_{\e,i,0})\le a(0)=k$.
Choosing $\gamma=2C_F$ we get from \cref{eq:maximum_middle}:
\begin{equation}\label{eq:max_princ_gronwall}
    \frac{1}{2}\|c_k(t)\|_{L^2(\Omega_\e(t))}^2
    +\int_0^t\|\nabla c_k\|_{L^2(\Omega_{\e}(s))}^2\di{s}
    +\frac{4\pi\e^3}{3}\sum_{\mathcal{I}_\e}\Psi(M_i(t),a(t))
    \le C_F\int_0^t\|c_k\|_{L^2(\Omega_{\e}(s))}^2\di{s}
\end{equation}
with $M_i(t)=r^3_{\e,i}(t)$. 
From here, $c_k(t)=0$  and therefore $c_\e(t,x)\le ke^{\gamma t}$ follows by Grönwall since $\Psi$ is non-negative.
This also implies $\Psi(M_i(t),a(t))=0$ for all $i\in\mathcal{I}_\e$, i.e., $Q(M_i(t))=c^{\mathrm{eq}}(r_{\e,i}(t))\le a(t)=ke^{\gamma t}$.
Next, we test with $\varphi=\min\{0,c_\e\}$ and $\eta=c^{\mathrm{eq}}(r_\e)$.
With the same arguments as before, the corresponding lower bound $c_\e(t,x)\ge 0$ follows.
\\[2mm]
\textit{Step 2: Comparison principle for the radius dynamics.}
The radii evolve via the dynamics
\[
\partial_tr_{\e,i}(t)=g(r_{\e,i}(t))\left(\fint_{\Gamma_{\e,i}(t)}c_\e(t,\gamma)\di\gamma-\cequ(r_{\e,i}(t))\right),\quad r_{\e,i}(0)=r_{\e,i,0}
\]
Let $\underline\xi,\overline\xi\in W^{1,\infty}(S)$ solve the ODEs (which are independent of $\e$ and $i\in\mathcal{I}_\e$)
\begin{alignat}{2}
\partial_t\underline\xi&=-g(\underline\xi)\cequ(\underline\xi),&\qquad \underline{\xi}(0)&=\frac1{\overline r},\label{eq:subsolution}\\
\partial_t\overline\xi&=g(\overline\xi)ke^{\gamma t},&\quad\overline\xi(0)&=\overline r\label{eq:supersolution}.
\end{alignat}
Clearly, $\underline\xi(t)\le r_{\e,i}(t)\le \overline\xi(t)$ for all $t\in S$ since they are sub/super solutions.
In the trivial case $g(\overline{r})=0$, we have $r_{\e,i}(t)\le\overline\xi(t)\equiv\overline{r}$.
Else, $g>0$ on $[\overline{r},\infty)$ and the solution $\overline\xi$ is global if and only if (Osgood criterion)
\[
\overline{\Xi}(r):=\int_{\overline{r}}^r\frac{1}{g(z)}\di{z}\stackrel{r\to\infty}{\to}\infty
\]
which is guaranteed via Assumption (A2).
The solution is (note that $\Xi$ is strictly increasing due to $g>0$).
\[
\overline\xi(t)=\overline{\Xi}^{-1}\left(\frac k{\gamma} (e^{\gamma t}-1)\right)\quad(t\in[0,\infty)].
\]
A similar argument is possible for the subsolution:
$\underline \xi(t)$ stays positive for all time if and only if
\[
\underline{\Xi}(r):=\int_r^{\frac1{\overline r}}\frac{1}{g(z)\cequ(z)}\di{z}\stackrel{r\to 0_+}{\to}\infty.
\]
which is again guaranteed by Assumption (A2).
The estimates for the time derivatives follow.
\end{proof}

\begin{remark}
    \begin{itemize}
        \item We actually have a stronger monotonicity estimate for the interface exchange term \cref{eq:max_interface},
        \begin{align*}
        \int_0^t(G(c_\e,r_{\e,i}),c^{\mathrm{eq}}_k-c_k)_{L^2(\Gamma_{\e,i}(s))}\di{s}
        &=\int_0^t\int_{\Gamma_{\e,i}(s)}g(r_{\e,i})\left[c^{\mathrm{eq}}(r_{\e,i})-c_\e\right]\left(c^{\mathrm{eq}}_k-c_k\right)\di{\gamma}\di{s}\\
        &\ge\int_0^t\int_{\Gamma_{\e,i}(s)}g(r_{\e,i})|c^{\mathrm{eq}}_k-c_k|^2\di{\gamma}\di{s}.
        \end{align*}
        which becomes useful when $g$ is uniformly positive.
        It follows because of 
        \[
        (x-y)(T_a(x)-T_a(y))\ge |T_a(x)-T_a(y)|^2
        \]
        where $x,y\in\R$ and $T_a(x)=(x-a)_+$.
        However, since we only assume $g\ge0$, this estimate is not helpful in our case.
        \item If there is $0\le r_c\leq\frac1{\overline r}$ such that $\cequ(r_c)=0$, the above argument also directly shows $r_{\e,i}(t)\ge r_c$ via
        \[
        \underline{\Xi}(r):=\int_r^{\frac1{\overline r}}\frac{1}{g(z)\cequ(z)}\di{z}\stackrel{r\to r_*}{\to}\infty
        \]
        for all $r_*<r$.
        This is relevant when the balls are considered to have a solid core which cannot be dissolved.
    \end{itemize}
\end{remark}

\begin{theorem}[Energy estimates]\label{theorem:energy estimates}
    Let $(c_\e,u_\e,p_\e,r_\e)$ be a solution in the sense of \cref{definition:weak_solution_moving} over the time interval $S=(0,T)$.
    Then, for all $t\in S$,
    \[
    \|c_\e\|_{L^\infty(0,t;L^2(\Omega_\e(s)))}^2+\|\nabla c_\e\|^2_{L^2(0,t;L^2(\Omega_\e(s)))}\lesssim_t1+\|c_{\e,0}\|_{L^2(\Omega_\e)}^2+\cequ(\overline r)\overline r^3.
    \]
\end{theorem}
\begin{proof}
Testing \cref{eq:weak_form_alternative_derivative} with $(c_\e,\cequ(r_\e))$ and arguing as in the proof of \cref{theorem:maximum_principle}, we obtain
\begin{multline*}
    \frac{1}{2}\|c_\e(t)\|_{L^2(\Omega_\e(t))}^2
    +\frac{4\pi\e^3}{3}\sum_{i\in\indexI}
    \int_0^t\partial_s(r_{\e,i}^3(s))\cequ(r_{\e,i}(s))\di{s}
    +\int_0^t\|\nabla c_\e\|_{L^2(\Omega_{\e}(s))}^2\di{s}\\
    \le C_F \int_0^t\int_{\Omega_\e(s)}(|c_\e|^2+|c_\e|)\di{x}\di{s}
    + \frac{1}{2}\|c_{\e,0}\|_{L^2(\Omega_\e)}^2.
\end{multline*}
Define
\[
    \Psi(M):=\int_0^M\cequ(\sqrt[3]{z})\di{z}.
\]
Then, by the chain rule,
\[
    \int_0^t\partial_s(r_{\e,i}^3(s))\cequ(r_{\e,i}(s))\di{s}
    =
    \Psi(r_{\e,i}^3(t))-\Psi(r_{\e,i}^3(0))
    \ge -\cequ(\overline r)\overline r^3,
\]
where we used that $\Psi\ge0$, $r_{\e,i}(0)\le\overline r$, and $\cequ$ is nondecreasing.
Hence, using $|\indexI|=\e^{-3}$,
\[
    \|c_\e(t)\|_{L^2(\Omega_\e(t))}^2
    +\int_0^t\|\nabla c_\e\|_{L^2(\Omega_{\e}(s))}^2\di{s}
    \lesssim
    \int_0^t\int_{\Omega_\e(s)}(|c_\e|^2+|c_\e|)\di{x}\di{s}
    +\|c_{\e,0}\|_{L^2(\Omega_\e)}^2
    +\cequ(\overline r)\overline r^3.
\]
The stated estimate follows from Grönwall's inequality.
\end{proof}

With the same strategy \cref{theorem:maximum_principle,theorem:energy estimates}, we can also establish estimates for the concentration for prescribed radial evolutions.

\begin{lemma}[A priori estimates for prescribed radial evolution]
\label{lemma:apriori_prescribed_radius}
    Let $r\in \mathcal{R}_\e(S)$ and $\e$ such that \cref{eq:eps_condition} holds and $\frac1{\lambda}\le r_{i}(t)\le \lambda$ for some $\lambda>\overline r$.
    Further, let $(c_\e,u_\e,p_\e)$ be any solution in the sense of \cref{definition:weak_solution_moving} with the radial evolution prescribed by $r$.    
    Then,
    \begin{align*}
    \|c_\e\|_{L^\infty(0,t;L^2(\Omega_\e(s)))}^2
    +\|\nabla c_\e\|_{L^2(0,t;L^2(\Omega_{\e}(s)))}^2
    \lesssim_t 
    1+\|c_{\e,0}\|_{L^2(\Omega_\e)}^2+t\sup_{\gamma\in[\nicefrac1\lambda,\lambda]}g(\gamma)\left[\gamma\cequ(\gamma)\right]^2.
    \end{align*}
    for all $t\in S$.
    Moreover,
    \[
    0\le c_\e(t,x)\le \max\{\overline c_0,\cequ(\lambda)\}e^{2C_F t}.
    \]
\end{lemma}
\begin{proof}
    Let $r\in \mathcal{R}_\e(S)$ and $\e$ such that \cref{eq:eps_condition} holds.
    We test with $c_\e$ in \cref{eq:micro_system_moving} and estimate
    \begin{multline*}
    \frac{1}{2}\|c_\e(t)\|_{L^2(\Omega_\e(t))}^2
    +\int_0^t\|\nabla c_\e\|_{L^2(\Omega_{\e}(s))}^2\di{s}\\
    \le C_F \int_0^t\int_{\Omega_\e(s)}|c_\e|^2+|c_\e|\di{x}\di{s}+ \frac{1}{2}\|c_{\e,0}\|_{L^2(\Omega_\e)}^2
    +\e^{3-2\alpha}\sum_{\mathcal{I}_\e}\int_0^t(g(r_{i})(\cequ(r_i)-c_\e),c_\e)_{L^2(\Gamma_{\e,i}(s))}\di{s}.
    \end{multline*}
    Using the identity
    \begin{equation}\label{eq:split_polarization}
    (\cequ-c)c=\frac{1}{2}(-c^2-(c-\cequ)^2+(\cequ)^2),
    \end{equation}
    we control the exchange term via
    \begin{align*}
    \e^{3-2\alpha}\sum_{\mathcal{I}_\e}\int_0^t(g(r_{i})(\cequ(r_i)-c_\e),c_\e)_{L^2(\Gamma_{\e,i}(s))}\di{s}
    &\le\frac{\e^{3-2\alpha}}2\sum_{\mathcal{I}_\e}\int_0^t|\Gamma_{\e,i}(s)|g(r_i)\left[\cequ(r_i(s))\right]^2\di{s}\\
    &=2\pi\e^{3}\sum_{\mathcal{I}_\e}\int_0^tg(r_i(s))\left[r_i(s)\cequ(r_i(s))\right]^2\di{s}\\
    &\le 2\pi t\sup_{\gamma\in[\nicefrac1\lambda,\lambda]}g(\gamma)\left[\gamma\cequ(\gamma)\right]^2    
    \end{align*}
    Hence, Grönwall's inequality gives the desired energy estimate.

    The $L^\infty$-estimates follow essentially along the same line as in \cref{theorem:maximum_principle} when testing with $c_k=\max\{c_\e-a(t),0\}$ for $a(t)=\max\{\overline c_0,\cequ(\lambda)\}e^{2C_Ft}$.
    The only difference is in the surface exchange term.
    Since \(r_i(t)\le \lambda\) and \(\cequ(r_i(t))\le a(t)\), we have \(c_\e=c_k+a\) on the support of \(c_k\), and therefore
    \[
        \int_0^t
            \bigl(
                g(r_i)(\cequ(r_i)-c_\e),
                c_k
            \bigr)_{L^2(\Gamma_{\e,i}(s))}
        \di s
        =
        \int_0^t
            \int_{\Gamma_{\e,i}(s)}
                g(r_i)
                \bigl(\cequ(r_i)-a-c_k\bigr)c_k
            \di\gamma\di s
        \le 0.
        \qedhere
    \]
\end{proof}

We recall the super/subsolutions $\underline\xi,\overline\xi\in W^{1,\infty}(S)$ from the maximum principle (with $k>0$),
\begin{alignat*}{2}
\underline \xi'&=-g(\underline \xi)\cequ(\underline \xi),&\quad \underline \xi(0)&=\frac1{\overline r},\\
\overline\xi_k'&=kg(\overline\xi)e^{2C_Ft},&\quad\overline\xi(0)&=\overline r.
\end{alignat*}
We set
\[
    k_0:=\max\{\overline c_0,\cequ(\overline r)\},
    \qquad
    \overline\xi:=\overline\xi_{k_0},
\]
so that $\underline\xi$ and $\overline\xi$ are precisely the
sub- and supersolutions used in the proof of
\cref{theorem:maximum_principle}.
For the fixed-point argument in the next section, we additionally need the dependence of the radius evolution on the prescribed concentration.

%
\begin{lemma}\label{lemma:concentration_to_radius}
    For every $c\in L^2(S;H^1(\Omega))$ with $0\le c(t,x)\le ke^{2C_Ft}$ for some $k>0$, there exists a unique solution $r_\e\in\mathcal R_\e(S)$ to the ODEs
    \[
    r'_{\e,i}=-\fint_{\Gamma_{\e,i}}G(c,r_{\e,i})\di{\gamma},\quad r_{\e,i}(0)=r_{\e,i,0} \qquad(i\in\indexI,\, t\in S).
    \]
    Then,
    \[
        0<\underline\xi(t)
        \le r_{\e,i}(t)
        \le \overline\xi_k(t)<\infty.
    \]
    For $c^{(1)},c^{(2)}\in L^2(S;H^1(\Omega))$ with $0\le c^{(j)}(t,x)\le ke^{2C_Ft}$, let $r_\e^{(j)}\in \mathcal R_\e(S)$ be the corresponding radial evolutions.
    For all $t>0$, it holds
    \[
    \|\delta r_\e\|_{L^\infty(0,t)}\lesssim_{\e^{-1},t}\sqrt{t}\|\delta c\|_{L^2(0,t;H^1(\Omega))},\quad \|\delta\partial_t r_\e\|_{L^2(0,t)}\lesssim_{\e^{-1},t}\|\delta c\|_{L^2(0,t;H^1(\Omega))}
    \]
    where we use $\delta$ to denote the differences, e.g., $\delta c=c^{(1)}-c^{(2)}$.
\end{lemma}
\begin{proof}
   \textit{Step 1: Well-posedness and boundedness.}
    Since $G$ is locally Lipschitz, there is a unique solution $r_{\e,i}\in W^{1,\infty}(0,t_\e)$ over some potentially small time interval $(0,t_\e)$.
    In the exact same way as in \cref{theorem:maximum_principle}, we get
    \begin{equation}\label{eq:estimate_max}
    0< \underline \xi(t)\le r_{\e,i}(t)\le \overline{\xi}_k(t)
    \end{equation}
    on the existence interval for monotone functions $\underline \xi,\overline{\xi}_k\colon[0,\infty)\to(0,\infty)$.
    Therefore, the solution $r_\e\in\mathcal R_\e(S)$ exists globally over $S$.
    The estimates for $r_{\e,i}(t)$ follow via direct calculation via the ODEs.
\\[2mm]
    \textit{Step 2: Lipschitz estimates.} Let $L_{t,k}$ be the Lipschitz constant of $G$ on $[0,ke^{2C_Ft}]\times[0,\overline\xi_k(t)]$.
    We estimate via
    \[
    |\delta r_{\e,i}(t)|\le L_{t,k}\int_0^t\left(
    |\delta r_{\e,i}(\tau)|+\fint_{\Gamma_{\e,i}}|\delta c(x,\tau)|\di{\gamma}\right)\di{\tau}
    \]
    and use Grönwall to get
    \[
    \|\delta r_{\e,i}\|_{L^\infty(0,t)}\le L_te^{tL_t}\left(\frac{t}{|\Gamma_{\e,i}|}\right)^{\frac12}\|\delta c\|_{L^2((0,t)\times\Gamma_{\e,i})}.
    \]
    Applying the trace estimate \cref{lemma:trace_estimates}, we get the desired estimate.
    The Lipschitz estimate for the time derivatives follows easily via direct calculation.
\end{proof}

\section{Well-posedness of the microscopic problem}\label{sec:analysis}
In this section, we establish existence and uniqueness of solutions to the $\e$-problem in the sense of \cref{definition:weak_solution_moving} given by Systems~\eqref{system:stokes_eps}, \eqref{system:diffusion_eps}, and~\eqref{system:radius_eps}.
The general strategy follows similar lines as in \cite{gahn_rigorous_2024}:
We start by transforming the moving-domain problem to the fixed reference configuration \(\Omega_\e=\Omega_\e(0)\). 
For a prescribed admissible radius evolution, we then solve the transformed Stokes and diffusion--reaction problems and derive Lipschitz estimates with respect to the prescribed radius evolution.
The resulting solution operators are combined with the radius ODE to obtain a local solution by a contraction argument.
Finally, we use the a priori bounds from \cref{ssec:maximum} to show that no degeneration of the radii can occur as long as $\e$ is sufficiently small.
This allows us to extend the solution to any prescribed time interval, which proves \cref{theorem:well_posedness})


\subsection{Problem in reference configuration}\label{ssec:reference_config}
We have an initial separation of $\dmin\e$ between any two balls (and between balls and $\partial\Omega$) and the size of these balls scale like $\e^\alpha$ with $\alpha>1$.
Let $\lambda>\overline r$. 
Then,
\[
|z_{\e,i}-z_{\e,j}|-2\e^\alpha\lambda\ge\e\frac{\dmin}{2}
\]
for all $i\neq j$ and all $0<\e\le\e_\lambda$ where
\begin{align}\label{eq:eps_lambda}
\e_\lambda:=\left(\frac{\dmin}{4\lambda}\right)^{\frac1{\alpha-1}}
\end{align}
In the following, let $T>0$ be arbitrary but fix and set $S=(0,T)$.
Also, let $r\in\mathcal R_\e(S)$ (see \eqref{mathcalR} for the definition) with $r_{i}(t)\le\lambda$ and $\e\le\e_\lambda$.
As a consequence, the balls stay separated by $\e\frac{\dmin}{2}$ and we recall
\begin{equation}\label{eq:omega_time}
\Omega_{\e}(t):=\Omega\setminus \bigcup_{i\in\indexI}B_{\e^\alpha r_i(t)}(z_{\e,i}),\quad \Omega_\e:=\Omega_\e(0).
\end{equation}
The idea is to describe the evolution of each hole via their own transformation and to then patch them together.
We introduce the small annuli
\[
 \mathcal U_{\e,i}^\lambda:=B_{2\e^{\alpha}\lambda}(z_{\e,i})\setminus B_{\e^\alpha r_{\e,i,0}}(z_{\e,i}),\qquad \mathcal U_\e^\lambda:=\bigcup_{i\in\indexI}\mathcal U_{\e,i}^\lambda,
\]
where $(\mathcal U_{\e,i}^\lambda)_{i\in\indexI}$ are pairwise disjoint for all $\e\le\e_\lambda$.
We then take a smooth cut-off function $\chi\colon\R\to[0,1]$ with 
\[
\chi(r)=1\  \text{in}\ (-\infty,1],\qquad\chi(r)=0\ \text{in} \ [2,\infty),\qquad -2\leq\chi'\leq0.
\]
We also set
\[
\rho_{\e,i}(x)=|x-z_{\e,i}|,\quad R_{\e,i}^\lambda(t,\rho)=\rho+\e^{\alpha}(r_i(t)-r_{\e,i,0})\chi\left(1+\frac{\rho-\e^\alpha r_{\e,i,0}}{\e^\alpha(2\lambda-r_{\e,i,0})}\right)
\]
and define, for $x\in \mathcal U_{\e,i}^\lambda$, the local transformation
\[
\Psi_{\e,i}^\lambda(t,x)=z_{\e,i}+R_{\e,i}^\lambda(t,\rho_{\e,i}(x))\frac{x-z_{\e,i}}{|x-z_{\e,i}|}
\]
By construction, $\Psi_{\e,i}^\lambda(t,\cdot)$ acts only in the radial
direction. On the initial boundary, the cut-off argument equals 1, and hence $\chi=1$.
Hence, for $x\in\Gamma_{\e,i}$,
\[
\left|\Psi_{\e,i}^\lambda(t,x)-z_{\e,i}\right|
=
\e^\alpha r_{\e,i,0}
+
\e^\alpha\bigl(r_i(t)-r_{\e,i,0}\bigr)
=
\e^\alpha r_i(t),
\]
and therefore
\[
\Psi_{\e,i}^\lambda(t,\Gamma_{\e,i})
=
\Gamma_{\e,i}(t).
\]
On the other hand, the cut-off function vanishes at the outer boundary of $\mathcal U_{\e,i}^\lambda$, so
\[
\Psi_{\e,i}^\lambda(t,x)=x
\]
outside $B_{2\e^\alpha\lambda}(z_{\e,i})$.
Since the annuli $\bigl(\mathcal U_{\e,i}^\lambda\bigr)_{i\in\mathcal I_\e}$ are pairwise disjoint, and since the local transformations agree with the identity near their outer boundaries, they can be patched together by setting
\begin{equation}
\label{eq:transform_psi}
\Psi_\e(t,x):=
\begin{cases}
\Psi_{\e,i}^\lambda(t,x),
&
x\in\mathcal U_{\e,i}^\lambda,\quad i\in\mathcal I_\e,
\\
x,& x\in\Omega_\e\setminus\mathcal U_\e^\lambda.
\end{cases}
\end{equation}

In the following, we suppress the $\lambda$-superscript for the transformation, i.e., we write $\Psi_\e$ instead of $\Psi_\e^\lambda$, to not overload the notation.
Note that we can chose $\lambda$ aribtrarily large as long as we constrain ourselves to $\e\le\e_\lambda$.
For $\e\le\e_\lambda$ and $r\in\mathcal R_\e(S)$ with $r_{i}(t)\le\lambda$, the map $\Psi_\e(t,\cdot)$ is a smooth diffeomorphism from $\Omega_\e$ onto $\Omega_\e(t)$ which satisfies $\Psi_\e\in W^{1,\infty}(S;C^\infty(\overline{\Omega_\e}))^3$.
We introduce the corresponding kinematic quantities
\[
F_\e=\nabla\Psi_\e^T,\quad J_\e=\det F_\e,\quad A_\e=J_\e F_\e^{-1},\quad J_\e^\Gamma=J_\e|F_\e^{-T}\nu_\e|,\quad D_\e=J_\e F_\e^{-1}F_\e^{-T}.
\]
In addition, we denote the transformed linearized symmetric gradient $e_\e\colon H^1(\Omega_\e)^3\to L^2(\Omega_\e)^{3\times3}$ defined by
\begin{equation}\label{symmetric_gradient_reference}
e_\e(\phi)=\frac12\left(F_\e^{-T}\nabla\phi+(F_\e^{-T}\nabla\phi)^T\right).
\end{equation}
Finally, we introduce the pull backs to the reference geometry $(t\in S,\ x\in\Omega_\e)$
\[
\hat{c}(t,x)=c(t,\Psi_\e(t;x)),\quad \hat{u}(t,x)=u(t,\Psi_\e(t;x)),\quad \hat{p}(t,x)=p(t,\Psi_\e(t;x)).
\]
and the relative velocity $\hat{v}_\e=\hat{u}_\e-\partial_t\Psi_\e$.
Note that $\hat{v}_\e=\hat{u}_\e$ in $\Omega_\e\setminus \mathcal U_\e(t)$ and $\hat v_\e=0$ on $\Sigma_2\cup\Gamma_\e$.
Some important estimates and properties of this transformation which are used in the analysis but are mostly technical are presented in \cref{appen:trafo}.

We introduce the space
\[
W(S;\Omega)=\{u\in L^2(S;H^1(\Omega_\e)\ : \ \partial_tu\in L^2(S;H^1(\Omega_\e))^*\}.
\]
With this we are ready to formulate the concept of a weak solution for the moving boundary problem.
\begin{definition}[Weak solutions (reference geometry)]
\label{definition:weak_solution_fixed}
Let $T>0$ and $S=(0,T)$. We say that
\[
(\hat{c}_\e,\hat{u}_\e,\hat p_\e,r_\e)\in W(S;\Omega)\times L^\infty(S;H^1_{\Sigma_2}(\Omega_\e))^3\times \LtLx{2}{2}{\Omega_\e}\times W^{1,\infty}(S)^\indexI
\]
is a local-in-time, weak solution to the moving boundary problem over $S$ (with respect to the reference configuration) if the following conditions are satisfied:

\begin{itemize}
    \item[$(i)$] The size of the balls are characterized by the radii function $r_\e$ via $\e^\alpha r_{\e,i}$ described via the corresponding diffeomorphism $\Psi_\e$ as defined in \cref{eq:transform_psi}.
    Moreover, they satisfy the ODEs 
    \[
    \partial_tr_{\e,i}=-\fint_{\Gamma_{\e,i}}G(\hat c_\e,r_{\e,i})\di{\gamma},\quad r_{\e,i}(0)=r_{\e,i,0}
    \]
    for almost all $t\in S$ and all $i\in\mathcal{I}_\e$.
    \item[$(ii)$] It is $\hat{u}_\e=\e^\alpha\partial_tr_{\e,i}\nu_\e$ on $\Gamma_\e$ and it holds
    \begin{subequations}\label{eq:micro_system_fixed}
    \begin{align}
        2\e^{3-\alpha}(J_\ezr e_{\e}(\hat{u}_\e),e_{\e}(\psi))_{L^2(\Omega_\e)}-(\hat{p}_\e,\dive(A_{\e}\psi))_{L^2(\Omega_\e)}
        &=(J_{\e}\hat{h},\psi)_{L^2(\Omega_{\e})},\label{eq:micro_system:b_fixed}\\
        (\dive(A_\e\hat{u}_\e),\phi)_{L^2(\Omega_\e)}&=0\label{eq:micro_system:c_fixed}
    \end{align}
    for all $(\psi,\phi)\in H^1_{\Sigma_2\cup\Gamma_\e}(\Omega_\e)^3\times L^2(\Omega_\e)$ and for almost all $t\in S$.    
    \item[$(iii)$] It is $\hat c_\e=0$ on $\Sigma_1$ and for all $\varphi\in C^\infty(\overline{S\times\Omega})$ with $\phi(T)=0$ and $\phi=0$ on $\Sigma_1$, it holds
    \begin{multline}\label{eq:micro_system:a_fixed}
    \int_S-(J_{\e}\hat{c}_\e,\partial_t\varphi)_{L^2(\Omega_\e)}+(D_{\e}\nabla\hat{c}_\e-A_{\e}\hat{c}_\e(\hat{u}_\e-\partial_t\Psi_\e),\nabla\varphi)_{L^2(\Omega_{\e})}\di{t}\\
     +\e^{3-2\alpha}\int_S(J_\e^\Gamma G(\hat{c}_\e,r_{\e}),\varphi)_{L^2(\Gamma_{\e})}\di{t}
    =\int_S(J_{\e}F(\hat{c}_\e),\varphi)_{L^2(\Omega_{\e})}\di{t}+(c_{\e,0},\varphi(0))_{L^2(\Omega_{\e})}.
    \end{multline}
    \end{subequations}
\end{itemize}
\end{definition}

\begin{remark}
\begin{itemize}
\item The tuples
\(
    (u_\e,p_\e,c_\e,r_\e)
\)
and \(
    (\hat u_\e,\hat p_\e,\hat c_\e,r_\e)
\)
solve the moving and reference formulations in the sense of \cref{definition:weak_solution_moving} and \cref{definition:weak_solution_fixed}, respectively, if and only if they are related by the pullback induced by $\Psi_\e$.
This follows from the standard transformation rules together with the mapping properties collected in \cref{lemma:pushforward_estimates}.
    \item
Note that $J_\e^{\Gamma}(t)=J_\e(t)|F_\e^{-T}(t)\nu_\e|=\left(\frac{r_{\e,i}(t)}{r_{\e,i,0}}\right)^2$ is constant on every $\Gamma_{\e,i}$.
With 
\[
\partial_tr_{\e,i}=-\frac{1}{|\Gamma_{\e,i}(t)|}\int_{\Gamma_{\e,i}}J_\e^{\Gamma}G(\hat c_\e,r_{\e,i})\di{\gamma},
\]
we get the coupled formulation
\begin{multline*}
    \langle\partial_t(J_{\e}\hat{c}_\e),\varphi\rangle_{H^1(\Omega_\e)}
    +\frac{4\pi\e^3}{3}(\partial_t(r_{\e}^3),\eta)_{\ell^2(\indexI)}
    +(D_{\e}\nabla\hat{c}_\e-A_{\e}\hat{c}_\e\hat{v}_\e,\nabla\varphi)_{L^2(\Omega_{\e})}\\
     +\e^{3-2\alpha}(r_\e^2 G(\hat{c}_\e,r_{\e}),\varphi-\eta)_{L^2(\Gamma_{\e})}
    =(J_{\e}F(\hat{c}_\e),\varphi)_{L^2(\Omega_{\e})}.
\end{multline*}
which has to hold for all $\varphi\in H_{\Sigma_1}^1(\Omega)$ all $\eta\in\R^\indexI$ and for almost all $t\in S$.
Here, we identify a vector $\eta\in\R^\indexI$ as a function $\eta\in L^2(\Gamma_\e)$ via $\eta=\eta_i$ on every component $\Gamma_{\e,i}$.
\item We may also solve for the relative velocity $\hat v_\e=\hat u_\e-\partial_t\Psi_\e\in H^1_{\Sigma_2\cup\Gamma_\e}(\Omega_\e(t))$ via
    \begin{align*}
        2\e^{3-\alpha}(J_\ezr e_{\e}(v_\e),e_{\e}(\psi))_{L^2(\Omega_\e)}-(\hat{p}_\e,\dive(A_\e\psi&))_{L^2(\Omega_\e)}\\
        &=(J_{\e}\hat{h},\psi)_{L^2(\Omega_{\e})}
        -2\e^{3-\alpha}(J_{\e}e_\e(\partial_t\Psi_\e),e_\e(\psi))_{L^2(\Omega_{\e})},\\
        (\dive(A_\e v_\e),\phi)_{L^2(\Omega_\e)}&=-(\dive(A_\e\partial_t\Psi_\e),\phi)_{L^2(\Omega_\e)},
    \end{align*}
    and
    \begin{multline*}
    \int_S-(J_{\e}\hat{c}_\e,\partial_t\varphi)_{L^2(\Omega_\e)}+(D_{\e}\nabla\hat{c}_\e-A_{\e}\hat{c}_\e\hat{v}_\e,\nabla\varphi)_{L^2(\Omega_{\e})}\di{t}\\
     +\e^{3-2\alpha}\int_S(J_\e^\Gamma G(\hat{c}_\e,r_{\e}),\varphi)_{L^2(\Gamma_{\e})}\di{t}
    =\int_S(J_{\e}F(\hat{c}_\e),\varphi)_{L^2(\Omega_{\e})}\di{t}+(c_{\e,0},\varphi(0))_{L^2(\Omega{\e})}.
    \end{multline*}
\end{itemize}
\end{remark}

\subsection{Analysis of the Stokes system}\label{ssec:stokes_radius_fixed}
In order to set up a fixed point argument, we first consider a prescribed evolution of the domain through a given  $r\in\mathcal{R}_{\e}(S)$ with $\frac1\lambda\le r_i(t)\le\lambda$ and $|\partial_t r_{\e,i}| \leq M$. In particular, this fixes the transformation $\Psi_\e$.
In the following, let $M>0$ and $\lambda>\overline r$ be fixed but arbitrary and $\e\le\e_\lambda$.

In this section, we analyze the quasi-steady Stokes system for a prescribed radius evolution function.
In this setting, time enters only as a parameter through the time-dependent radii and the corresponding transformation.
Related problems for evolving balls have been analyzed and discussed in detail before in \cite{gahn_rigorous_2024,wiedemann_homogenisation_2024} for the scaling $\alpha=1$ in periodic configurations.
However, since our domain is not periodic and since we consider the different scaling $\alpha>1$, there are some subtle differences.
We can leverage this scale separation into better $\e$-estimates which are very helpful for the contraction mapping argument later.
In particular, on sufficiently small time intervals, the Lipschitz estimates in \cref{lemma:stokes_for_fixed_radius} contain an explicit \(\e\)-control which is not available in \cite{gahn_rigorous_2024}.

For the estimates on $\hat u_\e$, we rely on the following Bogovskiǐ operator.
\begin{proposition} \label{pro:Bog}
    There exists a continuous linear operator $\Bog_\e \colon L^2(\Omega_\e) \to H^1_{\Sigma_2}(\Omega)^3$
    such that for all $f \in L^2(\Omega_\e)$ the function $v= \Bog_\e f$ satisfies $ v  = 0$ in $\cup_i B_{\e,i}$, $\dv v = f $ in $\Omega_\e$ and 
    \begin{align*}
         \e^{\frac{3 - \alpha}2} \|\nabla v \|_{L^2(\Omega)} + \|v\|_{L^2(\Omega)} \lesssim  \|f \|_{L^2(\Omega_\e)}
    \end{align*}
\end{proposition}
\begin{proof}
    Indeed, the result for $f \in L^2_0(\Omega_\e)$ is contained in \cite[Theorem 2.3]{DieningFeireislLu}, where  $L^2_0(\Omega_\e) := \{f \in L^2(\Omega_\e) : \int_{\Omega_\e)} f = 0\}$. This condition is necessary to get $\Bog_\e f \in H^1_0(\Omega_\e)^3$, but we do not impose zero boundary conditions on all of $\partial \Omega$. 

    To apply the result from \cite[Theorem 2.3]{DieningFeireislLu}, consider a Lipschitz domain $U \subset \R^3$ with  $\Omega \subset U$ and $\Sigma_2 \subset \partial U$ such that there exists $A \subset U \setminus \Omega$ with $|A| > 0$    and denote $U_\e := U \setminus \cup_i B_{\e,i}$.
    Let  $\Bog_{\e}^\circ$ be the operator from \cite[Theorem 2.3]{DieningFeireislLu} for the perforated domain $U_\e$. 
    Implicitly extending any $f \in L^2(\Omega_\e)$ by $0$ to a function in $L^2(U_\e)$ we define
    \begin{align*}
        v := \Bog_{\e}^\circ \Big(f - \frac{\1_A}{|A|} \int_{\Omega_\e} f\Big). 
    \end{align*}
    We then define $\Bog_\e f$ as the restriction of $v$ to $\Omega$.
    Then
    \begin{align*}
        \|\nabla \Bog_\e f\|_{L^2(\Omega)} \leq  \Big\|\nabla \Bog_{\e}^\circ \Big(f - \frac{\1_A}{|A|} \int_{\Omega_\e} f\Big)\Big\|_{L^2(U)} 
         \lesssim  \e^{\frac{\alpha - 3}2} \Big\|f - \frac{\1_A}{|A|}\int_{\Omega_\e} f\Big\|_{L^2(\Omega_\e)} 
        \lesssim   \e^{\frac{\alpha - 3}2} \|f \|_{L^2(\Omega_\e)}
    \end{align*}
    To get the estimate for the the $L^2$ norm of $v$, we apply the Poincar\'e inequality from Lemma \ref{lemma:poincare}. 
\end{proof}

\begin{lemma}\label{lemma:stokes_for_fixed_radius}
Let $M>0$, $\lambda\ge\overline r$, and $\e\le\e_\lambda$.
Let $r^{(1)},r^{(2)}\in \mathcal{R}_{\e}(S)$ with $\nicefrac1\lambda\le r_i^{(j)}(t)\le\lambda$ for all $t\in S$ and $\|\partial_tr^{(j)}\|_{L^\infty(S;\ell^\infty(\indexI))}\le M$.
\begin{itemize}
    \item[$(i)$] There is a unique weak solution $(\hat u_\e,\hat p_\e)\in H^1_{\Sigma_2}(\Omega_\e)^3\times L^2(\Omega_\e)$ to the Stokes problem \eqref{eq:micro_system:b_fixed}-\eqref{eq:micro_system:c_fixed} with the radial evolution $r=r^{(1)}$.
    This solution satisfies the energy estimate
    \begin{align}\label{eq:apriori_stokes_reference}
        \|\hat u_\e\|_{{L^2(\Omega_\e)}}+\e^{\frac{3-\alpha}{2}}\|\nabla\hat u_\e\|_{L^2(\Omega_\e)}&\lesssim_\lambda \|\hat h_\e\|_{L^2(\Omega_\e)}+\e^{\frac{3-\alpha}{2}}\|\partial_tr\|_{\ell^\infty(\mathcal{I}_\e)}
    \end{align}
    for almost all $t\in S$.
    \item[$(ii)$] Let $(\hat{u}_\e^{(j)},\hat p_\e^{(j)})$ be the solution to Stokes associated to radii functions $r^{(j)}$.
    Then,
    \begin{align}\label{eq:lipschitz_stokes}
    \e^{\frac{3-\alpha}{2}}\|\delta \nabla \hat v_\e\|_{L^2(\Omega_\e)}
    +\|\delta\hat  p_\e\|_{L^2(\Omega_\e)}
    \lesssim_{\lambda,M}\left(\e^{\frac{\alpha-3}{2}}\|\delta r\|_{\ell^\infty(\indexI)}
        +\e^{\frac32(\alpha-1)}\|\delta\partial_tr\|_{\ell^\infty(\mathcal{I}_\e)}\right)
    \end{align}
    where $\delta\hat  v_\e=\delta \hat u_\e-\partial_t\delta\Psi_\e$.
    Here, $\Psi_\e^{(j)}$ denotes the diffeomorphism corresponding to $r^{(j)}$ as defined in \eqref{eq:diffeomorphism_psieps}.
\end{itemize}
\end{lemma}

\begin{proof}
   $(i)$.
   For any given $r_\e\in\mathcal R_\e(S)$, this is a quasi-steady Stokes system with regular data, so the existence of unique solutions $(\hat u_\e(t),\hat p_\e(t))\in H^1_{\Sigma_2}(\Omega_\e)^3\times L^2(\Omega_\e)$ is easy to see for almost every $t\in S$.
   In the following, we suppress the time variable to lighten the notation; with the tacit understanding that all involved functions are time dependent.
   For convenience, we introduce the spaces $V=H^1_{\Sigma_2\cup\Gamma_\e}(\Omega_\e)^3$ and $Q=L^2(\Omega_\e)$ and furnish $V$ with the norm $\|v\|_V=\|\nabla v\|_{L^2(\Omega_\e)}$ and identify $Q$ with its dual $Q^*$.   
   
   The interesting part of this statement are the $\e$-controlled estimates.
    We introduce the forms\footnote{Here, we write $a_\e(t;\partial_t\Psi_\e,w)$ by a slight abuse of notation since $\partial_t\Psi_\e\notin V$. Note that $a_\e$ is clearly also well-defined over $H^1(\Omega)$ so the meaning is clear.}
    \begin{alignat*}{2}
    a_\e(t;\cdot,\cdot)&\colon V\times V\to\R,&\quad a_\e(t;v,w)&:=2\e^{3-\alpha} (J_\e e_\e(v),e_\e(w))_{L^2(\Omega_\e)},\\
    b_\e(t;\cdot,\cdot)&\colon Q\times V\to\R,&\quad b_\e(t;q,w)&:=-(q,J_\e F_\e^{-1}:\nabla w))_{L^2(\Omega_\e)},\\
     H_\e(t;\cdot)&\colon V\to\R,&\quad  H_\e(t;w)&:=(J_\e \hat h_\e,w)_{L^2(\Omega_\e)}-a_\e(t;\partial_t\Psi_\e,w),\\
    G_\e(t;\cdot)&\colon Q\to\R,&\quad G_\e(t;q)&:=(q,\partial_tJ_\e)_{L^2(\Omega_\e)}
    \end{alignat*}
    and note that $(\hat v_\e:=\hat u_\e-\partial_t\Psi_\e,\hat p_\e)\in V\times Q$ solves 
    \begin{align*}
    a_\e(t;\hat v_\e,\cdot)+b_\e(t;\hat p_\e,\cdot)&= H_\e(t;\cdot)\quad \text{in}\ V^*,\\
    b_\e(t;\cdot, \hat v_\e)&= G_\e(t;\cdot)\quad \text{in}\ Q.
    \end{align*}
   The bilinear form $a_\e$ is coercive and continuous via \cref{lemma:Korn,lemma:transform_estimates}:
    \begin{equation}\label{eq:coercivity_stokes}
    \e^{3-\alpha}\|v\|^2_{V}\leslam a_\e(t;v,v),\quad |a_\e(t;v,w)|\leslam\e^{3-\alpha}\|v\|_V\|w\|_V\qquad (v,w\in V).
    \end{equation}
    Let $q\in Q$ and $v\in V$ and set $\check q=(\Psi_{\e,t})_*q$ and $\check v=(\Psi_{\e,t})_*v$ (see \cref{lemma:pushforward_estimates} where we suppressed the $\e$-dependency).
    We have
    \[
    b_\e(t;q,v)=\int_{\Omega_\e(t)}\check q\dive\check v\di{x}
    \]
    We take the Bogovskiǐ operator $\Bog_\e\colon L^2(\Omega_\e(t))\to H^1_{\Sigma_2}(\Omega_\e(t))$ from \cref{pro:Bog} and use $\Bog_\e f=0$ on $\Gamma_\e$ for all $f\in L^2(\Omega_\e(t))$ to identify $\Bog_\e f=\Bog_\e f_{|_{\Omega_\e(t)}}\in H^1_{\Sigma_2\cup\Gamma_{\e}(t)}$.  
    For any $q\in Q$, we use $\Psi_{\e,t}^*\Bog_\e\check q\in V$ (with $\check q=(\Psi_{\e,t})_*q\in L^2(\Omega_\e(t))$) to estimate
    \[
    \sup_{v\in V}\frac{b_\e(t;q,v)}{\|v\|_V\|q\|_Q}\ge\frac{\|\check q\|_{L^2(\Omega_\e(t))}^2}{\|\Psi_{\e,t}^*\Bog_\e\check q\|_V\|q\|_Q}.
    \]
    We estimate the Bogovskiǐ term with \cref{lemma:pushforward_estimates,pro:Bog}
    \[
    \|\Psi_{\e,t}t^*\Bog_\e\check q\|_V\lesssim_{\lambda}\|\nabla\Bog_\e\check q\|_{L^2(\Omega_\e(t))}\lesssim_{\lambda}\e^{\frac{\alpha-3}{2}}\|\check q\|_{L^2(\Omega_\e(t))}\lesssim_{\lambda}\e^{\frac{\alpha-3}{2}}\| q\|_{L^2(\Omega_\e)},
    \]
    With $\|q\|^2_{L^2(\Omega_\e)}\lesssim_\lambda\|\check q\|_{L^2(\Omega_\e(t))}^2$, we thus deduce the inf-sup condition
    \begin{equation}\label{eq:infsup_condition}
         \e^{\frac{3-\alpha}{2}}\lesssim_\lambda\inf_{q\in Q}\sup_{v\in V}\frac{b_\e(t;q,v)}{\|v\|_V\|q\|_Q}.
    \end{equation}
    Applying standard results on saddle point problems (see, e.g., \cite[Theorems 3.4.1 and 4.2.3]{boffi_mixed_2013}) together with \cref{eq:coercivity_stokes,eq:infsup_condition}, we find that
    \begin{align}
        \|\hat v_\e\|_V&\leslam\e^{\alpha-3}\|H_\e\|_{V^*}
        +\e^{\frac{\alpha-3}{2}}\|{G}_\e\|_Q,\label{eq:saddle_estimate}\\
        \|\hat q_\e\|_Q&\leslam\e^{\frac{\alpha-3}{2}}\|H_\e\|_{V^*}
        +\|{G}_\e\|_Q.
    \end{align}
    and, together with $\hat u_\e=\hat v_\e+\partial_t\Psi_\e$ and \cref{lemma:transform_estimates},
    \[
    \e^{\frac{3-\alpha}{2}}\|\nabla\hat u_\e\|_{L^2(\Omega_\e)}+\|\hat q_\e\|_Q\leslam\e^{\frac{\alpha-3}{2}}\|H_\e\|_{V^*}
        +\|G_\e\|_Q+\e^{\frac{3-\alpha}{2}}\|\nabla\partial_t\Psi_\e\|_{L^2(\Omega_\e)}.
    \]
    The claimed estimates follow (after application of the Poincaré inequality from \cref{lemma:poincare}) when the right hand side is bounded uniformly in $\e$.
   This is easy to see: With \cref{cor:transform_lp_estimates} (used pointwise in time), we have
   \begin{align*}
     \|G_\e\|_Q=\|\partial_tJ_\e\|_{L^2(\Omega_\e)}\lesssim_{\lambda}\e^{\frac{3-\alpha}{2}}\|\partial_tr\|_{\ell^\infty(\indexI)}
   \end{align*}
    and (using again \cref{lemma:transform_estimates,cor:transform_lp_estimates,lemma:poincare})
    \begin{align*}
        \e^{\frac{\alpha-3}{2}}|H_\e(t;w)|
        &\leslam\e^{\frac{\alpha-3}{2}}\|h_\e\|_{L^2(\Omega)}\|w\|_{L^2(\Omega)}
        +\e^{\frac{3-\alpha}2}\|\nabla\partial_t\Psi_\e\|_{L^2(\Omega_\e)}\|w\|_V\\
        &\lesssim_{\lambda}\left(\|h_\e\|_{L^2(\Omega)}+\e^{\alpha}\|\partial_tr\|_{\ell^\infty(\indexI)}\right)\|w\|_{V}.
    \end{align*}
   
   $(ii)$. Similar Lipschitz estimates to \cref{eq:lipschitz_stokes} are established in \cite[Lemma 4.4]{gahn_rigorous_2024}, where the same Stokes problems (with scaling $\alpha=1$ inside a periodic configuration) in reference coordinates is considered, but without any explicit control on the parameter $\e$.
    They establish 
    \[
    \|\delta \hat{v}_\e\|_{L^2(S;H^1(\Omega_\e))}
    \leq L_1(\lambda,M,\e)\|\delta r\|_{L^\infty(S;\ell^2(\mathcal{I}_\e))}+L_2(\lambda,M,\e)\|\delta \partial_tr\|_{L^2(S;\ell^2(\mathcal{I}_\e))}.
    \]
    Since we want a better control on $\e$ in these estimates with respect to $\delta\partial_tr$, we outline the main parts of the argument.\footnote{Specifically, we show that $L_2(\lambda,\Lambda,M,\e)\to0$ for $\e\to0$, which is later used to show that the fixed-point operator is contractive.}
    However, the general structure is largely similar to \cite[Lemma 4.4]{gahn_rigorous_2024}.

    First of all, both solutions satisfy the a priori estimates in $(i)$ and the relative velocities $\hat v_\e^{(j)}=\hat u_\e^{(j)}-\partial_t\Psi_\e^{(j)}\in V$ satisfy \cref{eq:saddle_estimate}.
    In the following, we let $\hat h_\e^{(j)}=(\Psi_{\e,t}^*)^{(j)}h_\e$ denote the pullback of the volume forces to the fixed geometry from the corresponding domain $\Omega_\e^{(j)}(t)=(\Psi_\e^*)^{(j)}(t,\Omega_\e)$ for $j=1,2$.
    We take the same operators as above (with superscript $j=1,2$ for the corresponding radius function $r^{(j)}$).
    Adding the weak forms for $j=1,2$ and rearranging, we see that $(\delta v_\e,\delta q_\e)\in V\times Q$ solves
    \begin{align*}
    a_\e^{(1)}(t;\delta v_\e,\cdot)+b_\e^{(1)}(t;\delta p_\e,\cdot)&=\delta\mathcal{H}_\e(t)\quad \text{in}\ V^*,\\
    b_\e^{(1)}(t;\cdot,\delta v_\e)&=\delta\mathcal{G}_\e(t)\quad \text{in}\ Q.
    \end{align*}
    where $\delta\mathcal{H}_\e(t)\in V^*$ and $\delta\mathcal{G}_\e(t)\in Q$ are defined as
    \[
    \delta\mathcal{H}_\e(t):=\delta H_\e(t;\cdot)-\delta a_\e(t;v_\e^{(2)},\cdot)-\delta b_\e(t;p_\e^{(2)},\cdot),\qquad\delta\mathcal{G}_\e(t):=\delta G_\e(t;\cdot)-\delta b_\e(t;\cdot, v_\e^{(2)}).
    \]
    From $(i)$, we have
    \[
    \e^{\frac{3-\alpha}{2}}\|\delta v_\e\|_V+\|\delta q_\e\|_Q\leslam\e^{\frac{\alpha-3}{2}}\|\delta \mathcal{H}_\e\|_{V^*}
        +\|\delta \mathcal{G}_\e\|_Q.
    \]
    The only thing left is to estimate the source terms $\delta\mathcal{H}_\e$ and $\delta\mathcal{G}_\e$.
    To that end, we note that (using the Lipschitz estimates collected in \cref{lemma:transform_lipschitz})
    \begin{align*}
    |\delta a_\e(t;v,w)|&\leslam \e^{3-\alpha}\|\delta r \|_{\ell^\infty(\mathcal{I}_\e)}\|v\|_{V}\|w\|_{V},  \\
    |\delta b_\e(t;q,w)|&\leslam \|\delta r \|_{\ell^\infty(\mathcal{I}_\e)}\|q\|_Q\|w\|_V
    \end{align*}
    for all $v,w\in V$ and $q\in Q$.
    Moreover,
    \begin{multline*}
    |\delta\mathcal H_\e(t;w)|\le\left(\|\delta J_\e\|_{L^\infty(\Omega_\e)}\|\hat h_\e^{(2)}\|_{L^2(\Omega_\e)}+\|J_\e^{(1)}\|_{L^\infty(\Omega_\e)}\|\delta\hat h_\e\|_{L^2(\Omega_\e)}\right) \|w\|_{L^2(\Omega_\e)}\\
    +|\delta a_\e(t;\partial_t\Psi_\e^{(1)},w)|+|a_\e^{(2)}(t;\partial_t\delta\Psi_\e,w)|.
    \end{multline*}
    This can be estimated with \cref{lemma:transform_estimates,lemma:transform_lipschitz,lemma:pushforward_estimates}:
    \begin{multline*}
    |\delta H_\e(t;w)|\leslam\|\delta r\|_{\ell^\infty(\indexI)}\left(1+\|h_\e\|_{L^2(\Omega_\e)}+\e^\alpha\|\nabla h_\e\|_{L^2(\Omega_\e)}\right)\|w\|_{L^2(\Omega_\e)}\\
    +\e^{3-\alpha}\left(\|\nabla\partial_t\Psi_\e^{(1)}\|_{L^2(\Omega)}\|\delta r \|_{\ell^\infty(\mathcal{I}_\e)}
    +\|\delta\nabla\partial_t\Psi_\e\|_{L^2(\Omega)}\right)\|w\|_{V}
    \end{multline*}
    Since the transformation is localized on $\mathcal U_\e^\lambda$ (with $|\mathcal U_\e^\lambda|\lesssim_{\lambda}\e^{3(\alpha-1)}$), we can simplify
    \[
    \|\nabla\partial_t\Psi_\e^{(1)}\|_{L^2(\Omega_\e)}\lesssim_{\lambda} \e^{\frac32(\alpha-1)},\qquad  \|\delta\partial_t\Psi_\e\|_V\lesssim_{\lambda} \e^{\frac32(\alpha-1)}\|\delta\partial_tr \|_{\ell^\infty(\mathcal{I}_\e)}.
    \]
    Since $\sup_\e\e^\alpha\|\nabla h_\e\|_{L^2(\Omega)}$ is finite by Assumption~\ref{A2}, we can further estimate with \cref{lemma:poincare,lemma:pushforward_estimates}
    \begin{equation}\label{eq:estimate_H}
    \|\delta H_\e(t)\|_{V^*}\leslam\e^{\frac{3-\alpha}{2}}\|\delta r\|_{\ell^\infty(\indexI)}
    +\e^{\frac12(3+\alpha)}\left(\|\delta r \|_{\ell^\infty(\mathcal{I}_\e)}
    +\|\delta\partial_tr \|_{\ell^\infty(\mathcal{I}_\e)}\right).
    \end{equation}
    In addition, we find that (using the a priori estimates \cref{eq:apriori_stokes_reference})
    \begin{align}
       |\delta a_\e(t;v_\e^{(2)},w)|&\leslam\e^{\frac{3-\alpha}{2}}\|\delta r\|_{\ell^\infty(\mathcal{I}_\e)}\|w\|_{V},\label{eq:estimate_deltaa}\\
       |\delta b_\e(t;p_\e^{(2)},w)|&\leslam \|\delta r\|_{\ell^\infty(\mathcal{I}_\e)}\|w\|_V\label{eq:estimate_deltab}
    \end{align}
    Collecting the estimates \cref{eq:estimate_H,eq:estimate_deltaa,eq:estimate_deltab}, we are led to
    \begin{align*}
    \e^{\frac{\alpha-3}{2}}|\delta \mathcal{H}_\e(t,w)|
    &\lesssim_{\lambda,M}\|\delta r\|_{\ell^\infty(\indexI)}
    +\e^{\frac32(\alpha-1)}\left(\|\delta r\|_{\ell^\infty(\mathcal{I}_\e)}
    +\|\delta\partial_tr\|_{\ell^\infty(\mathcal{I}_\e)}\right).
    \end{align*}
    Finally, we estimate $\delta\mathcal{G}_\e$:
    \[
    \|\delta\mathcal{G}_\e\|_Q\le\|\delta \partial_tJ_\e\|_{L^2(\Omega_\e)}+\|\delta (J_\e F_\e^{-1})\|_{L^\infty(\Omega_\e)}\|v_\e^{(2)}\|_{V}
    \]
    which, via the a priori estimates for $v_\e^{(2)}$ and the available Lipschitz estimates for $J_\e$ and $F_\e$, leads to
    \begin{align*}
    \|\delta\mathcal{G}_\e\|_Q\lesssim_{\lambda,M} \e^{\frac32(\alpha-1)}\|\delta\partial_tr\|_{L^\infty(\Omega_\e)}+(\e^{\frac{\alpha-3}{2}}+\e^{2\alpha-3})\|\delta r\|_{L^\infty(\Omega_\e)}.
    \end{align*}
    In summary, we arrive at the claimed estimates
    \end{proof}

\begin{remark}\label{remark:lipschitz_stokes_refined}
    Note that this Lipschitz estimate blows up when $\e\to0$ for all $\alpha<3$ and stays uniformly bounded only in the critical case $\alpha=3$.
    Crucially, however, for all $\alpha>1$, the component with the difference of the time derivatives of the evolution vanishes as $\e\to0$.
    Since $\delta r(0)=0$, it holds
    \[
    \e^{\frac{3-\alpha}{2}}\|\delta \nabla v_\e\|_{L^2(\Omega_\e)}
    +\|\delta p_\e\|_{L^2(\Omega_\e)}
    \lesssim_{\lambda, M}\left(\e^{\frac{\alpha-3}{2}}t+\e^{\frac32(\alpha-1)}\right)\|\delta \partial_tr\|_{\ell^\infty(\indexI)},
    \]
    for almost all $t\in S$.
    Hence, the Lipschitz constant can be made arbitrarily small for small times: for all $q>0$ there are $\e>0$ and $t_\e>0$ such that
    \begin{align}\label{eq:Lipschitz_explicit_stokes}
    \e^{\frac{\alpha-3}{2}}\|\delta \nabla v_\e\|_{L^2(\Omega_\e)}
    +\|\delta p_\e\|_{L^2(\Omega_\e)}
    \le q\|\delta \partial_tr\|_{\ell^\infty(\indexI)}
    \end{align}
    for almost all $t\in(0,t_\e)$ and for all $r^{(1)},r^{(2)}\in\mathcal{R}_{\e}(S)$ with $\nicefrac1\lambda\le r^{(j)}\le\lambda$ and $|\partial_tr^{(j)}|\le M$.
    The time interval $(0,t_\e)$ is quite small (with $t_\e\to0$ as $\e\to0)$ for all $\alpha<3$), but this is enough to establish a contraction for the coupled problem; long-time existence then follows via a blow-up argument.
\end{remark}

\subsection{Analysis of the diffusion-reaction system}\label{ssec:stokes_chemistry_fixed}
In this section, we investigate the chemistry part of the problem for a given radius evolution given by \cref{eq:micro_system:a_fixed}.
Again, we chose $\lambda\ge\overline r$, $M>0$, and $S=(0,T)$ for $T>0$ and only take $\e\le\e_\lambda$.

\begin{lemma}
\label{lemma:existence_given_r}
    For any prescribed $r\in\mathcal{R}_{\e}(S)$ with $\nicefrac1\lambda\le r_i\le\lambda$, there is a global, unique solution $(\hat{c}_\e,\hat{u}_\e,\hat p_\e)$ for the coupled system \ref{eq:micro_system_fixed} in the sense of \cref{definition:weak_solution_fixed} with the radii evolution given via $(\e^\alpha r_i(t))_{i\in\indexI}$.
    Moreover, the solution satisfies the energy estimate
    \begin{equation}\label{eq:apriori_chemistry}
    \|\hat{c}_\e\|^2_{L^\infty(0,t;L^2(\Omega_{\e}))}
        +\|\nabla\hat{c}_\e\|_{L^2(0,t\times\Omega_\e)}^2
    \lesssim_{t,\lambda} 
    1+\|c_{\e,0}\|_{L^2(\Omega_\e)}^2+t\sup_{\gamma\in[\nicefrac1\lambda,\lambda]}\left(g(\gamma)\left[\gamma\cequ(\gamma)\right]^2\right).
    \end{equation}
    for almost all $t\in S$
    and it is bounded via
    \[
    0\le \hat c_\e(t,x)\le \max\{\overline c_0,\cequ(\lambda)\}e^{2C_F t}.
    \]
\end{lemma}
\begin{proof}
    $(i).$ \textit{Existence and uniqueness of solutions}.
    By first solving the linear Stokes system via \cref{lemma:stokes_for_fixed_radius} to get $(\hat{u}_\e,\hat p_\e)$ and then the semi-linear diffusion problem using $\hat{u}_\e$ in the advection term, we find that there is a unique solution as long as the balls stay well separated and are not collapsing.\footnote{We can use \cite[Chapter III, Propositions 3.2 $\&$ 3.3]{showalter_monotone_2014} together with the Lipschitz continuity of the right hand sides.}
    However, the balls cannot collapse due to $0<\frac1\lambda\le r_i(t)$ and they are kept separated via $r_i\le\lambda$ as long as $\e\le\e_\lambda$ (see \cref{eq:eps_lambda}).

    $(ii).$ \textit{Estimates}.
    We already have the following energy estimates with respect to the moving geometry via \cref{lemma:apriori_prescribed_radius}:
    \begin{align*}
    \|c_\e\|_{L^\infty(0,t;L^2(\Omega_\e(t)))}^2
    +\|\nabla c_\e\|_{L^2(0,t;L^2(\Omega_{\e}(t)))}^2
    \lesssim_t 
    1+\|c_{\e,0}\|_{L^2(\Omega_\e)}^2+t\sup_{\gamma\in[\nicefrac1\lambda,\lambda]}\left(g(\gamma)\left[\gamma\cequ(\gamma)\right]^2\right)
    \end{align*}
    which imply the claimed estimate via \cref{lemma:pushforward_estimates}.
    Finally, the $L^\infty$-bounds are also shown in \cref{lemma:apriori_prescribed_radius}.

\end{proof}

\begin{lemma}\label{lemma:time_derivative_estimate}
    Let $r\in\mathcal{R}_{\e}(S)$ with $\nicefrac1\lambda\le r_i(t)\le\lambda$.
    The corresponding solution $\hat{c}_\e$ due to \cref{lemma:existence_given_r} satisfies
    \[
    \|\partial_t(J_\e\hat{c}_\e)\|_{L^2(S;H^1(\Omega_\e))^*}
    +\e^\alpha\|\partial_t\hat{c}_\e\|_{L^2(S;H^1(\Omega_\e))^*}\leslam1.
    \]
\end{lemma}
\begin{proof}
    We chose as test function $\varphi\in H^1(\Omega_\e)$ with $\|\varphi\|_{H^1(\Omega_\e)}\leq1$ for \eqref{eq:micro_system:a_fixed}:
    \begin{equation*}
    \langle\partial_t(J_{\e}\hat{c}_\e),\varphi\rangle_{H^1(\Omega_\e)}=-(D_{\e}\nabla\hat{c}_\e-A_{\e}\hat{c}_\e\hat{v}_\e,\nabla\varphi)_{L^2(\Omega_{\e})}
    -\e^{3-2\alpha}\left(J_\e^\Gamma G\left(\hat c_\e,r\right),\varphi\right)_{L^2(\Gamma_{\e})}
    +(J_{\e}F(\hat{c}_\e),\varphi)_{L^2(\Omega_{\e})}
    \end{equation*}
    We can directly estimate the bulk terms on the right hand side
    \begin{align*}
       -(D_{\e}\nabla\hat{c}_\e-A_{\e}\hat{c}_\e\hat{v}_\e,\nabla\varphi)_{L^2(\Omega_{\e})}
       &\leq \|D_{\e}\|_{L^\infty(\Omega_\e)}\|\nabla\hat{c}_\e\|_{L^2(\Omega_{\e})}+\|A_\e\|_{L^\infty(\Omega_\e)}\|\hat{c}_\e\|_{L^\infty(\Omega_\e)}\|\hat v_\e\|_{L^2(\Omega_{\e})},\\
      (J_{\e}F(\hat{c}_\e),\varphi)_{L^2(\Omega_{\e})}
      &\le \|J_{\e}\|_{L^\infty(\Omega_\e)}\|F(\hat{c}_\e)\|_{L^2(\Omega_{\e})}.
    \end{align*}
    For the interface term, we use \cref{lemma:trace_estimates} to estimate
    \begin{align*}
      -\e^{3-2\alpha}\left(J_\e^\Gamma G\left(\hat c_\e,r\right),\varphi\right)_{L^2(\Gamma_{\e})}
      &\le C\|J_\e^\Gamma\|_{L^\infty(S)}\|G\left(\hat c_\e,r\right)\|_{L^\infty(\Omega_\e)}.
    \end{align*}
    With the estimates for $D_\e, J_\e, A_\e$ from Lemma \ref{lemma:transform_estimates}, for $\hat v_\e$ from Lemma \ref{lemma:stokes_for_fixed_radius}, and for $\|\hat c_\e\|_\infty$ from \cref{lemma:existence_given_r}, we get the desired result for $\partial_t(J_\e \hat c_\e)$.

    With the product rule, we can characterize $\partial_t\hat c_\e$ via
    \[
    \langle\partial_t\hat c_\e,\varphi\rangle_{H^1(\Omega_\e)}
    =\langle\partial_t(J_\e\hat c_\e),J_\e^{-1}\varphi\rangle_{H^1(\Omega_\e)}
    -(\hat c_\e\partial_tJ_\e,J_\e^{-1}\varphi)_{L^2(\Omega_\e)}
    \]
    which we can estimate by
    \begin{multline*}
    \left| \langle\partial_t\hat c_\e,\varphi\rangle_{H^1(\Omega_\e)}\right|
    \\
    \leq \|J_\e^{-1}\|_{W^{1,\infty}(\Omega_\e)}\|\partial_t(J_\e\hat c_\e)\|_{H^1(\Omega_\e)^*}\|\varphi\|_{H^1(\Omega_\e)}+\|\hat c_\e\|_{L^\infty(\Omega_\e)}\|J_\e^{-1}\|_{L^2(\Omega_\e)}\|\partial_tJ_\e\|_{L^\infty(\Omega_\e)}\|\varphi\|_{L^2(\Omega_\e)}.
    \end{multline*}
    Utilizing the estimates in \cref{lemma:transform_estimates}, the estimate for $\partial_t\hat c_\e$ follows.
\end{proof}

\begin{lemma}[Lipschitz estimates]
\label{lemma:solution_lipschitz_estimates}
For $j=1,2$, let $r^{(j)}\in\mathcal{R}_{\e}(S)$ with $\nicefrac1\lambda\le r_i^{(j)}\le\lambda$ and $|\partial_tr^{(j)}|\le M$.
Let $(\hat{c}^{(j)}_\e,\hat{v}^{(j)}_\e,\hat{p}^{(j)}_\e)$ denote the corresponding solutions due to \cref{lemma:stokes_for_fixed_radius,lemma:existence_given_r}.
For every $q>0$ there is a time horizon $S_\e=(0,t_\e)\subset S$ such that 
    \begin{equation*}
    \|\delta \hat{c}_\e\|_{L^\infty(S_\e;L^2(\Omega_\e))}^2+\|\delta\nabla \hat{c}_\e\|_{L^2(S_\e\times\Omega_\e)}^2
    \leq q\|\delta \partial_tr\|^2_{L^2(S_\e;\ell^\infty(\mathcal{I}_\e))}.
    \end{equation*}
\end{lemma}

\begin{proof}
$(i)$ \textit{Lipschitz energy estimates}.
We look at the difference of the weak forms (we use $J^{(j)}, D^{(j)}$ and $\delta{J},\delta{D}$ etc.~to denote the coefficients and their differences)\footnote{For the ease of notation, we drop the $\e$-subscripts in the proof.} and test with $\varphi=\delta\hat{c}$:
    \begin{multline*}
    \langle\partial_t(J^{(1)}\delta\hat{c}),\delta\hat{c}\rangle_{H^1(\Omega_\e)}
    +\langle\partial_t(\delta{J}\hat{c}^{(2)}),\delta\hat{c}\rangle_{H^1(\Omega_\e)}
   \\
    +(D^{(1)}\delta\nabla\hat{c}-A^{(1)}\delta\hat{c}\hat{v}^{(1)},\delta\nabla\hat{c})_{L^2(\Omega_{\e})}
    +(\delta{D}\nabla\hat{c}^{(2)}-\delta{A}\hat{c}^{(2)}\hat{v}^{(2)}-A^{(2)}\hat{c}^{(2)}\delta\hat{v},\delta\nabla\hat{c})_{L^2(\Omega_{\e})}\\
    =\e^{3-2\alpha}\left(J^{\Gamma,(1)} \left[G\left(\hat c^{(1)},r^{(1)}\right)-G\left(\hat c^{(2)},r^{(2)}\right)\right]+\delta J^{\Gamma} G\left(\hat c^{(2)},r^{(2)}\right),\delta\hat{c}\right)_{L^2(\Gamma_{\e})}\\
    +(\delta{J}F(\hat{c}^{(1)})+J^{(2)}\left(F(\hat{c}^{(1)})-F(\hat{c}^{(2)})\right),\delta\hat{c})_{L^2(\Omega_{\e})}.
    \end{multline*}
Using the transformed divergence relation for the relative velocity and the identity
\[
    \partial_t J^{(1)}
    =
    \dive\left(A^{(1)}\partial_t\Psi^{(1)}\right),
\]
we obtain
\[
\langle\partial_t(J^{(1)}\delta\hat{c}),\delta\hat{c}\rangle_{H^1(\Omega_\e)^*}
-(A^{(1)}\delta\hat{c}\hat{v}^{(1)},\delta\nabla\hat{c})_{L^2(\Omega_{\e})}
=\frac{1}{2}\ddt\left(J^{(1)}\delta\hat{c},\delta\hat{c}\right)_{L^2(\Omega_\e)}
\]
which leads to
\begin{multline*}
    \frac12\ddt\left(J^{(1)}\delta\hat{c},\delta\hat{c}\right)_{L^2(\Omega_\e)}
    +(D^{(1)}\delta\nabla\hat{c},\delta\nabla\hat{c})^2_{L^2(\Omega_{\e})}\\
    \le \underbrace{\e^{3-2\alpha}\left(J^{\Gamma,(1)} \left[G\left(\hat c^{(1)},r^{(1)}\right)-G\left(\hat c^{(2)},r^{(2)}\right)\right]+\delta J^{\Gamma} G\left(\hat c^{(2)},r^{(1)}\right),\delta\hat{c}\right)_{L^2(\Gamma_{\e})}}_{I_1}\\
    +\underbrace{(\delta{J}F(\hat{c}^{(1)})+J^{(2)}\left(F(\hat{c}^{(1)})-F(\hat{c}^{(2)})\right),\delta\hat{c})_{L^2(\Omega_{\e})}}_{I_2}
    \underbrace{-\langle\partial_t(\delta{J}\hat{c}^{(2)}),\delta\hat{c}\rangle_{H^1(\Omega_\e)}}_{I_3}\\
    \underbrace{-(\delta D\nabla\hat{c}^{(2)}-\delta{A}\hat{c}^{(2)}\hat{v}^{(2)}-A^{(2)}\hat{c}^{(2)}\delta\hat{v},\delta\nabla\hat{c})_{L^2(\Omega_{\e})}}_{I_4}.
\end{multline*}
We go on by looking at the individual terms $I_1,\dots,I_4$ and estimate them separately:
Starting with $I_1$, we recall that $G(c,r)=g(r)(\cequ(r)-c)$ and
\begin{multline*}
G\left(\hat c^{(1)},r^{(1)}\right)-G\left(\hat c^{(2)},r^{(2)}\right)\\
=
g(r^{(1)})\left(\cequ(r^{(1)})-\cequ(r^{(2)})-\delta\hat c)\right)
+\left(g(r^{(1)})-g(r^{(2)})\right)(\cequ(r^{(1)})-\hat c^{(1)}).
\end{multline*}
With $g$ non-negative and $g,\cequ$ locally Lipschitz,  the bound on $\delta J^\Gamma$ from  Lemma \ref{lemma:transform_lipschitz},  $r^{(j)}\le \lambda$ and $0\le c^{(j)}\le \{\overline c_0,\cequ(\lambda)\}e^{2C_F t}$ (\cref{lemma:existence_given_r}), we estimate using Cauchy-Schwarz
\begin{align*}
I_1&\lesssim_{\lambda,T}\e^{\frac{3-2\alpha}2}\|\delta r\|_{\ell^\infty(\indexI)}\|\delta\hat{c}\|_{L^2(\Gamma_\e)}
\end{align*}
The volume source terms in $I_3$ can be estimated directly using \cref{lemma:transform_lipschitz} and the local Lipschitz regularity of $F$:
\[
I_2\lesssim_{\lambda,T}\e^{\frac{3(\alpha-1)}{2}}\|\delta r\|_{\ell^\infty(\indexI)}\|\delta\hat{c}\|_{L^2(\Omega_\e)}+\|\delta\hat{c}\|_{L^2(\Omega_\e)}^2
\]
For the time derivative term $I_3$, we write
\[
I_3\leq \|\delta\partial_t J\|_{L^2(\Omega_\e)}\|\hat{c}^{(2)}\|_{L^\infty(\Omega_\e)}\|\delta\hat{c}\|_{L^2(\Omega_\e)}+\|\delta J\|_{L^\infty(\Omega_\e)}\|\partial_t\hat c^{(2)}\|_{H^1(\Omega_\e)^*}\|\delta \hat c\|_{H^1(\Omega_\e)}
\]
This we can estimate using \cref{lemma:transform_lipschitz,lemma:existence_given_r,lemma:time_derivative_estimate,cor:transform_lp_estimates} with
\[
I_3\lesssim_{\lambda,T,M}\e^{\frac{3(\alpha-1)}2}\|\delta\partial_t r\|_{\ell^\infty(\indexI)}\|\delta\hat{c}\|_{L^2(\Omega_\e)}+\e^{-\alpha}\|\delta r\|_{\ell^\infty(\indexI)} \|\delta\hat c\|_{H^1(\Omega_\e)}.
\]
For the last term, $I_4$, we have (using the a priori estimates available for $\hat{c}^{(2)}$ and $\hat{v}^{(2)}$)
\[
I_4\lesssim_{\lambda,T}\|\delta r\|_{\ell^\infty(\indexI)}\|\delta\nabla\hat{c}\|_{L^2(\Omega_{\e})} \left(1+\e^{\frac{3}{2}(\alpha-1)}\right)
+\|\delta\hat{v}\|_{L^2(\Omega_{\e})}\|\delta\nabla\hat{c}\|_{L^2(\Omega_{\e})}.
\]
Summarizing everything we have so far while also
\begin{itemize}
    \item using the trace inequality on $\|\delta\hat{c}\|^2_{L^2(\Gamma_\e)}$ and generalized Young's inequality to subsume the gradient terms on the left hand side,
    \item estimating $\|\delta r\|_{\ell^\infty(\indexI)}$ against $t\|\delta \partial_tr\|_{\ell^\infty(\indexI)}$,
    \item using the Lipschitz estimate for $\|\delta\hat v\|_{L^2(\Omega_\e)}$ from \cref{lemma:stokes_for_fixed_radius,remark:lipschitz_stokes_refined},
\end{itemize}
we are led to
\begin{equation*}
     \ddt\left(J^{(1)}\delta\hat{c},\delta\hat{c}\right)_{L^2(\Omega_\e)}
    +\|\delta\nabla\hat{c}\|^2_{L^2(\Omega_{\e})}\\
    \lesssim_{\lambda,T,M,\e^{-1}}\|\delta\hat{c}\|^2_{L^2(\Omega_\e)}
    +t\|\delta\partial_tr\|_{\ell^\infty(\mathcal{I}_\e)}^2.
\end{equation*}
With Grönwall's inequality, and using the lower bound on $J^{(1)}$ from Lemma \ref{lemma:transform_estimates}, the statement follows for sufficiently small time intervals.
\end{proof}

\subsection{Fixed-point argument}\label{ssec:fixed_point_argument}
In this section, we finally put together the fixed-point argument to show that there is a unique solution in the sense of \cref{definition:weak_solution_moving}.
Given a prescribed radius evolution $r$, we first solve the Stokes–diffusion system to obtain $\hat c_\e$, and then solve the radius ODEs driven by this concentration.
For $\lambda>\overline r$, we set $k_\lambda=\max\{\overline c_0,\cequ(\lambda)\}$ and $g_\lambda=\max\{1,\sup_{\gamma\in[\nicefrac1\lambda,\lambda]}g(\gamma)\}$.\footnote{Here, the maximum is introduced only to guarantee $g_\lambda>0$.}:
\begin{itemize}
    \item \Cref{lemma:concentration_to_radius}: For every \(k>0\) and
    \(c\in  W(S;\Omega_\e)\) with
    \(0\le c(t,x)\le k e^{2C_Ft}\), there exists a unique
    \(r_\e\in\mathcal R_\e(S)\) solving the ODEs for the
    radii evolution. Moreover,
    \[
    0<\underline\xi(t)\le r_{\e,i}(t)\le \overline\xi_k(t),
    \]
    where
    \[
    \partial_t\underline\xi
    =
    -g(\underline\xi)c_{\mathrm{eq}}(\underline\xi), \quad \underline \xi(0)=\frac1{\overline r},
    \qquad \qquad
    \partial_t\overline\xi_k
    =g(\overline\xi_k)e^{2C_Ft}, \quad \overline\xi_k(0)=\overline r.
    \]
    In particular, if \(r_{\e,i}(t)\in[1/\lambda,\lambda]\), then
    \[
    \|\partial_t r_\e(t)\|_{\ell^\infty(I_\e)}
    \le g_\lambda
    \left(
    k e^{2C_Ft}+c_{\mathrm{eq}}(\lambda)
    \right).
    \]
    \item\Cref{lemma:existence_given_r}: For every prescribed radial evolution $r\in\mathcal{R}_\e(S)$ with $\nicefrac1\lambda\le r_i(t)\le\lambda$, there is a unique solution $(\hat c_\e,\hat u_\e, \hat p_\e)$ with $0\le\hat c_\e\le k_\lambda e^{2C_Ft}$.
\end{itemize}
Since \(\lambda>\overline r\) and \(\underline\xi(0)=1/\overline r>1/\lambda\), we may choose
    \[
    T_\lambda:=\min\{T_\lambda^-,T_\lambda^+\}>0,
    \]
    where
    \[
    T_\lambda^-:=
    \sup\left\{
    t>0:\ \underline\xi(s)\ge \frac1\lambda
    \text{ for all }s\in[0,t]
    \right\},\qquad
    T_\lambda^+:=
    \sup\left\{
    t>0:\ \overline\xi_{k_\lambda}(s)\le\lambda
    \text{ for all }s\in[0,t]
    \right\}.
    \]
    %
    Next choose \(M_\lambda>0\) such that
    \[
    M_\lambda>g_\lambda
    \left(
    k_\lambda e^{2C_FT_\lambda}+\cequ(\lambda)
    \right)
    \]
    and define
    \[
    A_{\lambda}(T_\lambda)
    :=
    \left\{
    r\in\mathcal R_\e(0,T_\lambda):
    \frac1\lambda\le r_i(t)\le\lambda,\ 
    \|\partial_t r\|_{L^\infty(0,T_\lambda;\ell^\infty(\indexI))}
    \le M_\lambda
    \right\}.
    \]
    For \(r\in A_\lambda(T_\lambda)\), \cref{lemma:existence_given_r} gives a unique solution \((\hat c_\e,\hat u_\e,\hat p_\e)\) with
    \[
    0\le \hat c_\e(t,x)\le k_\lambda e^{2C_Ft}.
    \]
    Applying \cref{lemma:concentration_to_radius} with \(k=k_\lambda\), we obtain a unique radius evolution \(R=:\mathcal L_\e(r)\) satisfying
    \[
    \underline\xi(t)\le R_i(t)\le \overline\xi_{k_\lambda}(t).
    \]
    By the choice of \(T_\lambda^-\) and \(T_\lambda^+\), it follows that
    \[
    \frac1\lambda\le R_i(t)\le \lambda
    \qquad\text{for all }t\in[0,T_\lambda].
    \]
    Moreover,
    \[
    \|\partial_t R(t)\|_{\ell^\infty(\indexI)}
    \le g_\lambda
    \left(
    k_\lambda e^{2C_Ft}+\cequ(\lambda)
    \right)
    \le M_\lambda .
    \]
    Hence
    \[
    \mathcal L_\e\colon A_\lambda(T_\lambda)\to A_\lambda(T_\lambda)
    \]
    is well-defined.
    Note that $A_\lambda(T_\lambda)$ is a closed subset of 
\[
W^{1,2}(0,T_\lambda;\ell^\infty(\indexI))
\]
and therefore a complete metric space with the metric (note that $r^{(1)}(0)=r^{(2)}(0)=r_{\e,0}$ since $\mathcal{R}_\e(S)$ enforces the initial conditions)
\[
d(r^{(1)},r^{(2)})=\|\delta\partial_tr\|_{L^2(0,T_\lambda;\ell^\infty(\indexI))}.
\]
%

\begin{theorem}[Fixed-point] \label{thm:well-posed}
    For every $T>0$, there is $\e_T$ such that for all $\e\le\e_T$  there is a unique fixed-point $r_\e\in\mathcal{R}_\e(S)$ with $\mathcal{L}_\e(r_\e)=r_\e$.
\end{theorem}

\begin{proof}
\textit{Step 1: Small-time existence.}
Let \(0<t\le T_\lambda\) and take \(r^{(j)}\in A_\lambda(t)\), \(j=1,2\).
We set
\[
R_\e^{(j)}:=\mathcal L_\e(r^{(j)}),
\]
and let $\hat c^{(j)}_\e$ be the concentration solution to input $r^{(j)}$ via \cref{lemma:existence_given_r}.
By \cref{lemma:concentration_to_radius}, we have
\[
\|\delta\partial_t R\|_{L^2(0,t;\ell^\infty(\indexI))}
\lesssim_{\e^{-1},t}
\|\delta\hat c_\e\|_{L^2(0,t;H^1(\Omega_\e))}.
\]
By \cref{lemma:solution_lipschitz_estimates}, for every $q>0$, there exists
\(t_\e\in(0,T_\lambda]\) such that, for all \(0<t\le t_\e\),
\[
\|\delta\hat c_\e\|_{L^2(0,t;H^1(\Omega_\e))}
\le
q\|\delta\partial_t r\|_{L^2(0,t;\ell^\infty(\indexI))}.
\]
Choosing \(q>0\) sufficiently small and then \(t_\e>0\) accordingly, we obtain
\[
\|\delta\partial_t R\|_{L^2(0,t_\e;\ell^\infty(\indexI))}
\le
\frac12
\|\delta\partial_t r\|_{L^2(0,t_\e;\ell^\infty(\indexI))}.
\]
Hence \(\mathcal L_\e\) is a contraction on \(A_\lambda(t_\e)\) with respect to the metric $d(r^{(1)},r^{(2)})$.
Since \(A_\lambda(t_\e)\) is complete, Banach's fixed-point theorem yields a unique fixed point
\(r_\e\in A_\lambda(t_\e)\).
\\[2mm]
\textit{Step 2: Continuation up to \(T_\lambda\).}
Let \(t_\e^*\le T_\lambda\) be the maximal existence time of the fixed point
constructed in Step 1. Suppose, for contradiction, that \(t_\e^*<T_\lambda\).
By the self-mapping bounds from above, the corresponding solution satisfies
\[
\frac1\lambda\le r_{\e,i}(t)\le \lambda
\qquad\text{for all }t\in[0,t_\e^*),\ i\in\indexI,
\]
and
\[
\|\partial_t r_\e(t)\|_{\ell^\infty(\indexI)}
\le M_\lambda
\qquad\text{for a.e. }t\in(0,t_\e^*).
\]
Hence \(r_\e\) is uniformly Lipschitz in time with values in
\(\ell^\infty(\indexI)\). Therefore the limit
\[
r_\e(t_\e^*)
:=
\lim_{t\uparrow t_\e^*} r_\e(t)
\]
exists in \(\ell^\infty(\indexI)\), and
\(
\frac1\lambda\le r_{\e,i}(t_\e^*)\le \lambda.
\)
Moreover, since \(t_\e^*<T_\lambda\), there exists \(\eta>0\) such that
\[
\underline\xi(t)\ge \frac1\lambda,
\qquad
\overline\xi_{k_\lambda}(t)\le \lambda
\qquad\text{for all }t\in[0,t_\e^*+\eta]\subset[0,T_\lambda].
\]
Using \(r_\e(t_\e^*)\) as new initial data, the small-time fixed-point
argument from Step 1 can be restarted on \((t_\e^*,t_\e^*+\tau)\)
for some \(0<\tau\le\eta\).\footnote{Of course, \(\mathcal R_\e\) has to be shifted to the new initial conditions. This causes no difficulty, since \(r_\e(t_\e^*)\in[1/\lambda,\lambda]^{I_\e}\), and the separation condition remains valid by the choice \(\e\le\e_\lambda\).}
This extends the fixed point beyond \(t_\e^*\),
contradicting the maximality of \(t_\e^*\). Hence \(t_\e^*=T_\lambda\).
\\[2mm]
\textit{Step 3: Extension to an arbitrary finite time horizon.}
Let  \(T>0\) be arbitrary. By \cref{theorem:maximum_principle}, the a priori
bounds for the radii and the concentration hold on every finite time interval.
In particular, there exists \(\lambda_T>1\) such that
\(
\frac1{\lambda_T}\le r_{\e,i}(t)\le \lambda_T
\)
and there is \(C_T>0\) such that
\(
0\le c_\e(t,x)\le C_T.
\)
Consequently, with
\[
g_{\lambda_T}:=\sup_{\gamma\in[1/\lambda_T,\lambda_T]}g(\gamma),
\]
the radius equation yields
\[
\|\partial_t r_\e(t)\|_{\ell^\infty(\indexI)}
\le
g_{\lambda_T}
\left(C_T+\cequ(\lambda_T)\right)
=:M_T.
\]
Thus, at every time \(s<T\) reached by the solution, the admissibility constants
\(\lambda_T\) and \(M_T\) remain finite. Repeating the continuation argument from
Step 2, with \(\lambda_T\) and \(M_T\) in place of \(\lambda\) and \(M_\lambda\),
shows that the solution cannot have a maximal existence time strictly smaller
than \(T\). Hence the fixed point exists on the whole interval \((0,T)\) as long as $\e\le\e_{\lambda_T}$.
\end{proof}

\begin{remark}
    \begin{itemize}
    \item This fixed point indeed gives a solution of the
    fully coupled problem in the sense of \cref{definition:weak_solution_fixed}.
    Let \(r_\e\) be the fixed point constructed above, and let
    \(
    (\hat u_\e,\hat p_\e,\hat c_\e)
    \)
    be the solution of the prescribed-radius problem corresponding to
    \(r_\e\).
    By definition of the map \(\mathcal L_\varepsilon\), the
    radius \(R_\varepsilon:=\mathcal L_\varepsilon(r_\varepsilon)\) is the unique solution of the radius equation driven by \(\hat c_\varepsilon\). Since \(r_\varepsilon\) is a fixed point, we have \(
    R_\varepsilon=r_\varepsilon.
    \)
    Consequently, the radii determining the domain, the velocity field and the
    concentration equation coincide with the radii generated by the growth law.
    \item There can be no other solutions in the sense of  \cref{definition:weak_solution_fixed}. If \(
    (\hat u_\e,\hat p_\e,\hat c_\e, r_\e)
    \) is a solution, then the triple $ (\hat u_\e,\hat p_\e,\hat c_\e)$ uniquely solve the prescribed radius evolution for $r_\e$.
    Moreover, \(r_\e\) uniquely solves the radius equation driven by \(\hat c_\e\), with the prescribed initial data. It follows that
    \[
    \mathcal L_\e(r_\e)=r_\e.
    \]
    \end{itemize}
    \end{remark}

\section{Passage to the limit in the critical case \texorpdfstring{$\alpha = 3$}{alpha equals 3}}
\label{sec:hom_crit}

In this section, we prove the quantitative homogenization result in the critical regime $\alpha=3$, stated in \cref{th:hom.critical}.
We  do not pass to the limit directly in the weak formulation. Instead, we use a relative energy argument to simultaneously  prove quantitative convergence estimates for the convergence of the solutions $c_\e$, $u_\e$, and the empirical density $f_\e$ (as defined in \eqref{eq:f_vareps}) 
This in particular allows us to work in the moving geometry throughout this section.

Throughout this section, we fix $T>0$ and consider $\e < \e_T$ where $\e_T$ is sufficiently small such that \cref{theorem:well_posedness} guarantees the existence of a unique solution on $[0,T)$ to \eqref{system:stokes_eps}--\eqref{eq:mass_depos} which satisfies $B_{\e,i}(t) \subset  B_{\frac {\dmin\e}{4}}(z_{\e,i})$ on $[0,T)$.
Even though $\alpha =3$ is fixed in the whole section, we retain the parameter $\alpha$ in the equations and estimates whenever they will be reused in the next section.
On the other hand, in equations and estimates that are specific to the critical scaling, we directly replace $\alpha$ by $3$.
For most parts of the proof, we fix a time $t \in [0,T)$ and to lean the notation, we will often not make the time dependence explicit. For instance, we write $c_\e(x)$ and $\Omega_\e$ instead of $c_\e(t,x)$ and $\Omega_\e(t)$.

\subsection{The relative energy argument}

\subsubsection*{Approximation for the concentration}
The macroscopic concentration $c$ does not satisfy the microscopic exchange conditions on the particle boundaries.
In the critical scaling, the resulting discrepancy is of order one and therefore cannot be neglected.
We thus introduce local correctors around each particle which approximate the corresponding boundary profile.
To that end, we consider the approximate density
\begin{align} \label{c^app}
    c_\e^{app} := c + \sum_{i\in\indexI} c_{\e,i}^{app},
\end{align}
where the correctors $c_{\e,i}^{app}$ are defined to approximately match the boundary conditions of $c_\e$. 
More precisely, we recall the definition of $H$ from \eqref{H}, introduce the short hand notation $$H_{\e,i}:=H(c(z_{\e,i}),r_{\e,i})$$  and define (suppressing all time dependencies)
\begin{align} \label{corrector.c}
    c_{\e,i}^{app} (x) := \begin{cases}H_{\e,i}\left(1 - \frac{4 \e^2 r_{\e,i}}{\dmin} \right) &\quad \text{in } B_{\e,i},\\
     H_{\e,i} r_{\e,i} \left(\frac{\e^3 }{|x- z_{\e,i}|} - \frac{4 \e^2}{\dmin} \right)&\quad \text{in } B_{\frac {\dmin\e}{4}}(z_{\e,i}) \setminus B_{\e,i}, \\
    0 & \quad \text{in } \Omega \setminus B_{\frac {\dmin\e}{4}}(z_{\e,i})
    \end{cases}
\end{align}
Note that by \eqref{ass:strong.separation}, the supports of $c_{\e,i}^{app}$ are disjoint and contained in $\Omega$. In particular, we find by a direct calculation
\begin{align}
    \|\sum_i \nabla c_{\e,i}^{app}\|_{L^2(\Omega)}^2 = \sum_i \| \nabla c_{\e,i}^{app}\|_{L^2(\Omega)}^2 &\lesssim \e^3 \sum_i r_{i,\e}H_{\e,i}^2 \lesssim 1, \label{corrector.c.H^1}\\
     \|\sum_i  c_{\e,i}^{app}\|_{L^2(\Omega)}^2 &\lesssim \e^7 \sum_i r_{i,\e}^2  H_{\e,i}^2 \lesssim \e^4. \label{corrector.c.L^2}
\end{align}

\subsubsection*{Wasserstein distance for the empirical density}
For the empirical density $f_\e(t)$, we consider the $2$-Wasserstein distance from \eqref{W_2}.
Rather than choosing an optimal transport plan independently at each time, we propagate an optimal initial coupling along the corresponding characteristic flows.
This allows us to control the evolution of its transport cost.
To that end, let $H_\e \colon [0,T) \times \Omega \times \R_+ \to \R $ be a Lipschitz function such that\footnote{Note that this defines $H_\e$ only in the positions and radii of the actual particles. Outside, we can take an arbitrary Lipschitz function.}
\begin{align} \label{H_e}
    H_\e(t,z_{\e,i},r_{\e,i}(t)) = \fint_{\Gamma_{\e,i}(t)}G\left(c_\e(t,x),r_{\e,i}(t)\right) \dd\gamma.
\end{align}
Then, we define
\begin{equation}\label{eq:floweps}
\begin{array}{rcl}
\partial_t \Sigma_\e(t,x,\sigma)&=&-H_\e(t,x,\Sigma_\e(t,x,\sigma)),  \\
\Sigma_\e(0,x,\sigma)&=& \sigma
\end{array}
\end{equation}
such that 
\begin{equation}\label{eq:def_F^N}
r_{\e,i}(t)=\Sigma_\e(t,z_{\e,i},r_{\e,i,0}).
\end{equation}
We also define the characteristics associated to the the limit equation \eqref{evolution.f}
\begin{equation}\label{eq:flow}
\begin{array}{rcl}
\partial_t \Sigma(t,x,\sigma)&=& -\frac 1 {\Sigma(t,x,\sigma)} H(c(t,x),\Sigma(t,x,\sigma))\\
\Sigma(0,x,\sigma)&=& \sigma
\end{array}
\end{equation}
We can construct the transport plan $\pi_t (\d x_1, \d \sigma_1, \d x_2, \d \sigma_2)$, having $f_\e(t)$ as a first marginal and $f(t)$ as second marginal as follows. Let $\pi_0(\d x_1, \d \sigma_1, \d x_2, \d \sigma_2)$ be an optimal transport plan for  $\mathcal{W}_2(f_{\e}(0),f_{0})$ and define
$$
\pi_t:= ( P^1_x, \Sigma_\e(t)\circ P^1, P^2_x,  \Sigma(t)\circ P^2)_\# \pi_0.
$$
where $P^1, P^1_x, P^2, P^2_x$  are the projections from $(\Omega \times [0,\infty))^2$ defined by
\begin{align*}
    P^i(x_1,\sigma_1,x_2,\sigma_2) = (x_i,\sigma_i), &&  P^i_x(x_1,\sigma_1,x_2,\sigma_2) = x_i.
\end{align*}
Equivalently, $\pi_t$ is defined as the measure that satisfies for all $\varphi \in C_c^\infty(\Omega \times [0,\infty))^2$
\begin{align*}
   & \int_{(\Omega \times (0,\infty))^2} \varphi(x_1,\sigma_1, x_2,\sigma_2) \pi_t (\d x_1, \d \sigma_1, \d x_2, \d \sigma_2) \\
   &  \qquad =   \int_{(\Omega \times (0,\infty))^2} \varphi(x_1,\Sigma_\e(t,x_1,\sigma_1), x_2,\Sigma(t,x_2,\sigma_2)) \pi_0 (\d x_1, \d \sigma_1, \d x_2, \d \sigma_2)
\end{align*} 
 We define $\eta$  as the supremum of the squared  transport cost associated to  the transport plan $\pi_t$ 
 \begin{align}\label{eq:def_eta}
 \eta(t)&= \sup_{s \leq t}\int_{(\Omega \times (0,\infty))^2} (|x_1-x_2|^2+|\sigma_1-\sigma_2|^2) \pi_s (\d x_1, \d \sigma_1, \d x_2, \d \sigma_2) \\ \nonumber
 &= \sup_{s \leq t} \int_{(\Omega \times (0,\infty))^2} (|x_1-x_2|^2+|\Sigma_\e(s,x_1,\sigma_1)-\Sigma(s,x_2,\sigma_2)|^2)  \pi_0 (\d x_1, \d \sigma_1, \d x_2, \d \sigma_2).
 \end{align}


\subsubsection*{Relative energy inequality and proof of Theorem \ref{th:hom.critical}}

We consider the relative energy
\begin{align} \label{rel.energy}
    E_\e(t) := \|c_{\e} - c_{\e}^{app}\|_{L^2(\Omega_\e)}^2 + \eta(t) 
\end{align}

In the process of estimating $c_{\e,i} - c_{\e}^{app}$, we will need an estimate for the difference of the fluid velocities $u_\e - u$.
Since the equation for the fluid velocities are quasi stationary, we do not include the difference in the relative energy. Instead, we will estimate the difference by the relative energy.

We will prove the following three estimates which will directly imply Theorem \ref{th:hom.critical}.
The proof of these propositions is the content of \cref{ssec:estimate_ueps_u,ssec:estimate_ceps_c,ssec:estimate_eta} below.

\begin{proposition} \label{pro:fluid}
Under the assumptions of Theorem \ref{th:hom.critical}, for $\e < \e_0$
    \begin{align*}
    \| u_\e - u\|_{L^2(\Omega_\e)}^2 \lesssim \eta + \e^2 + \|f_\e^{app}(0) - f_0\|_{(H^1_xW_\sigma^{1,\infty})^\ast}^2 +  \|h_\e - h\|_{L^2(S;(H^{1}(\Omega))^*)}^2.
\end{align*}
\end{proposition}

\begin{proposition} \label{pro:chemistry}
Under the assumptions of Theorem \ref{th:hom.critical}, for $\e < \e_0$
\begin{align*}
    &\frac{\d}{\d t } \frac 1 2 \|c_{\e} - c_{\e}^{app}\|_{L^2(\Omega_\e)}^2 + \| \nabla (c_{\e} - c_{\e}^{app})\|_{L^2(\Omega_\e)}^2 \\
    & \qquad \qquad \lesssim \e^2 + \eta + \|f_\e^{app}(0) - f_0\|_{(H^1_xW_\sigma^{1,\infty})^\ast}^2 +  \|c_{\e} - c_{\e}^{app}\|_{L^2(\Omega_\e)}^2 + \|u_\e - u\|_{L^2(\Omega_\e)}^2
\end{align*}
\end{proposition} 

\begin{proposition} \label{pro:eta}
Under the assumptions of Theorem \ref{th:hom.critical}, for $\e < \e_0$
    \begin{align*}
    \frac{\d}{\d t } \eta  \leq \frac 1 2 \|\nabla (c_{\e} - c_{\e}^{app})\|_{L^2(\Omega_\e)}^2 + C(\eta + \e^2 +   \|c_{\e} - c_{\e}^{app}\|_{L^2(\Omega_\e)}^2 ) .
\end{align*}
\end{proposition} 
\begin{proof}[Proof of Theorem \ref{th:hom.critical}]
    Combining the above propositions, we have
 \begin{align*}
     \frac{\d}{\d t } E_\e(t) + \frac 1 2 \|\nabla (c_{\e} - c_{\e}^{app})\|_{L^2(\Omega_\e)}^2 \lesssim E_\e(t) + \e^2 + \|f_\e^{app}(0) - f_0\|_{(H^1_xW_\sigma^{1,\infty})^\ast}^2 + \|h_\e - h\|_{(H^{1}(\Omega))^*}^2.
 \end{align*}
 Hence, by Gronwall's inequality
 \begin{multline*}
     E_\e(t) + \|\nabla (c_{\e} - c_{\e}^{app})\|_{L^2(0,t;L^2(\Omega_\e(t)))}^2 \\ \lesssim  E_\e(0) + \e^2  + \|f_\e^{app}(0) - f_0\|_{(H^1_xW_\sigma^{1,\infty})^\ast}^2 + \|h_\e - h\|_{L^2(S;(H^{1}(\Omega))^*)}^2 .
 \end{multline*}
 We have, using \eqref{corrector.c.L^2},
 \begin{align*}
     E_\e(0) &\leq (\mathcal W_2(f_\e(0), f_0))^2 + 2\|c_{\e,0} - c_0\|_{L^2(\Omega_\e)}^2 + 2 \|c_0 - c_{\e}^{app}(0)\|_{L^2}^2 \\&\lesssim (\mathcal W_2(f_\e(0), f_0))^2 + \|c_{\e,0} - c_0\|_{L^2(\Omega_\e)}^2 + \e^4
 \end{align*}
 and
 \begin{align*}
      E_\e(t) &\geq \mathcal W_2(f^\e(t), f(t)))^2 + \frac 1 2 \|c_{\e}(t) - c(t)\|_{L^2(\Omega_\e(t))}^2  - 2 \|c(t) - c_{\e}^{app}(t)\|_{L^2(\Omega_\e(t))}^2  \\
      &\geq \mathcal W_2(f^\e(t), f(t)))^2 + \frac 1 2 \|c_{\e}(t) - c(t)\|_{L^2(\Omega_\e(t))}^2 - C \e^4.
 \end{align*}
 Hence, we deduce
 \begin{multline} \label{corr.est.c}
     \|c_{\e}(t) - c(t)\|_{L^2(\Omega_\e(t))}^2 + \mathcal W_2(f^\e(t), f(t)))^2 +  \|\nabla (c_{\e} - c_{\e}^{app})\|_{L^2(0,t;L^2(\Omega_\e(t)))}^2\\
     \lesssim   (\mathcal W_2(f_\e(0), f_0))^2 + \|c_{\e,0} - c_0\|_{L^2(\Omega_\e)}^2 + \|f_\e^{app}(0) - f_0\|_{(H^1_xW_\sigma^{1,\infty})^\ast}^2 + \e^2 + \|h_\e - h\|_{L^2(S;(H^{1}(\Omega))^*)}^2 .
 \end{multline}
 This concludes the proof. 
\end{proof}

\subsection{Proof of \texorpdfstring{\cref{pro:fluid}}{Proof of Proposition 5.1}: Estimate of \texorpdfstring{$u_\e - u$}{u epsilon minus u}}\label{ssec:estimate_ueps_u}

Obtaining quantitative estimates for $u_\e - u$ is a classical problem in (quantitative) homogenization. Specifically, in the ciritical scaling leading to the Brinkman equations, this has been addressed for instance in \cite{allaire_homogenization_1990, HoferJansen24, MecherbetHillairet20}. However, these results are not directly applicable in our case: First, the results in \cite{allaire_homogenization_1990} and many other papers are (mainly) concerned with periodic distributions of particles. Second,  the estimates in \cite{HoferJansen24} are restricted to a stochastic setting with i.i.d. particles. Finally, \cite{MecherbetHillairet20} considers particle distributions much as in our setting but they obtain $L^p$ estimates for $p<3/2$ only, which does not seem sufficient for our relative energy argument. Finally, the inhomogeneous boundary conditions \eqref{inh.Dirichlet} are absent in all previous results. Nevertheless, relying on the existing methods in the literature, the adaptations to the present setting are not very difficult. For the sake of completeness, we present the full argument. Our methods are based on classical corrector estimates from \cite{allaire_homogenization_1990} and further based on \cite{Hoefer23, HoeferHuebnerRosenau26} who prove quantitative convergence results for the homogenization of the Navier-Stokes equations with vanishing viscosity. 

\subsubsection*{Setup of the approximation}

Naively, one  wants to test the equations of $u$ and $u_\e$  with $u_\e - u$. The problem is that $u \neq  u_\e $ on $\Gamma_\e$ and hence $u_\e - u$ is not an admissible testfunction for the equation of $u_\e$. Therefore, one corrects the function $u$ to a suitable function $u_\e^{app}$  in order to match the boundary conditions. In fact this is necessary because smallness of $u_\e - u$ in $H^1$ fails and it is classical that correctors are needed. In our current setting, we have to provide additional correctors due to the inhomogeneous bounday conditions.

We therefore define an approximation of the fluid velocity
\begin{align}\label{eq:ueps_app}
    u_{\e}^{app} := w_\e u - \mathcal B_\e(u) + u_{\e}^{inh},
\end{align}
Here  $w_\e \in W^{1,\infty}(\Omega)$ is the classical corrector adapted from Allaire (see e.g. \cite{allaire_homogenization_1990}), which is a matrix valued divergence free function such that $w_\e = 0$ in $\cup_i B_{\e,i}$. Moreover,  $\mathcal B_\e$ is a Bogovskiǐ type operator that satisfies 
\begin{align*}
    \mathcal B_\e (u) \in H^1_0(\Omega_\e) \qquad \dv \mathcal B_\e (u) = w_\e : \nabla u
\end{align*}

Finally,  $u_{\e}^{inh}$ is a corrector for the inhomogeneous boundary conditions \eqref{inh.Dirichlet}. This is defined as follows.
 First, we take a cut-off function 
$\theta \in C_c^\infty(B_{2}(0))$ with $\theta = 1$ in $B_{1}(0)$.
We then set
\begin{align*}
    v(x) &:= \eta(x) \frac{x}{|x|^3}, \\
    \bar u_{\e}^{inh} &:= \e^{\alpha} \sum_{i} \partial_t r_{\e,i} v\left(\frac{x - z_{\e,i}}{\e^\alpha r_{\e,i}}\right).
\end{align*}
Then, $\bar u_{\e}^{inh} = \e^\alpha \partial_t r_{\e,i} \nu_{\e,i}$ on $\Gamma_{\e,i}$ for all $i\in\indexI$.
Moreover,  $  \e^{-\alpha}\|\bar u_{\e}^{inh}\|_\infty + \|\nabla \bar u_{\e}^{inh}\|_\infty \lesssim 1$ and hence
\begin{align*}
    \|\nabla \bar u_{\e}^{inh}\|_{L^2(\Omega)} + \e^{-2\alpha} \|\bar u_{\e}^{inh}\|_{L^2(\Omega)} \lesssim \e^{\frac 3 2 (\alpha -1)}.
\end{align*}
To get a divergence free function $u_{\e}^{inh}$, we finally employ a Bogovskiǐ operator in $\Omega_\e$. 
By Proposition \ref{pro:Bog} there exists $\check u_{\e}^{inh} \in H^1(\Omega)$ such that $\check u_{\e}^{inh}  = 0$ in $\cup_i B_{\e,i}$, $\dv \check u_{\e}^{inh} = \dv \bar u_{\e}^{inh} $ and 
\begin{align*}
    \e^{\frac{3- \alpha}2} \|\nabla\check u_{\e}^{inh} \|_2 + \|\check u_{\e}^{inh} \|_2   \lesssim  \|\dv \bar u_{\e}^{inh} \|_2 \lesssim  \e^{\frac 3 2 (\alpha -1)}.
\end{align*}
We then define 
\begin{align*}
    u_{\e}^{inh} = \bar u_{\e}^{inh} - \check u_{\e}^{inh},
\end{align*}
which satisfies $\dv u_{\e}^{inh} = 0$ in $\Omega$, $u_{\e}^{inh} = \e^\alpha \partial_t r_{\e,i} \nu_{\e,i}$ for all $i\in\indexI$ and 
\begin{align} \label{est.u^inh}
    \e^{\frac{3- \alpha}2}  \|\nabla u_{\e}^{inh} \|_2 +  \|u_{\e}^{inh} \|_2 \lesssim \e^{\frac 3 2 (\alpha -1)}.
\end{align}



\subsubsection*{Corrector estimates}

We  introduce $w_\e$. For $k=1,2,3$, we define $(w_k,q_k) \in \dot H^1(\R^3 \setminus B_1(0)) \times L^2(\R^3 \setminus B_1(0))$ as the unique solution to 
\begin{align} \label{w_k}
\begin{array}{rcll}
    - \Delta w_k + \nabla q_k &=& 0 &\qquad \text{in } \R^3 \setminus B_1(0), \\
    \dv w_k &=& 0&\qquad \text{in } \R^3 \setminus B_1(0), \\
    w_k &=& e_k& \qquad \text{on } \partial B_1(0).
\end{array}
\end{align}
Here, $e_k$ denotes the $k$-th unit vector of the standard basis of $\R^3$.
Moreover, we consider
\begin{align}
    C_i^\e &:= B_{\frac{\dmin \e}4}(z_{\e,i}) \setminus \mathcal B_{\e,i} \label{C_i},  \\
    D_i^\e &:= B_{\frac{\dmin \e}2}(z_{\e,i}) \setminus B_{\frac{\e}4}(z_{\e,i}). \label{D_i}
\end{align}
 Then,  we define $(w_{\e,k},q_{\e,k}) \in W^{1,\infty}(\Omega)^3 \times L^2(\Omega)$ such that for all $ i\in\indexI$
\begin{align} \label{w^eps}
 \begin{array}{rl}
 w_{\e,k}(x)=0,\quad q_{\e,k} = 0 &\qquad \text{in } B_{\e,i}, \\
    w_{\e,k}(x) = e_k - w_k\left(\frac{x-z_{\e,i}}{r_{\e,i} \e^\alpha}\right), \quad q_{\e,k}(x) = - \e^{-\alpha} r_{\e,i}^{-1} q_k\left(\frac{x-z_{\e,i}}{r_{\e,i} \e^\alpha}\right) &\qquad \text{in } C_i^\e, \\
    - \Delta w_{\e,k}+ \nabla q_{\e,k} = 0,  \quad \dv w_{\e,k} = 0 &\qquad \text{in } D_i^\e, \\
    w_{\e,k} = e_k, \quad  q_{\e,k} = 0 &\qquad \text{in } \Omega_\e \setminus \bigcup_i B_{\frac{\dmin\e}2}(z_{\e,i}),
\end{array}
\end{align}
Note that the Stokes equations in $D_i^\e$ are complemented with inhomogeneous no-slip boundary conditions due to the requirement $w_k^\e \in W^{1,\infty}(\Omega_\e)$.
We will write $w_\e$ for the matrix-valued function with columns $w_{\e,k}$, and $q_\e$ for the (row-)vector with entries $q_{\e,k}$.

Following \cite[Lemmas 2.1 and 2.3]{HoeferHuebnerRosenau26}   we have the following estimates.
\begin{lemma} \label{lem:correctors}
We have
\begin{align}
    \|\nabla w_\e\|_{L^2} &\lesssim \e^{\frac{\alpha-3}{2}}, \label{bounds.w}\\
    \|\mathcal B_\e(u)\|_{H^1(\Omega)} + \|w_\e - {\rm I}_3\|_{L^2(\Omega)} +    \|u_{\e}^{app} - u\|_{L^2(\Omega)} &\lesssim \e^{\alpha - 1} \label{bounds.w.etc}.
\end{align}
\end{lemma}
\begin{proof}
    In \cite{HoeferHuebnerRosenau26}, the particles are identical. However, since the estimates are obtained from estimates around each particle, they straightforwardly generalize to our setting as long as the radii are uniformly bounded. We then apply  \cite[Lemmas 2.1 and 2.3]{HoeferHuebnerRosenau26}  with $\eta_{\e,i} = \min\{\e \dmin,1\}$ for all $ i \in \mathcal I_\e$.
    This yields immediately \eqref{bounds.w}, 
    \begin{align*}
        \|w_\e - {\rm I}_3\|_{L^2(\Omega)}^2 \lesssim \sum_{i \in \mathcal I_\e} \e^{2\alpha +1} = \e^{2(\alpha -1)}, \\
        \|\mathcal B_\e(u)\|^2_{H^1(\Omega)} \lesssim \|w_\e - {\rm I}_3\|_{L^2(\Omega)}^2 \lesssim \e^{2(\alpha -1)}, \\
         \|u_{\e}^{app} - u\|^2_{L^2(\Omega)} \lesssim \|\mathcal B_\e(u)\|_{L^2(\Omega)}^2 + \|w_\e - {\rm I}_3\|_{L^2(\Omega)}^2 + \|u_\e^{inh}\|_{L^2(\Omega)}^2 \lesssim \e^{2(\alpha -1)},
    \end{align*}
    where we used \eqref{est.u^inh} in the last estimate. 
\end{proof}

Moreover, we adapt \cite[Lemma 2.2]{HoeferHuebnerRosenau26}. To this end, we abbreviate
\begin{align} \label{R}
    \mathcal R(t,x) := 6 \pi {\rm I}_3 \int \sigma f(t,x,\d\sigma),
\end{align}
and we recall the definition of $\eta$  from \eqref{eq:def_eta}.
The following decomposition is the central step in the classical Allaire corrector argument: it identifies the collective force exerted by the local Stokes correctors with the effective Brinkman resistance.
Here, we quantify this identification by showing that the corresponding distribution $M_\e$ is close to $\mathcal R$ in $H^{-1}(\R^3)$, with an error controlled by the discrepancy of the particle distributions.

\begin{lemma} \label{lem:M_eps} We can  write
\begin{align} \label{decomposition.M.gamma}
    (-\Delta w_\e +  \nabla q_\e) =  \e^{\alpha- 3} M_\e - \gamma_\e 
\end{align}
for some $M_\e,\gamma_\e \in W^{-1,\infty}(\R^3)$
where $\langle \gamma_\e, v \rangle = 0$ for all $v \in H^1_{\Gamma_\e}(\Omega_\e)$ and, for all $\psi \in H^1(\R^3)$,
\begin{equation} \label{M.W^-1,infty}
    \langle M_\e - \mathcal R, \psi \rangle
    \lesssim \left(\e^{\alpha - 1} + \eta^{\frac 1 2} \right) \|\psi\|_{L^2(\R^3)} + (\e + \|f_\e^{app}(0) - f_0\|_{(H^1_xW_\sigma^{1,\infty})^\ast}) \| \psi\|_{H^1(\R^3)}.
\end{equation}
\end{lemma}
Compared to \cite[Lemma 2.2]{HoeferHuebnerRosenau26}, the radii are not constant, which leads to the slightly more complex structure of $\mathcal R$. On the other hand, the separation assumption \eqref{ass:strong.separation} simplifies the argument. In particular, it is not necessary to introduce the enlarged cubes $\tilde Q_{i,\e}$ considered in \cite{HoeferHuebnerRosenau26}.
For the convenience of the reader, we include the proof of this lemma in Appendix \ref{app:Delta.w}.

\subsubsection*{Proof of Proposition \ref{pro:fluid}}
We write $\Delta (w_\e u) = \Delta w_\e u + 2 \sum_{j=1}^3 \partial_j (w_\e - {\rm I}_3) \partial_j u + w_\e \Delta u$ and use the PDEs of $u_\e$ and $u$  to compute
\begin{align*}
    \|\nabla( u_\e - u_\e^{app})\|_{L^2(\Omega_\e)}^2 = \sum_{j=1}^5 I_j,
\end{align*}
where
\begin{align*}
   I_1 &=   \int_{\Omega_\e} (h_\e - w_\e h) \cdot (u_\e - u_\e^{app}), \\
   I_2 &= -\int_{\Omega_\e} 2 (\nabla (w_\e - {\rm I}_3) \nabla u) \cdot (u_\e - u_\e^{app}), \\
   I_3 &= - \int_{\Omega_\e} \nabla \mathcal B_\e (u) :  \nabla( u_\e - u_\e^{app}), \\
      I_4 &= - \int_{\Omega_\e} \nabla u_\e^{inh}:  \nabla( u_\e - u_\e^{app}), \\
      I_5 &= -\langle \Delta w_\e u + 6 \pi \int_0^\infty \sigma u f(\cdot, \dd \sigma) , u_\e - u_\e^{app} \rangle,
\end{align*}
Clearly, using Poincar\'e and Lemma \ref{lem:correctors},
\begin{align*}
    |I_1| &\leq \|w_\e - {\rm I}_3\|_{L^2(\Omega_\e)} \|h\|_\infty  \|u_\e - u_\e^{app}\|_{L^2(\Omega_\e)} + \|h_\e - h\|_{(H^{1}(\Omega))^*} \|\nabla(u_\e - u_\e^{app})\|_{L^2(\Omega_\e)} \\
    & \leq C_\delta\left(\|w_\e - {\rm I}_3\|^2_{L^2(\Omega_\e)} \|h\|^2_{L^\infty(\Omega_\e)} + \|h_\e - h\|^2_{(H^{1}(\Omega))^*}\right) +  \delta \|\nabla(u_\e - u_\e^{app})\|^2_{L^\infty(\Omega_\e)} \\
    &\leq C_\delta \left(\e^{2 \alpha -2} + \|h_\e - h\|^2_{(H^{1}(\Omega))^*}\right) +  \delta \|\nabla(u_\e - u_\e^{app})\|^2_{L^\infty(\Omega_\e)}
\end{align*}
Next, for $I_2$, we integrate by parts and use the Poincar\'e inequality to get
\begin{align*}
    \frac 1 2 |I_2| &\leq  \|w_\e - {\rm I}_3\|_{L^2(\Omega_\e)} \| \Delta u\|_{L^\infty(\Omega_\e)} \|u_\e - u_\e^{app}\|_{L^2(\Omega_\e)} \\
    &\qquad + \|w_\e - {\rm I}_3\|_{L^2(\Omega_\e)}\| \nabla  u\|_{L^\infty(
    \Omega_\e)} \|\nabla (u_\e - u_\e^{app})\|_{L^2(\Omega_\e)} \\
    & \leq C_\delta \|u\|_{W^{2,\infty}(\Omega_\e)}^2 \|w_\e - {\rm I}_3\|_{L^2(\Omega_\e)}^2 + \delta  \|\nabla (u_\e - u_\e^{app})\|_{L^2(\Omega_\e)}^2 \\
    &\leq  C_\delta \e^{2 \alpha -2} +  \delta \|\nabla(u_\e - u_\e^{app})\|^2_{L^2(\Omega_\e)}
\end{align*}
Further,
\begin{align*}
    |I_3| \leq C_\delta  \| \nabla \mathcal B_\e (u)\|_{L^2(\Omega_\e)}^2  + \delta \|\nabla( u_\e - u_\e^{app})\|_{L^2(\Omega_\e)}^2 \lesssim  C_\delta \e^{2 \alpha -2} +  \delta \|\nabla(u_\e - u_\e^{app})\|^2_{L^2(\Omega_\e)}, \\
       |I_4| \leq  C_\delta  \| \nabla u_\e^{inh}\|_{L^2(\Omega_\e)}^2  + \delta \|\nabla( u_\e - u_\e^{app})\|_{L^2(\Omega_\e)}^2 \lesssim C_\delta\e^{3(\alpha-1)}  + \delta \|\nabla( u_\e - u_\e^{app})\|_{L^2(\Omega_\e)}^2,
\end{align*}
where we used \eqref{est.u^inh} in the last estimate.

Finally, we have by Lemma \ref{lem:M_eps} using that $u_\e - u_\e^{app} \in H^1_{\Gamma_\e}(\Omega_\e)$ is divergence free and using that $\alpha =3$ and the Poincar\'e inequality
\begin{align*}
    I_5 &= -\langle \Delta w_\e + 6 \pi {\rm I}_3 \int \sigma f(\cdot, \dd \sigma) , u_\e - u_\e^{app} \rangle \\
    &= \langle -\Delta w_\e + \nabla q_\e -  6 \pi {\rm I}_3 \int \sigma f(\cdot, \dd \sigma), u_\e - u_\e^{app} \rangle \\
    &=  \langle M_\e - \mathcal R, u_\e - u_\e^{app} \rangle \\
    & \leq C_\delta (\e^2 + \eta + \|f_\e^{app}(0) - f_0\|_{(H^1_xW_\sigma^{1,\infty})^\ast}^2) + \delta  \|\nabla(u_\e - u_\e^{app})\|^2_{L^2(\Omega_\e)}.
\end{align*}

In total, we deduce 
\begin{align*}
    \|\nabla( u_\e - u_\e^{app})\|_{L^2(\Omega_\e)}^2 \lesssim \eta + \e^2 + \|f_\e^{app}(0) - f_0\|_{(H^1_xW_\sigma^{1,\infty})^\ast}^2 +  \|h_\e - h\|_{(H^{1}(\Omega))^*}^2.
\end{align*}
Combining this with \eqref{est.u^inh} and \eqref{bounds.w}, we infer 
\begin{align} \label{corr.est.u}
    \|\nabla( u_\e - w_\e u)\|_{L^2(\Omega_\e)}^2 \lesssim \eta + \e^2 + \|f_\e^{app}(0) - f_0\|_{(H^1_xW_\sigma^{1,\infty})^\ast}^2+ \|h_\e - h\|_{(H^{1}(\Omega))^*}^2.
\end{align}
Together with the Poincar\'e inequality and \eqref{bounds.w}, we conclude
\begin{align*}
    \|u_\e - u\|_{L^2(\Omega_\e)}^2 \lesssim \eta + \e^2 + \|f_\e^{app}(0) - f_0\|_{(H^1_xW_\sigma^{1,\infty})^\ast}^2 + \|h_\e - h\|_{(H^{1}(\Omega))^*}^2.
\end{align*}
as claimed.

\subsection{Proof of Proposition \ref{pro:chemistry}: Estimate of \texorpdfstring{$\|c_{\e} - c_{\e}^{app}  \|_{L^2(\Omega_\e(t))}^2 $}{c epsilon minus c app}}\label{ssec:estimate_ceps_c}

We have by Lemma \ref{lemma:times_derivative}
\begin{align*}
   \frac{\d}{\d t} \frac 1 2  \|c_{\e} - c_{\e}^{app}\|_{L^2(\Omega_\e)}^2  &=    \langle \partial_t c_{\e} - \partial_t c_{\e}^{app}, c_{\e} - c_{\e}^{app} \rangle_{H^1(\Omega_\e)} - \frac 1 2 \e^{\alpha } \sum_i  \int_{\Gamma_{\e,i}} \partial_t r_{\e,i} |c_{\e} - c_{\e}^{app}|^2.
\end{align*}
We compute separately using the PDEs for $c_\e$ and $c$, respectively, for almost every $t \in S$
\begin{align*}
 \langle \partial_t c_{\e}, c_{\e} - c_{\e}^{app} \rangle_{H^1(\Omega_{\e}(t))^*}
    &=  \int_{\Omega_\e} ( u_\e c_\e - \nabla c_\e ) \cdot  \nabla (c_{\e} - c_{\e}^{app}) + F(c_\e) (c_{\e} - c_{\e}^{app}) \\
    &+\sum_i  \int_{\Gamma_{\e,i}} \left( \e^\alpha \partial_t r_{\e,i} c_\e +\e^{3-2\alpha} G(c_\e,r_{\e,i})\right) (c_{\e} - c_{\e}^{app}) \\
\end{align*}
and 
\begin{align*}
 \int_{\Omega_\e}  \langle \partial_t c, c_{\e} - c_{\e}^{app} \rangle &=  \int_{\Omega_\e} ( u c -  \nabla c )\cdot  \nabla (c_{\e} - c_{\e}^{app}) + F(c) (c_{\e} - c_{\e}^{app}) \\
    &+\int_{\Omega_\e} \int_0^\infty 4 \pi\sigma H(c,\sigma)f(\cdot,\d\sigma) (c_{\e} - c_{\e}^{app}) + \int_{ \Gamma_\e} ( u c -  \nabla c ) \cdot \nu_\e (c_{\e} - c_{\e}^{app}).
\end{align*}
We compute explicitly
\begin{align*}
    - \Delta c_{\e,i}^{app} =  4 \pi \e^3 r_{i,\e} H_{\e,i}( \delta_{\Gamma_{\e,i}}  - \delta_{\partial B_{\frac {\dmin \e}{4}}(z_{\e,i})}), 
\end{align*}
where $\delta_{\partial_{B_r(x)}}$ denotes the uniform probability measure on $\partial B_r(x)$. Hence,
\begin{align*}
    \sum_i \int_{\Omega_\e} \nabla c_{\e,i}^{app} \cdot \nabla (c_{\e} - c_{\e}^{app} ) =  4 \pi \e^3 \sum_i \langle  r_{i,\e} H_{\e,i}( \delta_{\Gamma_{\e,i}}  - \delta_{\partial B_{\frac {\dmin \e}{4}}(z_{\e,i})}), c_{\e} - c_{\e}^{app} \rangle
\end{align*}
In total, we find 
\begin{align} \label{Gronwall.c.0}
      \frac 1 2 \frac{\d}{\d t} \|c_{\e}(t) - c_{\e}^{app}(t)\|_{L^2(\Omega_\e)}^2  +   \|\nabla ( c_{\e}(t) - c_{\e}^{app}(t))\|_{L^2(\Omega_\e)}^2  =  \sum_{j=1}^6 I_j
\end{align}
where
\begin{align*}
    I_1 &= -\sum_i \int_{\Omega_\e} \partial_t c_{\e,i}^{app} ( c_{\e}- c_{\e}^{app}), \\
    I_2 &=  \int_{\Omega_\e}( u_\e c_\e - u c) \cdot  \nabla (c_{\e} - c_{\e}^{app} ) \\
    I_3& = \sum_i   \int_{\Gamma_{\e,i}} \e^\alpha \partial_t r_{\e,i} c_\e (c_{\e} - c_{\e}^{app}) - \frac{\e^\alpha} 2 \sum_i  \int_{\Gamma_{\e,i}} \partial_t r_{\e,i} |c_{\e} + c_{\e}^{app}|^2 - \int_{ \Gamma_\e} ( u c -  \nabla c ) \cdot \nu_\e (c_{\e} - c_{\e}^{app})\\
    I_4 &=  \int_{\Omega_\e} ( F( c_\e) - F(c))  (c_{\e} - c_{\e}^{app} ) \\
    I_5 &= \sum_i   \int_{\Gamma_{\e,i}}  \e^{3-2\alpha} G(c_\e,r_{\e,i}) (c_{\e} - c_{\e}^{app}) - 4 \pi \e^3 \sum_i \langle  r_{i,\e} H_{\e,i} \delta_{\Gamma_{\e,i}}  , c_{\e} - c_{\e}^{app} \rangle\\
    I_6 &=  4 \pi \e^3 \sum_i \langle  r_{i,\e} H_{\e,i}(\delta_{\partial B_{\frac {\dmin \e}{4}}(z_{\e,i})}), c_{\e} - c_{\e}^{app} \rangle -   \int_{\Omega_\e} \int_0^\infty 4 \pi\sigma H(c,\sigma) f(\cdot,\d\sigma) (c_{\e} - c_{\e}^{app}). 
\end{align*}
We estimate these terms separately. For $I_1$, we use
\[
|\partial_t c_{\e,i}^{app}|^2\lesssim 
\left[\dot r_i^2H_{\e,i}^2+r_i^2|\nabla H_{\e,i}|^2\left(|\partial_tc|^2+|\dot r_i|^2\right)\right]\left|\frac{\e^3}{|x-z_{\e,i}|}-\frac{4\e^2}{\dmin}\right|^2
\]
 and an estimate analogous to \eqref{corrector.c.L^2} to get
\begin{align*}
    |I_1| &\lesssim \|c_{\e} - c_{\e}^{app}\|^2_{L^2(\Omega_\e)} +  \sum_i\|\partial_t c_{\e,i}^{app}\|_{L^2(\Omega_\e)}^2  
     \lesssim  \|c_{\e} - c_{\e}^{app}\|^2_{L^2(\Omega_\e)} + \e^4.
\end{align*}

Regarding $I_2$, 
\begin{align*}
      |I_2| &\leq  \|u_\e - u\|_{L^2(\Omega_\e)} \|c_\e\|_{L^\infty(\Omega_\e)} \|\nabla ( c_{\e} - c_{\e}^{app})\|_{L^2(\Omega_\e)} +  \|u\|_{L^\infty(\Omega_\e)}  \|c_{\e} - c\|_{L^2(\Omega_\e)}\|\nabla ( c_{\e} - c_{\e}^{app})\|_{L^2(\Omega_\e)} \\
      &\leq  C_\delta \|u_\e - u\|^2_{L^2(\Omega_\e)} \|c_\e\|^2_{L^\infty(\Omega_\e)} +  C_\delta \|u\|_{L^\infty(\Omega_\e)}^2 \left(\|c_{\e} - c_\e^{app}\|_{L^2(\Omega_\e)}^2 + \|c - c_\e^{app}\|_{L^2(\Omega_\e)}^2 \right)  \\
      &\qquad+ \delta \|\nabla ( c_{\e} - c_{\e}^{app})\|^2_{L^2(\Omega_\e)} \\
      & \leq \delta \|\nabla ( c_{\e} - c_{\e}^{app})\|^2_{L^2(\Omega_\e)} + C_\delta(\|u_\e - u\|^2_{L^2(\Omega_\e)} + \|c_{\e} - c_\e^{app}\|_{L^2(\Omega_\e)}^2 +\e^4)
\end{align*}
where we used \eqref{corrector.c.L^2} in the last estimate.

For $I_3$, we  use the trace estimate from Lemma \ref{lemma:trace_estimates} (ii) to obtain
\begin{multline*}
       \sum_i   \int_{\Gamma_{\e,i}} \e^\alpha \partial_t r_{\e,i} c_\e (c_{\e} - c_{\e}^{app}) - \frac 1 2 \e^{\alpha} \sum_i  \int_{\Gamma_{\e,i}} \partial_t r_{\e,i} |c_{\e} - c_{\e}^{app}|^2 \\
       \lesssim \e^\alpha \|c_{\e} - c_{\e}^{app}\|_{L^2(\Gamma_\e)}^2 + \e^\alpha\|1\|_{L^2(\Gamma_\e)}^2 \\
      \lesssim \e^{3(\alpha-1)} \|c_{\e} - c_{\e}^{app}\|_{L^2(\Omega_\e)}^2 + \e^{2\alpha} \|\nabla(c_{\e} - c_{\e}^{app})\|_{L^2(\Omega_\e)}^2 + \e^{3(\alpha-1)}.
\end{multline*}
For the last term in $I_3$, we introduce the notation 
\begin{align} \label{[]_Gamma}
    [g]_{\Gamma_{\e,i}} := \fint_{\Gamma_{\e,i}} g(\gamma) \dd \gamma. 
\end{align}
Then, 
\begin{multline*}
\int_{ \Gamma_\e} ( u c -  \nabla c )\cdot \nu (c_{\e} - c_{\e}^{app}) \\
= \sum_i \int_{ \Gamma_{\e,i}} \left(( u c -  \nabla c ) \cdot \nu - [(u c -  \nabla c)\cdot \nu ]_{\Gamma_{\e,i}} \right) \left( (c_{\e} - c_{\e}^{app}) - [c_{\e} - c_{\e}^{app}]_{\Gamma_{\e,i}} \right) \\
+ \int_{ \Gamma_\e} [(u c -  \nabla c)\cdot \nu ]_{\Gamma_{\e,i}} (c_{\e} - c_{\e}^{app}).
\end{multline*}
For the first right-hand side term, we use Lemma \ref{lemma:trace_estimates} (iii) to get
\begin{multline*}
    -\sum_i \int_{ \Gamma_{\e,i}} \left(( u c -  \nabla c ) \cdot \nu - [(u c -  \nabla c)\cdot \nu ]_{\Gamma_{\e,i}} \right) \left( (c_{\e} - c_{\e}^{app}) - [c_{\e} - c_{\e}^{app}]_{\Gamma_{\e,i}} \right) \\
    \leq C \e^{\alpha/2} \|( u c -  \nabla c ) \cdot \nu - [(u c -  \nabla c)\cdot \nu]_{\Gamma_{\e,i}}\|_{L^2(\Gamma_\e)}  \|\nabla(c_{\e} - c_{\e}^{app})\|_{L^2(\Omega_\e)} \\
    \leq C_\delta \e^{3(\alpha-1)} + \delta \|\nabla(c_{\e} - c_{\e}^{app})\|_{L^2(\Omega_\e)}^2
\end{multline*}
where we used $|\Gamma_\e| \lesssim \e^{2\alpha -3}$ and $(uc-\nabla c)\in L^{\infty}(\Omega)^3\cap H^1(\Omega)^3$ in the last estimate.
For the second right-hand side term, we use that $[V \cdot \nu]_{\Gamma_{\e,i}} = 0 $ for every $V\in\R^3$ such that, by Lipschitz continuity of $ u c -  \nabla c$, 
\begin{align*}
    \int_{ \Gamma_\e} [(u c -  \nabla c)\cdot \nu ]_{\Gamma_{\e,i}} (c_{\e} - c_{\e}^{app})
    &=\sum_i\int_{\Gamma_{\e,i}} \left[(u c -  \nabla c)\cdot \nu-(u c -  \nabla c)(z_{\e,i})\cdot \nu \right]_{\Gamma_{\e,i}} (c_{\e} - c_{\e}^{app}) \\
    &\lesssim \e^{\alpha} \|c_{\e} - c_{\e}^{app}\|_{L^1(\Gamma_\e)} \\
    &\lesssim \e^{3(\alpha-1)} \|c_{\e} - c_{\e}^{app}\|_{L^2(\Omega_\e)}^2 + \e^{2\alpha} \|\nabla(c_{\e} - c_{\e}^{app})\|_{L^2(\Omega_\e)}^2 + \e^{3(\alpha-1)}
\end{align*}
as for the first term in $I_3$.
In total, we have (for any fixed $\delta$ and $\e$ sufficiently small)
\begin{align*}
    I_3 \leq C \e^{3(\alpha-1)} \|c_{\e} - c_{\e}^{app}\|_{L^2(\Omega_\e)}^2 + C_\delta \e^{3(\alpha-1)} + \delta \|\nabla(c_{\e} - c_{\e}^{app})\|_{L^2(\Omega_\e)}^2
\end{align*}
For $I_4$, by local Lipschitz continuity of $F$ and uniform boundedness of $c_\e, c$,
\begin{align*}
    |I_4| \lesssim  \|c - c_{\e}^{app}\|_{L^2(\Omega_\e)}^2 + \|c_{\e} - c_{\e}^{app}\|_{L^2(\Omega_\e)}^2 \lesssim \e^{4} + \|c_{\e} - c_{\e}^{app}\|_{L^2(\Omega_\e)}^2
\end{align*}

We turn to the genuinely nonstandard part of the concentration estimate, namely the homogenization of the interfacial exchange contained in the remaining terms $I_5$ and $I_6$.
Here, $I_5$ measures the local mismatch between the microscopic exchange law and the boundary-layer corrected approximation, while $I_6$ accounts for the replacement of the discrete particle distribution by its macroscopic limit.
For $I_5$, we split $I_5 = I_{5,1} + I_{5,2}$ with 
\begin{align} \label{I_5.1}
    I_{5,1} &= \e^{-3} \sum_i \int_{\Gamma_{\e,i}} \left( G(c_\e,r_{\e,i}) -  G(c_\e^{app},r_{\e,i}) \right) ( c_{\e} - c_{\e}^{app}), \\
    I_{5,2} &= \e^{-3} \sum_i \int_{\Gamma_{\e,i}} \left(  G(c_\e^{app},r_{\e,i}) -  r_{i,\e}^{-1} H_{\e,i}\right) ( c_{\e} - c_{\e}^{app}) \label{I_5.2}
\end{align}
The first term is non-positive since 
\begin{align*}
    \left(G(c_\e,r_{\e,i}) -  G(c_\e^{app},r_{\e,i}) \right) ( c_{\e} - c_{\e}^{app}) = -g(r_{\e,i}) (c_\e - c_{\e}^{app})  ( c_{\e} - c_{\e}^{app}) \leq 0
\end{align*}
%
Regarding $I_{5,2}$, we observe that $c_\e^{app} = c +c_{\e,i}^{app}$ on $\Gamma_{\e,i}$. Moreover, by  its definition in \eqref{corrector.c} $c_{\e,i}^{app}$ is constant in $\Gamma_{\e,i}$ and we have for all $x \in \Gamma_{\e,i}$
\begin{align*}
    \left|c(x) + c_{\e,i}^{app}(x) - (c(z_{\e,i}) + H_{\e,i})\right|
    &\le\e^\alpha r_{\e,i} \|\nabla c\|_\infty + 4 \e^2 \frac{|r_{i,\e}|}{\dmin} | H_{\e,i}| \lesssim \e^2.
\end{align*}
Since $G$ is linear in its first argument, we therefore have\footnote{Note that the same convergence order can still be established if $G$ is nonlinear in $c$ as long as it is locally Lipschitz.}
\begin{align} \label{G.H.Lipschitz}
    |G(c_\e^{app}(x),r_{\e,i}) - G\left(c(z_{\e,i}) + H_{\e,i},r_{\e,i}\right) |&=g(r_{\e,i})|c_\e^{app}(x)-c(z_{\e,i})-H_{\e,i}|\\
    &\lesssim \e^2.
\end{align}
By the definition of $H$ in \eqref{H}, we have
\begin{align} \label{G.H.relation}
    G(c(z_{\e,i}) + H_{\e,i},r_{\e,i})) = r_{i,\e}^{-1} H_{\e,i}.
\end{align}
Gathering these observations yields
\begin{align*}
   I_5 \leq  I_{5,2} &\lesssim \e^{-1}  \|c_{\e} - c_{\e}^{app}\|_{L^1(\Gamma_\e)} \\
    &\lesssim  \e^{-1} \e^{\alpha - \frac 3 2}  \|c_{\e} - c_{\e}^{app}\|_{L^2(\Gamma_\e)} \\
    &\lesssim \e^{-1} \e^{2\alpha - 3}  (\|c_{\e} - c_{\e}^{app}\|_{L^2(\Omega_\e)} + \e^{3-\alpha}\|\nabla(c_{\e} - c_{\e}^{app})\|_{L^2(\Omega_\e)}) \\
    &\lesssim  \e^2 \left(1 + \|c_{\e} - c_{\e}^{app}\|_{L^2(\Omega_\e)}^2 +\|\nabla(c_{\e} - c_{\e}^{app})\|_{L^2(\Omega_\e)}^2\right)
\end{align*}
where we used that the surface area of $\Gamma_{\e}$ scales like $\e^{2\alpha -3}$ and we set $\alpha = 3$. 

%
Finally, we turn to $I_6$ and split it into $I_6 = I_{6,1} + I_{6,2} + I_{6,3} + I_{6,4}$, where
\begin{align*}
I_{6,1} &=  4 \pi \e^3 \sum_i \langle  r_{i,\e} H_{\e,i} (\delta_{\partial B_{\frac {\dmin \e}{4}}(z_{\e,i})} - \delta_{B_{\frac {\dmin \e}{4}}(z_{\e,i})}), \mathcal E_\e (c_{\e} - c_{\e}^{app}) \rangle ,\\
I_{6,2} &=  4 \pi \e^3 \sum_i \langle  r_{i,\e} (H_{\e,i} - H(c,r_{\e,i}))\delta_{B_{\frac {\dmin \e}{4}}(z_{\e,i})}, \mathcal E_\e (c_{\e} - c_{\e}^{app}) \rangle,\\
     I_{6,3} &=    \int_{\Omega\setminus \Omega_\e} \int_0^\infty 4 \pi\sigma H(c,\sigma) f(\cdot,\d\sigma) \mathcal E_\e (c_{\e} - c_{\e}^{app}) \\
     I_{6,4} &=     4 \pi \e^3 \sum_i \langle  r_{i,\e}  H(c,r_{\e,i})\delta_{B_{\frac {\dmin \e}{4}}(z_{\e,i})}, \mathcal E_\e  (c_{\e} - c_{\e}^{app}) \rangle - \int_{\Omega} \int_0^\infty 4 \pi\sigma H(c,\sigma) f(\cdot,\d\sigma) (\mathcal E_\e  (c_{\e} - c_{\e}^{app})).
\end{align*}
Here, $\mathcal E_\e$ denotes the extension operator from Lemma \ref{lemma:extension_operators} with $A_i = \Id$ for all $1\leq i \leq N_\e$. 

The first three terms are easily estimated: For $I_{6,1}$, one observes by scaling and Poincar\'e inequality that for any $R>0$, $x \in \R^3$ and $\psi \in H^1(B_R(x))$, we have
\begin{align} \label{Poincare.averages.0}
   \left| \fint_{B_R(x)} \psi  \dd y - \fint_{\partial B_R(x)} \psi \dd y \right| \leq   \fint_{\partial B_R(x)} \left|\psi  - \fint_{B_R(x)} \psi \dd y\right| \dd z  \lesssim R^{-\frac 1 2} \|\nabla \psi\|_{L^2(B_R(x))}.
\end{align}
Hence, using the estimates from Lemma \ref{lemma:extension_operators}, 
\begin{align*}
    I_{6,1} &\leq C \e^{\frac 5 2} \sum_i \|\nabla \mathcal E_\e (c_{\e} - c_{\e}^{app})\|_{L^2(B_{\frac {\dmin \e}{4}}(z_{\e,i}))} \\
    &\leq C \e \|\nabla \mathcal E_\e (c_{\e} - c_{\e}^{app})\|_{L^2(\Omega)} \\
    &\leq C_\delta \e^2 + \delta  \|\nabla (c_{\e} - c_{\e}^{app})\|_{L^2(\Omega_\e)}^2.
\end{align*}
Lipschitz continuity of $H$ and $c$ yields
\begin{align*}
    I_{6,2} \lesssim \e \| \mathcal E_\e (c_{\e} - c_{\e}^{app})\|_{L^1(\Omega)}^2 \lesssim \e^2 +   \|c_{\e} - c_{\e}^{app}\|_{L^2(\Omega_\e)}^2 + \e^3  \|\nabla (c_{\e} - c_{\e}^{app})\|_{L^2(\Omega_\e)}^2.
\end{align*}
Furthermore, by H\"older's inequality using the smallness of the measure $|\Omega\setminus \Omega_\e| \lesssim \e^{6}$ for $\alpha = 3$
\begin{align*}
    I_{6,3} \lesssim \e^3 \| \mathcal E_\e (c_{\e} - c_{\e}^{app})\|_{L^2(\Omega)} \lesssim \e^6 +  \| c_{\e} - c_{\e}^{app}\|_{L^2(\Omega_\e)}^2 +  \e^3  \|\nabla (c_{\e} - c_{\e}^{app})\|_{L^2(\Omega_\e)}^2.
\end{align*}
Finally, we turn to $I_{6,4}$.
We rewrite
\begin{align*}
    I_{6,4} =  4\pi \int_{\Omega \times [0,\infty)} \mathcal E_\e(c_{\e} - c_{\e}^{app}) \sigma H(c,\sigma) (\dd f_\e^{app}(t)-\dd f(t)) 
\end{align*}
where we remind the definition of $f^{app}_\e$ from \eqref{f^app}. 
We would like to  show an estimate of the form 
\begin{align*}
    I_{6,4} \lesssim \|\nabla(c_{\e} - c_{\e}^{app})\|_{L^2(\Omega_\e)}  \mathcal W_{2}(f_\e^{app},f).
\end{align*}

It is classical that the $\mathcal W_1$ distance, which is controlled by the $\mathcal W_2$ distance is just the dual Lipschitz norm, i.e. that for two probability densities   $f_1,f_2 \in L^\infty(U)$ for some convex set $U$
\begin{align*}
    \int_U \varphi \dd (f_1 - f_2) \leq  \|\nabla \varphi\|_{L^\infty(U)} \mathcal W_{1}(f_1,f_2).
\end{align*}
However, such an estimate is not  sufficient for our purposes since we only control the Dirichlet energy of $c_\e - c_{\e}^{app}$.
On the other hand, it is well known (see e.g. \cite[Lemma 5.33]{Santambrogio}) that one has
\begin{align*}
    \int \varphi \dd (f_1 - f_2) \leq  \|\nabla \varphi\|_2 \sqrt{\|f_1\|_\infty + \|f_2\|_\infty} \mathcal W_{2}(f_1,f_2).
\end{align*}
Clearly $f_\e^{app}$ is not even absolutely continuous with respect to the Lebesgue measure due to the Dirac measures in the distribution of the radii. However its marginal with respect to the space variable $\rho_\e^{app} := \int_{[0,\infty)} f_\e^{app} \dd \sigma$ is indeed uniformly bounded in $L^\infty$. Since, we only test with Lipschitz functions with respect to the $\sigma$ variable, one could therefore hope to relate $\mathcal W_2(f_\e,f)$ to the norm $\|f_\e^{app}(t) - f(t)\|_{(H^1_xW_\sigma^{1,\infty})^\ast}$. 
It seems not obvious how to generalize the proof of \cite[Lemma 5.33]{Santambrogio} in this way, though. 
We therefore exploit that the spatial evolution $f_\e,f$ is trivial to prove the following slightly weaker estimate.

\begin{lemma} \label{f^app,f}
    We have 
    \begin{align*}
        \|f_\e^{app}(t) - f(t)\|_{(H^1_xW_\sigma^{1,\infty})^\ast} \leq C (\|f_\e^{app}(0) - f_0\|_{(H^1_xW_\sigma^{1,\infty})^\ast}  + \sqrt{\eta(t)}  + \e).
    \end{align*}
\end{lemma}
\begin{remark} \label{rem:improved.f.f^app}
    We actually prove the refined estimate
    \begin{align*}
        \langle f_\e^{app}(t) - f(t), \varphi \rangle &\lesssim \|f_\e^{app}(0) - f_0\|_{(H^1_xW_\sigma^{1,\infty})^\ast} \|\varphi\|_{H^1_xW_\sigma^{1,\infty}}  \\
        &+ (\sqrt{\eta(t)} + \e) \|\nabla_\sigma \varphi\|_{L^2(\R^3;L^\infty([0,\infty))}. 
    \end{align*}
\end{remark}

Before we proof this lemma, we apply it with $\varphi = \mathcal E_\e(c_{\e} - c_{\e}^{app}) \sigma H(c,\sigma)$ to estimate $I_{6,4}$. To be precise, we first use that $f(t)$ and $f^{app}_\e(t)$ are compactly supported in $\overline \Omega \times [0,\infty)$ thanks to the upper bounds for the radii $r_{\e,i}$ which allows us to consider a truncated version of  $\mathcal E_\e(c_{\e} - c_{\e}^{app}) \sigma H(c,\sigma)$ in the $\sigma$ variable, and may assume that it is defined in all of $\R^3$ with respect to the $x$ variable. Hence,
\begin{align*}
     I_{6,4} &\leq C (\|f_\e^{app}(0) - f_0\|_{(H^1_xW_\sigma^{1,\infty})^\ast} + \sqrt{\eta}  + \e)\| \mathcal E_\e (c_{\e} - c_{\e}^{app})\|_{H^1(\Omega)} \\
     &\leq   C_\delta (\|f_\e^{app}(0) - f_0\|_{(H^1_xW_\sigma^{1,\infty})^\ast}^2 + \eta +  \e^2) + \|c_{\e} - c_{\e}^{app}\|_{L^2(\Omega_\e)}^2 + \delta \|\nabla(c_{\e} - c_{\e}^{app})\|_{L^2(\Omega_\e)}^2.
\end{align*}
In total, we get
\begin{align*}
    I_6 \leq C_\delta (\eta +  \e^2+ \|f_\e^{app}(0) - f_0\|_{(H^1_xW_\sigma^{1,\infty})^\ast}^2) + \|c_{\e} - c_{\e}^{app}\|_{L^2(\Omega_\e)}^2 + \delta \|\nabla(c_{\e} - c_{\e}^{app})\|_{L^2(\Omega_\e)}^2.
\end{align*}

Gathering the estimates for $I_k$, $k=1,\dots,6$ and choosing $\delta$ sufficiently small, the proof of Proposition \ref{pro:chemistry} is complete.

\begin{proof}[Proof of Lemma \ref{f^app,f}]
    We first construct a transport plan for $f_\e^{app}(t) - f(t)$. To this end, let us assume that the map $H_\e$ from \eqref{H_e} satisfies
    \begin{align*}
        H_\e(t,x,\sigma) = H_\e(t,z_{\e,i},\sigma) \qquad \text{for all } x \in B_{\frac{\dmin \e}{4}}(z_{\e,i})
    \end{align*}
    such that also the flow $\Sigma_\e$ in \eqref{eq:floweps} satisfies
        \begin{align*}
        \Sigma_\e(t,x,\sigma) = \Sigma_\e(t,z_{\e,i},\sigma) \qquad \text{for all } x \in B_{\frac{\dmin \e}{4}}(z_{\e,i}).
    \end{align*}
    Then we have $f_\e^{app}(t) = (P_x, \Sigma_\e(t))_\# f_\e^{app}(0)$.
    Let
\begin{align*}
    \tilde \pi_t = (\mathrm{Id},\mathcal T_\e, P_\sigma)_\#  f_\e^{app}(t),
\end{align*}
where $P_\sigma(x,\sigma) = \sigma$ and $\mathcal T_\e \colon \Omega \to \Omega$ is any measurable map that satisfies
\begin{align*}
   \mathcal  T_\e  = z_{\e,i} \quad \text {in } B_{\frac {\dmin \e}{4}}(z_{\e,i}).
\end{align*}
Then $\tilde \pi_t$ is a transport plan for $(f_{\e}^{app}(t), f_\e(t))$ with
\begin{align*}
    \int_{(\R^3 \times [0,\infty))^2} |x_1 - x_2|^2 + |\sigma_1 - \sigma_2|^2 \tilde \pi_t 
    &= \int_{(\R^3 \times [0,\infty))^2} |x_1 - \mathcal T_\e(x_1)|^2 \dd f_{\e}^{app}(t) 
    \lesssim \e^2,
\end{align*}
since $|x_1 - \mathcal T_\e(x_1)| \lesssim  \e$ in the support of $f_\e^{app}(t)$.

By the gluing lemma, there exists $\gamma_0 \in \mathcal (\R^3 \times [0,\infty))^3 $ such that $P^{1,2}_\# \gamma_0 = \tilde \pi_0$, $P^{2,3}_\# \gamma_0 = \pi_0$, where $P^{1,2}$ and $P^{2,3}$ denotes the projection to the first two (pairs of) coordinates and the last two (pairs of) coordinates, respectively.
Hence, $\pi_{0}^{app} :=P^{1,3}_\# \gamma_0 $ is a transport plan for $(f_{\e}^{app}(0), f_0)$.  
Moreover, 
$$
\pi_t^{app}:= ( P^1_x, \Sigma_\e(t)\circ P^1, P^2_x,  \Sigma(t)\circ P^2)\# \pi^{app}_0
$$
is a transport plan for $(f_{\e}^{app}(t), f(t))$. 
Indeed, we have
\begin{align*}
    \int_{(\R^3 \times [0,\infty))^2} \varphi(x_1,\sigma_1) \dd \pi_t^{app}(x_1,\sigma_1,x_2,\sigma_2) &=   \int_{(\R^3 \times [0,\infty))^2} \varphi(x_1,\Sigma_\e(t,x_1,\sigma_1)) \dd \pi_0^{app}(x_1,\sigma_1,x_2,\sigma_2) \\
    &= \int_{(\R^3 \times [0,\infty))^2} \varphi(x_1,\Sigma_\e(t,x_1,\sigma_1)) \dd f_\e^{app}(0, x_1,\sigma_1) \\
    &= \int_{(\R^3 \times [0,\infty))^2} \varphi(x_1,\Sigma(t,x_1,\sigma_1)) \dd  ((P_x, \Sigma_\e(t))_\#  f_\e^{app})(0, x_1,\sigma_1) \\
    &= \int_{(\R^3 \times [0,\infty))^2} \varphi(x_1,\sigma_1)) \dd   f_\e^{app}(t,x_1,\sigma_1)
\end{align*}
and similarly one shows that the second marginal of $\pi_t^{app}$ is $f(t)$.
Analogously, one shows that 
\[
\pi_t^{\mathrm{app}}=P^{1,3}_\# \gamma_t,\qquad \pi_t = P^{2,3}_\# \gamma_t,  \qquad
    \tilde \pi_t = P^{1,2}_\# \gamma_t,
\]
where
\[
\gamma_t := \bigl( P_x^1,\,
\Sigma_\e(t)\circ P^1,\,
P_x^2,\,
\Sigma_\e(t)\circ P^2,\,
P_x^3,\,
\Sigma(t)\circ P^3
\bigr)_\# \gamma
\]
In particular, 
\begin{align*}
    \int_{(\R^3 \times [0,\infty))^2} |x_1 - x_3|^2 + |\sigma_1 - \sigma_3|^2  \dd \pi_t^{app}  &= \int_{(\R^3 \times [0,\infty))^3} |x_1 - x_3|^2 + |\sigma_1 - \sigma_3|^2  \dd \gamma_t \\
    &\leq2\int_{(\R^3 \times [0,\infty))^3} |x_1 - x_2|^2 + |\sigma_1 - \sigma_2|^2  \dd \gamma_t  \\
    &\qquad + 2\int_{(\R^3 \times [0,\infty))^3} |x_2 - x_3|^2 + |\sigma_2 - \sigma_3|^2  \dd \gamma_t\\
    &= 2\int_{(\R^3 \times [0,\infty))^2} |x_1 - x_2|^2 + |\sigma_1 - \sigma_2|^2  \dd \tilde \pi_t  \\
    &\qquad + 2\int_{(\R^3 \times [0,\infty))^2} |x_2 - x_3|^2 + |\sigma_2 - \sigma_3|^2  \dd \pi_t  \\
    &\lesssim \e^2 + \eta.
\end{align*}

    Let  $\varphi \in H^1_xW_\sigma^{1,\infty}$ with $ \|\varphi\|_{H^1_xW_\sigma^{1,\infty}} \leq 1$. Then
    \begin{align*}
        \langle \varphi, f^{app}_\e(t) - f(t) \rangle &= \int_{(\R^3 \times [0,\infty))^2} (\varphi(x_1,\sigma_1) - \varphi(x_2,\sigma_2)) \dd \pi^{app}_t \\
        &= \int_{(\R^3 \times [0,\infty))^2} (\varphi(x_1,\Sigma_\e(t,x_1,\sigma_1)) - \varphi(x_2,\Sigma(t,x_2,\sigma_2))) \dd \pi^{app}_0
    \end{align*}
    which implies
    \begin{multline*}
        \left| \langle \varphi, f^{app}_\e(t) - f(t) \rangle\right|
        \le
        \left|\int_{(\R^3 \times [0,\infty))^2} (\varphi(x_1,\Sigma(t,x_1,\sigma_1)) - \varphi(x_2,\Sigma(t,x_2,\sigma_2))) \dd \pi^{app}_0\right| \\
        +\left|\int_{(\R^3 \times [0,\infty))^2} (\varphi(x_1,\Sigma_\e(t,x_1,\sigma_1)) - \varphi(x_1,\Sigma(t,x_1,\sigma_1))) \dd \pi^{app}_0\right|
    \end{multline*}
    Since $\Sigma(t)$ is Lipschitz, the first right-hand side term is estimated by 
    \begin{align*}
        \Biggl|\int_{(\R^3 \times [0,\infty))^2} (\varphi(x_1,\Sigma(t,x_1,\sigma_1)) - \varphi(x_2,\Sigma(t,x_2,\sigma_2))) \dd \pi^{app}_0\Biggr| &= \left|\langle \varphi \circ (P_x,\Sigma(t)), f^{app}_\e(0) - f(0)\rangle\right| \\
        &\lesssim  \|f_\e^{app}(0) - f_0\|_{(H^1_xW_\sigma^{1,\infty})^\ast}.
    \end{align*}
    For the second right-hand side term, we have
    \begin{align*}
       & \left|\int_{(\R^3 \times [0,\infty))^2} (\varphi(x_1,\Sigma_\e(t,x_1,\sigma_1)) - \varphi(x_1,\Sigma(t,x_1,\sigma_1))) \dd \pi^{app}_0\right| \\
        &\leq \int_{(\R^3 \times [0,\infty))^2} \|\partial_\sigma \varphi(x_1,\cdot)\|_{L^\infty([0,\infty))} |\Sigma_\e(t,x_1,\sigma_1) - \Sigma(t,x_1,\sigma_1)| \dd \pi^{app}_0 \\
        &\leq  \left(\int_{(\R^3 \times [0,\infty))^2} \|\partial_\sigma \varphi(x_1,\cdot)\|_{L^\infty([0,\infty))}^2 \dd \pi^{app}_0 \right)^{\frac 1 2}  \left(\int_{(\R^3 \times [0,\infty))^2} |\Sigma(t,x_1,\sigma_1)) - \Sigma_\e(t,x_1,\sigma_1))|^2  \dd \pi^{app}_0\right)^{\frac 1 2}
    \end{align*}
    with 
    \begin{align*}
        \int_{(\R^3 \times [0,\infty))^2} \|\partial_\sigma \varphi(x_1,\cdot)\|_{L^\infty([0,\infty))}^2 \dd \pi^{app}_0 =  \int_{\R^3} \|\partial_\sigma \varphi(x_1,\cdot)\|_{L^\infty([0,\infty))}^2 \dd \rho^{app}_0(x_1) \leq \|\rho^{app}_0\|_{L^\infty(\R^3)} \lesssim 1
    \end{align*}
    and using again that $\Sigma$ is Lipschitz
    \begin{align*}
        &\int_{(\R^3 \times [0,\infty))^2} |\Sigma(t,x_1,\sigma_1)) - \Sigma_\e(t,x_1,\sigma_1))|^2  \dd \pi^{app}_0 \\
        &\lesssim  \int_{(\R^3 \times [0,\infty))^2} |\Sigma(t,x_1,\sigma_1)) - \Sigma(t,x_2,\sigma_2))|^2  \dd \pi^{app}_0 \\
        &\qquad+ 2\int_{(\R^3 \times [0,\infty))^2} |\Sigma(t,x_2,\sigma_2) - \Sigma_\e(t,x_1,\sigma_1)|^2  \dd \pi^{app}_0 \\
        &\lesssim \int_{(\R^3 \times [0,\infty))^2} |x_1 - x_2|^2 + |\sigma_1 - \sigma_2|^2  \dd \pi^{app}_0 \\
        &\qquad+ \int_{(\R^3 \times [0,\infty))^2} |\sigma_1 - \sigma_2|^2  \dd \pi^{app}_t \\
        &\lesssim \eta(t)  + \e^2   
    \end{align*}
    Collecting these estimates finishes the proof.
\end{proof}

\subsection{Proof of Proposition \ref{pro:eta}: Estimate of \texorpdfstring{$\eta$}{eta}}\label{ssec:estimate_eta}

We compute the time derivative of $\eta$. Note that due to the supremum in time in the definition of $\eta$ in \eqref{eq:def_eta}, $\eta(t)$ is of the form $\eta(t) = \sup_{s\le t} \zeta(s)$. The function $\zeta$ is continuously differentiable, hence $\eta$ is Lipschitz, and its (weak) derivative satisfies 
\begin{align*}
    \frac{\d}{\d t} \eta(t) = \begin{cases}
    \frac{\d}{\d t}  \zeta(t) &\qquad \text{if } \eta(t) =\zeta(t) \text{ and } \frac{\d}{\d t}  \zeta(t) >0, \\
    0 &\qquad \text{otherwise.} 
        \end{cases}
\end{align*}
In particular $\frac{\d}{\d t} \eta(t) \leq  
   \left| \frac{\d}{\d t}  \zeta(t) \right| $ and hence
\begin{align*}
    \frac{\d}{\d t} \eta(t) &\leq \left|\frac{\d}{\d t} \int_{(\Omega \times (0,\infty))^2} (|x_1-x_2|^2+|\Sigma_\e(t,x_1,\sigma_1)-\Sigma(t,x_2,\sigma_2)|^2)  \pi_t (\d x_1, \d \sigma_1, \d x_2, \d \sigma_2) \right| \\
    & = 2 \left| \int_{(\Omega \times (0,\infty))^2} (\Sigma_\e(t,x_1,\sigma_1)-\Sigma(t,x_2,\sigma_2)) \right. \\
    & \qquad \qquad \left.\left( H_\e(t,x_1,\Sigma_\e(t,x_1,\sigma_1))  - \frac {H(c(t,x_2),\Sigma(t,x_2,\sigma_2))}{\Sigma(t,x_2,\sigma_2)} \right) \pi_0 (\d x_1, \d \sigma_1, \d x_2, \d \sigma_2) \right|\\
    & \leq 2 \eta^{1/2} \left(\int_{(\Omega \times (0,\infty))^2}\left| H_\e(t,x_1,\sigma_1) - \frac {H(c(t,x_2),\sigma_2)}{\sigma_2} \right|^2 \pi_t (\d x_1, \d \sigma_1, \d x_2, \d \sigma_2) \right)^{\frac 1 2}
\end{align*}
We split 
\begin{align*}
 \int_{(\Omega \times (0,\infty))^2}\left| H_\e(t,x_1,\sigma_1) - \frac {H(c(t,x_2),\sigma_2)}{\sigma_2} \right|^2 \pi_t (\d x_1, \d \sigma_1, \d x_2, \d \sigma_2) \lesssim I_1 + I_2,
\end{align*}
where
\begin{align*}
    I_1 &:= \int_{(\Omega \times (0,\infty))^2}\left|\frac {H(c(t,x_1),\sigma_1)}{\sigma_1} - \frac{H(c(t,x_2),\sigma_2)}{\sigma_2}\right|^2 \pi_t (\d x_1, \d \sigma_1, \d x_2, \d \sigma_2), \\
    I_2 &:= \int_{(\Omega \times (0,\infty))^2}\left| H_\e(t,x_1,\sigma_1) - \frac {H(c(t,x_1),\sigma_1)}{\sigma_1} \right|^2 \pi_t (\d x_1, \d \sigma_1, \d x_2, \d \sigma_2).
\end{align*}
Recall that by Theorem \ref{theorem:maximum_principle} $r_{\eps,i} \leq \bar \xi$. Hence, integration with respect to $\sigma$ can be restricted to $[0,\bar \xi]$. Since $(x,\sigma) \mapsto \frac {H(c(t,x),\sigma)}{\sigma}$ is Lipschitz for $\sigma \in [0,\bar \xi]$ (as can easily be seen by its definition in \eqref{H}) to which we can restrict all integration, 
\begin{align*}
    I_1 \lesssim \eta(t).
\end{align*}
We compute
\begin{align*}
    I_2 = \frac 1 {N_\e} \sum_i \left|\frac 1 {r_{\e,i}} H(c(t,z_{\e,i}),r_{\e,i}) -  \fint_{\Gamma_{\e,i}(t)}G\left(c_\e(t,x),r_{\e,i}(t)\right) \dd x \right|^2
\end{align*}
This term can be treated very similarly as \eqref{I_5.1}--\eqref{I_5.2} above: We estimate for all $x \in \Gamma_{\e,i}$
\begin{align*}
 &\left| \frac 1 {r_{\e,i}} H(c(t,z_{\e,i}),r_{\e,i}) -G\left(c_\e(t,x),r_{\e,i}(t)\right) \right|  \\
 &\lesssim \left| \frac 1 {r_{\e,i}} H(c(t,z_{\e,i}),r_{\e,i}) -G\left(c_\e^{app}(t,x),r_{\e,i}(t)\right) \right| +\left| G\left(c_\e(t,x),r_{\e,i}(t)\right) - G\left(c_\e^{app}(t,x),r_{\e,i}(t)\right) \right| \\
 &\lesssim \e^2 + |c_\e^{app}(x) - c_\e(x)| 
\end{align*}
where the last estimate follows from \eqref{G.H.Lipschitz}--\eqref{G.H.relation} and the Lipschitz continuity of $G$ on $[0,r_{\max}]$.
Hence, we deduce thanks to Lemma \ref{lemma:trace_estimates} (ii)
\begin{align*}
    I_2 &\lesssim \e^{-3}\|c_{\e} - c_{\e}^{app}\|^2_{L^2(\Gamma_\e)} + \e^4 \\
    &\lesssim \|c_{\e} - c_{\e}^{app}\|^2_{L^2(\Omega_\e)} +\|\nabla(c_{\e} - c_{\e}^{app})\|^2_{L^2(\Omega_\e)} + \e^4
\end{align*}

We conclude
\begin{align*}
    \frac {\dd}{\dd t} \eta \leq 2 \eta^{\frac 1 2} (I_1 + I_2)^{\frac 1 2} \leq C \eta + \|c_{\e} - c_{\e}^{app}\|^2_{L^2(\Omega_\e)} +\frac 1 2 \|\nabla(c_{\e} - c_{\e}^{app})\|^2_{L^2(\Omega_\e)} + \e^4 
\end{align*}
as claimed.

\section{Passage to the limit in the supercritical case \texorpdfstring{$\alpha \in (1,3)$}{alpha less than 3}}\label{sec:hom_supercrit}

In this section, we prove the quantitative homogenization result for the supercritical regime $\alpha\in(1,3)$ stated in \cref{thm:hom.supercritical}.
The proof follows the same general strategy as in the critical case $\alpha=3$ and we therefore only detail the modifications required by the different scaling.
On the one hand, the interfacial energy is now subcritical, so no boundary-layer correction of the concentration is needed  and the microscopic exchange law $G$ enters the limit problem directly.
However, on the other hand, the fluid dynamics converge to the Darcy-type system \eqref{Darcy}, which provides weaker boundary information than the Brinkman system.
In particular, the no-slip condition $u_\e=0$ on $\Sigma_2$ is lost in the limit where only normal outflow is prohibited via $u\cdot n=0$.
This prevents us from using \(u_\e-u\) directly as an admissible test function in the microscopic Stokes problem and necessitates an additional boundary-layer correction near \(\Sigma_2\).

As before, we construct a transport plan $\pi_t$ between $f(t)$ and $f_\e(t)$ but now we consider the characteristics to the limit equation \eqref{evolution.f.supercritical}.
Instead of \eqref{rel.energy}, we define the relative energy without any corrector for the concentration, i.e.
\begin{align*}
    E_\e(t):=  \|c_{\e} - c\|_{L^2(\Omega_\e(t))}^2 + \eta(t) 
\end{align*}
and replace Propositions \ref{pro:fluid}--\ref{pro:eta} by the following statements.
These statements already imply \cref{thm:hom.supercritical} and are proved in the following sections.
\begin{proposition} \label{pro:fluid.super}
Under the assumptions of Theorem \ref{thm:hom.supercritical}, for $\e < \e_0$
    \begin{align*}
    \| u_\e - u\|_{L^2(\Omega_\e(t))}^2 \lesssim  \e^{\alpha - 1} + \e^{\frac{3-\alpha}2} + \eta +  \e^{\alpha -3}\|f_\e^{app}(0) - f_0\|_{(H^1_xW_\sigma^{1,\infty})^\ast}^2 + \|h_\e - h\|_{L^2({\Omega})}^2 .
\end{align*}
\end{proposition} 

\begin{proposition} \label{pro:chemistry.super}
Under the assumptions of Theorem \ref{thm:hom.supercritical}, for $\e < \e_0$
\begin{multline*}
    \frac{\d}{\d t } \frac 1 2 \|c_{\e} - c\|_{L^2(\Omega_\e(t))}^2+ \|\nabla(c_{\e} - c)\|_{L^2(\Omega_\e(t))}^2 \\\lesssim \e^{3-\alpha} + \e^{3(\alpha-1)}+ \|f_\e^{app}(0) - f_0\|_{(H^1_xW_\sigma^{1,\infty})^\ast}^2 + \eta + \|c_{\e} - c\|_{L^2(\Omega_\e(t))}^2 + \|u_\e - u\|_{L^2(\Omega_\e(t))}^2.
\end{multline*}
\end{proposition} 

\begin{proposition} \label{pro:eta.super}
Under the assumptions of Theorem \ref{thm:hom.supercritical}, for $\e < \e_0$
    \begin{align*}
    \frac{\d}{\d t } \eta  \leq \frac 1 2 \|\nabla (c_{\e} - c)\|_{L^2(\Omega_\e(t))}^2 + C(\eta + \e^2 +   \|c_{\e} - c\|_{L^2(\Omega_\e(t))}^2 ) .
\end{align*}
\end{proposition} 

\begin{proof}[Proof of Theorem \ref{thm:hom.supercritical}]
    Combining the above propositions, we have
 \begin{multline*}
     \frac{\d}{\d t } E_\e(t) + \frac 1 2 \|\nabla (c_{\e} - c)\|_{L^2(\Omega_\e(t))}^2 \\
     \lesssim E_\e(t) +\e^{\frac{3-\alpha}2} +  \e^{\alpha - 1}  +  \e^{\alpha-3}\|f_\e^{app}(0) - f_0\|_{(H^1_xW_\sigma^{1,\infty})^\ast}^2 + \|h_\e - h\|_{L^2(\Omega)}^2.
 \end{multline*}
 Hence, by Gronwall's inequality
 \begin{multline*}
     E_\e(t) + \|\nabla (c_{\e} - c)\|_{L^2(0,t;L^2(\Omega_\e(t)))}^2 \\\lesssim E_\e(0) +  \e^{\frac{3-\alpha}2}+  \e^{\alpha - 1}  +  \e^{\alpha-3}\|f_\e^{app}(0) - f_0\|_{(H^1_xW_\sigma^{1,\infty})^\ast}^2 + \|h_\e - h\|_{L^2(\Omega)}^2.
 \end{multline*}
Inserting the definition of $E_\e$ yields \eqref{quantitative.est.supercritical}. 
\end{proof}

\subsubsection*{Estimate of $u_\e - u$}
Note that we again drop the time dependency in the notation; in particular, we write $\Omega_\e$ instead of $\Omega_\e(t)$.
Note that unlike in the critical case,  $u_\e^{app} - u_\e$ with $u_\e^{app}$ as in \eqref{eq:ueps_app} is not an admissible test function for the equation of $u_\e$ anymore since $u \neq 0$ on $\Sigma_2$ in general. We therefore invoke a boundary layer corrector as in \cite{BalaziAllaireOmnes25, HoferLuOschmann26}.

With the assumptions of \cref{thm:hom.supercritical}, there exists a $C^{3,\mu}$ domain $U \subset \R^3$ such that $\Omega \subset U$ and $ \Sigma_2 \subset \partial U$ and $\bar u$ a  $C^{1,\mu}$ extension of $u$ to $U$ such that $\bar u \cdot n = 0$ on $\partial U$.
We then invoke \cite[Lemma~4.10]{BalaziAllaireOmnes25} to obtain $A \in L^\infty(S;C^{2,\mu}(U;\R^{3}))$ such that $\curl A = \bar u$ and $A = 0$ on $\partial U$.  

For $\varpi > 0$, let $U^\varpi = \{ x \in U : \dist(x, \partial U) < \varpi \}$ and $\theta_\varpi \in C^2(U)$ be such that $\theta_\varpi = 1$ on $U^{\varpi /2}$, $\theta_\varpi = 0$ in $U \setminus U^\varpi$, and
\begin{align*}
0\leq \theta_\varpi\leq 1, \qquad  \|\nabla \theta_\varpi\|_{L^\infty(U)} \lesssim \varpi^{-1}, \qquad \|\nabla^2 \theta_\varpi\|_{L^\infty(U)} \lesssim \varpi^{-2}.
\end{align*}

Set $u_\varpi = {\rm curl}(\theta_\varpi A)$, 
and replace the definition of $u_\e^{app}$ from \eqref{eq:ueps_app} by
\begin{align}\label{eq:ueps_app.super}
    u_{\e}^{app} := w_\e (u-u_\varpi) - \mathcal B_\e(u-u_\varpi) + u_{\e}^{inh}.
\end{align}
and note that $u_{\e}^{app} = 0$ on $\Sigma_2$ such that $u_\e^{app} - u_\e$ is an admissible testfunction for the equation of $u_\e$.

Then, testing the weak forms for System~\eqref{system:stokes_eps} and System~\eqref{Darcy} with $u_\e - u_\e^{app}$:
\begin{align*}
    \e^{3-\alpha} \|\nabla( u_\e - u_\e^{app})\|_{L^2(\Omega_\e)}^2 = \sum_{j=1}^5 \bar I_j,
\end{align*}
where
\begin{align*}
   \bar I_1 &=   \int_{\Omega_\e} (h_\e - w_\e h) \cdot (u_\e - u_\e^{app}), \\
   \bar I_2 &= - \e^{3-\alpha}\int_{\Omega_\e} 2 (\nabla (w_\e - {\rm I}_3) \nabla u) + w_\e \Delta u)) \cdot (u_\e - u_\e^{app}), \\
  \bar  I_3 &= - \e^{3-\alpha} \int_{\Omega_\e} \nabla (w_\e u_\varpi + \mathcal B_\e (u - u_\varpi)) :  \nabla( u_\e - u_\e^{app}), \\
        \bar I_4 &= - \e^{3-\alpha}\int_{\Omega_\e} \nabla u_\e^{inh}:  \nabla( u_\e - u_\e^{app}), \\
      \bar I_5 &= -\langle \e^{3-\alpha} \Delta w_\e u + 6 \pi u\int_0^\infty \sigma f(\cdot,\d\sigma), u_\e - u_\e^{app} \rangle ,\\
      \bar I_6 &= \e^{3-\alpha}\int_{\Sigma_1} \partial_n (w_\e u) \cdot ( u_\e^{app} - u_\e)
\end{align*}
We estimate using the Poincar\'e inequality from Lemma \ref{lemma:poincare}, $h\in L^\infty(\Omega)$, and \cref{lem:correctors}
\begin{align*}
    \bar I_1 &\lesssim \|w_\e -{\rm I}_3\|_{L^2(\Omega_\e)} \|h\|_{L^\infty(\Omega_\e)}  \|u_\e - u_\e^{app}\|_{L^2(\Omega_\e)} + \|h_\e - h\|_{L^2(\Omega)} \|u_\e - u_\e^{app}\|_{L^2(\Omega_\e)} \\
    &\lesssim \e^{3-\alpha} \delta \|\nabla (u_\e - u_\e^{app})\|_{L^2(\Omega_\e)}^2 + C_\delta \e^{2\alpha -2 } + C_\delta\|h_\e - h\|_{L^2(\Omega)}^2
\end{align*}
Next, integrating by parts in the first term of $\bar I_2$ and using the boundary conditions, we obtain via \cref{lem:correctors}
\begin{align*}
    \bar I_2 &\lesssim \e^{3-\alpha} \|w_\e - {\rm I}_3\|_{L^2(\Omega_\e)}\left( \| \Delta u\|_{L^\infty(\Omega_\e)} \|u_\e - u_\e^{app}\|_{L^2(\Omega_\e)} +
   \| \nabla  u\|_{L^\infty(\Omega_\e)} \|\nabla (u_\e - u_\e^{app})\|_{L^2(\Omega_\e)}\right) \\
    &\qquad + \e^{3-\alpha} \|w_\e \Delta u\|_{L^2(\Omega_\e)} \|u_\e - u_\e^{app}\|_{L^2(\Omega_\e)} \\
    & \lesssim C_\delta \e^{3-\alpha} \|u\|_{W^{2,\infty}}^2 \|w_\e - {\rm I}_3\|_{L^2(\Omega_\e)}^2 + \e^{6-2\alpha} \|\Delta u\|_{L^2(\Omega_\e)}^2 + \delta \e^{3-\alpha} \|\nabla (u_\e - u_\e^{app})\|_{L^2(\Omega_\e)}^2 \\
    &\lesssim C_\delta (\e^{6-2 \alpha} + \e^{\alpha +1}) +\delta \e^{3-\alpha} \|\nabla (u_\e - u_\e^{app})\|_{L^2(\Omega_\e)}^2.
\end{align*}
Regarding $\bar I_3$, we use the estimate
\begin{align*}
    \|\nabla (w_\e u_\varpi + \mathcal B_\e (u - u_\varpi))\|_{L^2(\Omega_
    \e)} \lesssim \varpi^{1/2} \e^{\frac{\alpha - 3}{2}} + \varpi^{-\frac 1 2}
\end{align*} 
which is shown as in \cite[p. 23]{HoferLuOschmann26}. Hence,
\begin{align*}
    \bar I_3 \lesssim C_\delta \e^{3-\alpha} \left(\varpi \e^{\alpha - 3} + \varpi^{- 1 }\right) + \delta \e^{3-\alpha} \|\nabla( u_\e - u_\e^{app})\|_{L^2(\Omega_\e)}^2.
\end{align*}
Further,
\begin{align*}
       |\bar I_4| \leq  C_\delta  \e^{3-\alpha} \| \nabla u_\e^{inh}\|_{L^2(\Omega_\e)}^2  + \delta \e^{3-\alpha}\|\nabla( u_\e - u_\e^{app})\|_{L^2(\Omega_\e)}^2 \lesssim C_\delta \e^{2 \alpha}  + \delta \e^{3-\alpha}\|\nabla( u_\e - u_\e^{app})\|_{L^2(\Omega_\e)}^2.
\end{align*}
By Lemma \ref{lem:M_eps}
\begin{align*}
    \bar I_5 &\lesssim  \left(\e^{\alpha - 1} + \eta^{\frac 1 2} \right) \|u_\e - u_\e^{app}\|_{L^2(\Omega_\e)} \\
    &\qquad + (\e + \|f_\e^{app}(0) - f_0\|_{(H^1_xW_\sigma^{1,\infty})^\ast}) \| u_\e - u_\e^{app}\|_{H^1(\Omega_\e)}  \\
    & \lesssim  C_{\delta}\left(\eta+ \e^{\alpha-1} + \e^{\alpha -3}\|f_\e^{app}(0) - f_0\|_{(H^1_xW_\sigma^{1,\infty})^\ast}^2\right) + \delta \e^{3-\alpha}\| \nabla(u_\e - u_\e^{app})\|^2_{L^2(\Omega_\e)}.
\end{align*}
Finally, since $w_\e = {\rm I}_3$ on $\Sigma_1$, a trace estimate and the Poincar\'e inequality yield
\begin{align*}
     \bar I_6 = \e^{3-\alpha} \int_{\Sigma_1} \partial_n u \cdot ( u_\e^{app} - u_\e) \lesssim C_\delta \e^{3-\alpha} + \e^{3-\alpha} \delta \|\nabla(u_\e^{app} - u_\e)\|^2_{L^2(\Omega_\e)}.
\end{align*}
Choosing $\varpi = \e^{\frac{3-\alpha}2}$, we conclude
\begin{multline*}
   \e^{3-\alpha} \|\nabla (u_\e - u_\e^{app})\|_{L^2(\Omega_\e)}^2 +  \|u_\e - u_\e^{app}\|_{L^2(\Omega_\e)}^2 \\ \lesssim \e^{\alpha - 1} + \e^{\frac{3-\alpha}2} + \eta  + \e^{\alpha -3}\|f_\e^{app}(0) - f_0\|_{(H^1_xW_\sigma^{1,\infty})^\ast}^2 +  \|h_\e - h\|_{L^2(\Omega)}^2.
\end{multline*}

To conclude the proof, we estimate
\begin{align*}
    \|u_\e - u\|_{L^2(\Omega_\e)}^2 \lesssim \|u_\e - u_\e^{app}\|_{L^2(\Omega_\e)}^2 + \|u -  u_\e^{app}\|_{L^2(\Omega_\e)}^2
\end{align*}
and 
\begin{align*}
     \|u -  u_\e^{app}\|_{L^2(\Omega_\e)}^2 &\lesssim  \|u_\e^{inh}\|_{L^2(\Omega_\e)}^2
     + \|w_\e u_\varpi + \mathcal B_\e (u - u_\varpi)\|_{L^2(\Omega_\e)}^2
     + \|({\rm I}_3-w_\e)u\|_{L^2(\Omega_\e)}^2\\
      &\lesssim \e^{3(\alpha -1)}+  \e^{\frac{3-\alpha}2}+\e^{2(\alpha -1)},
\end{align*}
where the estimate for $w_\e u_\varpi + \mathcal B_\e (u - u_\varpi)$ can be found again in in \cite[p. 23]{HoferLuOschmann26}.

\subsubsection*{Estimate of $\|c_{\e} - c\|_{L^2(\Omega_\e)}^2 $}
This is largely analogous to the corresponding arguments in \cref{ssec:estimate_ceps_c}.
Indeed, in the same way as we have obtained \eqref{Gronwall.c.0}, we get\footnote{To keep the notation consistent with \cref{ssec:estimate_ceps_c}, we retain the numbering of the corresponding error terms.
Since no concentration corrector is required in the present regime, there is no analogue of the term \(I_1\), and we therefore start with \(\bar I_2\).}
\begin{align} \label{Gronwall.c.0.sup}
     \frac 1 2 \frac{\d}{\d t}  \|c_{\e} - c\|_{L^2(\Omega_\e)}^2  +   \|\nabla( c_{\e} - c)\|_{L^2(\Omega_\e)}^2   =  \sum_{j=2}^5 \bar I_j
\end{align}
where
\begin{align*}
    \bar I_2 &=  \int_{\Omega_\e}( u_\e c_\e - u c) \cdot  \nabla (c_{\e} - c), \\
    \bar I_3 &= \sum_i  \int_{\Gamma_{\e,i}} \e^\alpha\partial_t r_{\e,i} c_\e (c_{\e} - c)  - \frac 1 2 \e^{\alpha} \sum_i  \int_{\Gamma_{\e,i}} \partial_t r_{\e,i} |c_{\e} - c|^2 - \int_{ \Gamma_\e} ( u c -  \nabla c ) \cdot \nu (c_{\e} - c)\\
    \bar I_4 &= \int_{\Omega_\e} ( F( c_\e) - F(c))  (c_{\e} - c ) \\
    \bar I_5 &= \sum_i    \int_{\Gamma_{\e,i}}  \e^{3-2\alpha} G(c_\e,r_{\e,i}) (c_{\e} - c)  -    4 \pi\int_{\Omega_\e} \int_0^\infty \sigma^2 G(c,\sigma) f(\cdot,\d\sigma) (c_{\e} - c) 
\end{align*}
The terms  $\bar I_j$, $j=2,3,4$, are estimated as in the previous section. Regarding $I_5$, we split
\begin{align*}
   \bar I_5 &= \sum_i   \int_{\Gamma_{\e,i}}  \e^{3-2\alpha} (G(c_\e,r_{\e,i}) - G(c,r_{\e,i})) (c_{\e} - c)\di{\gamma} \\  &+   4 \pi\int_{\Omega} \int_0^\infty  \sigma^2 G(c,\sigma)  \mathcal E_\e (  c_{\e} - c) \dd (\bar f^{app}_\e -   f^{app}_\e) \\ 
   &+4 \pi \int_{\Omega\setminus \Omega_\e} \int_0^\infty \sigma^2 G(c,\sigma)  f(\cdot,\d\sigma) \mathcal E_\e (  c_{\e} - c)\di{x} \\
    &+     4 \pi\int_{\Omega} \int_0^\infty \sigma^2 G(c,\sigma)  \mathcal E_\e (  c_{\e} - c) \dd (   f^{app}_\e -   f) \\
    &=:  \bar I_{5,1} + \bar I_{5,2} +  \bar I_{5,3} +  \bar I_{5,4}
\end{align*}
where $f^{app}_\e$ is defined as in \eqref{f^app} and 
\begin{align} \label{bar.f^app}
   \bar f_\e^{app}(t) := \frac 1 {N_\e} \sum_i \delta_{\Gamma_{\e,i}} \otimes \delta_{r_{\e,i}}.
\end{align}
Since $G$ is Lipschitz in the first variable, we get by the trace estimate in Lemma, for $\e$ sufficiently small
\ref{lemma:trace_estimates} (ii)
\begin{align*}
   | \bar I_{5,1}| \leq C \e^{3-2\alpha} \|c_\e - c\|^2_{L^2(\Gamma_\e)} \leq  \delta  \|\nabla(c_\e - c)\|^2_{L^2(\Omega_\e)} + C_\delta \|c_\e - c\|^2_{L^2(\Omega_\e)}.
\end{align*}
Furthermore, by H\"older's inequality using $|\Omega\setminus \Omega_\e| \lesssim \e^{3(\alpha-1)}$, $\int_0^\infty \sigma^2 G(c,\sigma)  f(\cdot,\d\sigma)\in L^\infty(S\times\Omega)$, and \cref{lemma:extension_operators}
\begin{align*}
    \bar I_{5,3} \lesssim \e^{3(\alpha-1)/2} \| \mathcal E_\e (c_{\e} - c)\|_{L^2(\Omega)} \lesssim \e^{3(\alpha-1)} +  \| c_{\e} - c\|_{L^2(\Omega_\e)}^2 +  \e^{2\alpha}  \|\nabla (c_{\e} - c)\|_{L^2(\Omega_\e)}^2.
\end{align*}
Moreover,  $ \bar I_{5,4}$ is estimated with the help of Lemma \ref{f^app,f} as the term $I_{6,4}$ before, i.e.
\begin{align*}
     \bar I_{5,4} \leq C_\delta (\|f_\e^{app}(0) - f_0\|_{(H^1_xW_\sigma^{1,\infty})^\ast}^2 + \eta+  \e^2) + \|c_{\e} - c\|_{L^2(\Omega_\e)}^2 + \delta \|\nabla(c_{\e} - c)\|_{L^2(\Omega_\e)}^2.
\end{align*}
We set
\[
h_i=G(c,r_{\e,i})\mathcal E_\e(c_{\e} - c)
\]
and use the definitions of $\bar f_\e^{app}$ in \eqref{bar.f^app} and  $f_\e^{app}$ in \eqref{f^app} to write 
\begin{align*}
    \bar I_{5,2}=   4\pi  \e^3 \sum_i\left( \fint_{\Gamma_{\e,i}} h_i\di{\gamma}-   \fint_{B_{\frac {\dmin \e}{4}}(z_{\e,i})}h_i\di{x}\right).
\end{align*}
By the Poincar\'e type inequality \eqref{poincare.type.averages} and since the balls $B_{\frac {\dmin \e}{4}}(z_{\e,i})$ are disjoint, we have
\begin{align} 
     \e^3 \sum_i \left|\fint_{\Gamma_{\e,i}} h_i\di\gamma - \fint_{B_{\frac {\dmin \e}{4}}(z_{\e,i})} h_i\di x \right| 
     \lesssim\e^{\frac{3-\alpha}2}\left(\sum_i \|\nabla h_i\|^2_{L^2(B_{\frac {\dmin \e}{4}}(z_{\e,i}))}\right)^\frac12
     \lesssim\e^{\frac{3-\alpha}2}\|c_\e-c\|_{H^1(\Omega_\e)}
\end{align}
Here, we have also used the estimates for the extension $\mathcal E_\e$ (\cref{lemma:extension_operators}).
Via Young's inequality, we get
\begin{align*}
    |\bar I_{5,2}| \lesssim  C_\delta \e^{3-\alpha} + \delta\|\nabla (c_{\e} - c)\|_{L^2(\Omega_\e)}^2 + \|c_{\e} - c\|_{L^2(\Omega_\e)}^2.
\end{align*}
Combining all estimates proves \cref{pro:chemistry.super}.

\subsubsection*{Estimate of $\eta$}
We compute 
\begin{align*}
    \frac{\d}{\d t} \eta(t) &\leq  \left|\frac{\d}{\d t} \int_{(\Omega \times (0,\infty))^2} (|x_1-x_2|^2+|\Sigma_\e(t,x_1,\sigma_1)-\Sigma(t,x_2,\sigma_2)|^2)  \pi_t (\d x_1, \d \sigma_1, \d x_2, \d \sigma_2) \right|\\
    & = 2 \left|\int_{(\Omega \times (0,\infty))^2} (\Sigma_\e(t,x_1,\sigma_1)-\Sigma(t,x_2,\sigma_2)) \right. \\
    &\qquad \left. \phantom{\int_{(\Omega \times (0,\infty))^2}} \left(H_\e(t,x_1,\Sigma_\e(t,x_1,\sigma_1)) -  G(c(t,x_2),\Sigma(t,x_2,\sigma_2))  \right) \pi_0 (\d x_1, \d \sigma_1, \d x_2, \d \sigma_2) \right| \\
    & \leq 2 \eta^{1/2} \left(\int_{(\Omega \times (0,\infty))^2}\left| H_\e(t,x_1,\sigma_1)  -G(c(t,x_2),\sigma_2) \right|^2 \pi_t (\d x_1, \d \sigma_1, \d x_2, \d \sigma_2) \right)^{\frac 1 2}
\end{align*}
where $H_\e$ is given by \eqref{H_e}.
We split 
\begin{align*}
 \int_{(\Omega \times (0,\infty))^2}\left| H_\e(t,x_1,\sigma_1)  -G(c(t,x_2),\sigma_2) \right|^2 \pi_t (\d x_1, \d \sigma_1, \d x_2, \d \sigma_2) \lesssim I_1 + I_2,
\end{align*}
where
\begin{align*}
    I_1 &:= \int_{(\Omega \times (0,\infty))^2}\left|G(c(t,x_1),\sigma_1) - G(c(t,x_2),\sigma_2)\right|^2 \pi_t (\d x_1, \d \sigma_1, \d x_2, \d \sigma_2), \\
    I_2 &:= \int_{(\Omega \times (0,\infty))^2}\left| H_\e(t,x_1,\sigma_1)  -G(c(t,x_1),\sigma_1) \right|^2 \pi_t (\d x_1, \d \sigma_1, \d x_2, \d \sigma_2).
\end{align*}
Since $(x,\sigma) \mapsto G(c(t,x),\sigma)$ is Lipschitz (up to a truncation at large $\sigma$)
we have 
\begin{align*}
    I_1 \lesssim \eta(t).
\end{align*}
We compute
\begin{align} \label{I_2.eta.sup}
    I_2 = \frac 1 {N_\e} \sum_i \left|G(c(t,z_{\e,i}),r_{\e,i}) -  \fint_{\Gamma_{\e,i}(t)}G\left(c_\e(t,x),r_{\e,i}(t)\right) \dd x \right|^2
\end{align}
We deduce thanks to the local Lipschitz continuity of $G$ and Lipschitz continuity of $c$ and to Lemma \ref{lemma:trace_estimates} (ii), for $\e $ sufficiently small
\begin{align*}
    I_2 &\leq \e^{3 - 2\alpha}\|c_{\e} - c\|^2_{L^2(\Gamma_\e)} + \e^2 \\
    &\lesssim  C \|c_{\e} - c\|^2_{L^2(\Omega_\e)} + \e^{3-\alpha} \|\nabla(c_{\e} - c)\|^2_{L^2(\Omega_\e)} + \e^2.
\end{align*}
Combining this with the estimate for $I_1$ and using Young's inequality, we directly obtain the claimed estimate for sufficiently small $\e$, which proves \cref{pro:eta.super}.

\appendix 
\section{Coordinate transform}\label{appen:trafo}
During our analysis, we rely on certain properties and estimates for the coordinate transform.
Since the statements and proofs are somewhat technical but neither surprising nor particularly insightful, we collect them here.
In the following, let $S=(0,T)$ ($T>0$), $M>0$, and $\e>0$.
For notational brevity, we write
$\|\cdot\|_\infty$ for the $L^\infty$-norm over the respective domain throughout this appendix, whenever the underlying domain is clear from the context.

The transformation is defined in \cref{eq:transform_psi} via
\begin{equation}
\label{eq:diffeomorphism_psieps}
\Psi_{\e}(t)\colon\overline{\Omega_\e}
\to\overline{\Omega_\e(t)},\quad
\Psi_{\e}(t;y)
=y+
\e^\alpha
\sum_{i\in\mathcal I_\e}
\bigl(r_i(t)-r_{\e,i,0}\bigr)
\chi\left(1+
\frac{|y-z_{\e,i}|-\e^\alpha r_{\e,i,0}}
{\e^\alpha(2\lambda-r_{\e,i,0})}
\right)
\frac{y-z_{\e,i}}{|y-z_{\e,i}|}.
\end{equation}
We introduce the shorthand notation
\[
s_{\e,i}(y)
:=1+
\frac{|y-z_{\e,i}|-\e^\alpha r_{\e,i,0}}
{\e^\alpha(2\lambda-r_{\e,i,0})},
\qquad
\nu_{\e,i}(y)
:=
\frac{y-z_{\e,i}}{|y-z_{\e,i}|},\qquad 
P_{\e,i}(y)
:=
{\rm I}_3-\nu_{\e,i}(y)\otimes\nu_{\e,i}(y).
\]
We note that $\nu_{\e,i}\otimes \nu_{\e,i}$ and $P_{\e,i}$ are complementary orthogonal projections; in particular $P_{\e,i}(\nu_{\e,i}\otimes \nu_{\e,i})=(\nu_{\e,i}\otimes \nu_{\e,i})P_{\e,i}=0$ and 
\(
P_{\e,i}+(\nu_{\e,i}\otimes \nu_{\e,i})={\rm I}_3.
\)
Moreover, we set
\[
a_{\e,i}(t,y)
:=
\e^\alpha
\frac{r_i(t)-r_{\e,i,0}}{|y-z_{\e,i}|}
\chi\bigl(s_{\e,i}(y)\bigr),\qquad
b_{\e,i}(t,y)
:=
\frac{r_i(t)-r_{\e,i,0}}
{2\lambda-r_{\e,i,0}}
\chi'\bigl(s_{\e,i}(y)\bigr).
\]
We also recall the quantities defined via the transformation $\Psi_\e$:
\[
F_\e := \nabla \Psi_\e^{T},
\qquad
J_\e := \det F_\e,
\qquad
A_\e := J_\e F_\e^{-1},
\qquad
D_\e := J_\e F_\e^{-1}F_\e^{-T},
\]
and
\[
J_\e^\Gamma := J_\e \lvert F_\e^{-T}\nu_\e\rvert
\qquad\text{on }\Gamma_\e.
\]

\begin{lemma}[Estimates for the transformation]
\label{lemma:transform_estimates}
Let $\lambda>\overline r$, $\e\le\e_\lambda$, and let $r\in\mathcal R_\e(S)$ satisfy
\[
\frac1\lambda\le r_i(t)\le \lambda,
\qquad
|\partial_t r_i(t)|\le M
\quad\text{for all } i\in\mathcal I_\e,\ t\in S.
\]
Then,
\begin{align*}
\e^{-\alpha}\|\Psi_\e-\ID\|_\infty
+\e^{-\alpha}\|\Psi_\e^{-1}-\ID\|_\infty
&\lesssim
\|r-r_{\e,0}\|_{L^\infty(S;\ell^\infty(\mathcal I_\e))},
\\
\|F_\e-{\rm I}_3\|_\infty+\|F_\e^{-1}-{\rm I}_3\|_\infty+\|J_\e-1\|_\infty+\|A_\e-{\rm I}_3\|_\infty
+\|D_\e-{\rm I}_3\|_\infty
&\lesssim_{\lambda}
\|r-r_{\e,0}\|_{L^\infty(S;\ell^\infty(\mathcal I_\e))},
\\
\e^{-\alpha}\|\partial_t\Psi_\e\|_\infty
+\|\partial_tF_\e\|_\infty+\|\partial_tJ_\e\|_\infty+\|\partial_tA_\e\|_\infty
&\lesssim_\lambda
\|\partial_t r\|_{L^\infty(S;\ell^\infty(\mathcal I_\e))},\\
\|\nabla J_\e^{-1}\|_\infty&\lesssim_{\lambda}\e^{-\alpha}\|r-r_{\e,0}\|_{L^\infty(S;\ell^\infty(\mathcal I_\e))}
\end{align*}
where ${\rm I}_3$ denotes the unit matrix and $\ID$ the identity transformation.
Moreover, for $\phi\in H^1(\Omega_\e)^3$,
\begin{align*}
\|e_\e(\phi)\|_{L^2(\Omega_\e)}
&\lesssim_\lambda
\|\nabla \phi\|_{L^2(\Omega_\e)}.
\end{align*}
Finally,
\[
F_\e^{-1}(t,x)F_\e^{-T}(t,x)\xi\cdot\xi\ge
    \min\left\{
\frac1{\lambda^4},
\frac19
\right\}
    |\xi|^2,
\qquad
J_\e(t,x)\ge \frac{1}{\lambda^4(2\lambda^2-1)},
\]
uniformly in $\xi\in\mathbb R^3$ and $(t,x)\in S\times\Omega_\e$.
\end{lemma}
\begin{proof}
Very similar estimates were established via the same scaling considerations, see \cite{eden_thermo-elasticity_2026,gahn_homogenization_2023,wiedemann_homogenisation_2023}; but since our case is slightly different in its scaling and setup we present the main calculations.

    Note that $\Psi_{\e}(t;\cdot)=\ID$ on $\Omega\setminus \mathcal U_\e$ so all estimates are trivially satisfied here.
    The estimate for $\Psi_\e-\ID$ follows immediately from the definition \cref{eq:diffeomorphism_psieps}.
Differentiating $\Psi_\e$ in time yields
\[
\e^{-\alpha}\partial_t\Psi_{\e,i}(t;y)
=
\partial_t r_i(t)
\eta_{\e,i}(y)
\frac{y-z_{\e,i}}{|y-z_{\e,i}|}.
\]
From which the estimate for $\partial_t\Psi_\e$ follows.
For the inverse transformation holds (with $x=\Psi_{\e,i}(t,y)$)
\[
\Psi_{\e,i}^{-1}(t;x)-x=y-\Psi_{\e,i}(t;y),
\]
which shows the estimate for $\Psi_\e^{-1}-\ID$.

Next, a simple calculation shows that the Jacobian $F_\e=D\Psi_\e$ takes the form (for $y\in\mathcal U_{\e,i})$
    \begin{align*}
    D\Psi_\e(t;y)&=(1+a_{\e,i}(t,y))P_{\e,i}(y)+(1+b_{\e,i}(t,y))\nu_{\e,i}(y)\otimes \nu_{\e,i}(y)\\
    &={\rm I}_3+a_{\e,i}(t,y)P_{\e,i}(y)+b_{\e,i}(t,y)\nu_{\e,i}(y)\otimes \nu_{\e,i}(y).
    \end{align*}
    Hence, the quantities $1+a_{\e,i}$ and $1+b_{\e,i}$ are the tangential and radial eigenvalues of the Jacobian $F_\e=D\Psi_\e$ over $\mathcal U_{\e,i}$.
    Since $P_{\e,i}$ and $\nu_{\e,i}\otimes\nu_{\e,i}$ are complementary orthogonal projections,
    \[
    |D\Psi_\e(t;y)|
    =
    \begin{cases}
    1,
    &
    y\in\Omega_\e\setminus\mathcal U_\e^\lambda,
    \\
    \displaystyle
    \max\left\{
    |1+a_{\e,i}(t,y)|,
    |1+b_{\e,i}(t,y)|
    \right\},
    &y\in\mathcal U_{\e,i}^\lambda.
    \end{cases}
    \]
    With 
    \[
    |a_{\e,i}(t,y)|
    \le
    \frac{|r_i(t)-r_{\e,i,0}|}{r_{\e,i,0}}
    \le
    \lambda|r_i(t)-r_{\e,i,0}|,
    \]
    and, since $2\lambda-r_{\e,i,0}\ge\lambda$,
    \[
    |b_{\e,i}(t,y)|
    \le
    \frac{2}{\lambda}
    |r_i(t)-r_{\e,i,0}|,
    \]
    we thus get
    \[
    |F_\e(t,y)-{\rm I}_3|=|D\Psi_\e(t;y)-{\rm I}_3|\le
    \max\left\{
    \lambda,\frac2\lambda
    \right\}
    |r_i(t)-r_{\e,i,0}|
    \]
For the tangential eigenvalue, we have
\[
1+a_{\e,i}(t,y)
\ge
\min\left\{
1,\frac{r_i(t)}{r_{\e,i,0}}
\right\}
\ge
\frac1{\lambda^2}
\]
since
\[
|y-z_{\e,i}|
\ge
\e^\alpha r_{\e,i,0},
\]
For the radial eigenvalue, if $r_i(t)\le r_{\e,i,0}$, then
$b_{\e,i}(t,y)\ge0$ and trivially \(1+b_{\e,i}(t,y)\ge1\).
If $r_i(t)\ge r_{\e,i,0}$, then, using $-2\le\chi'\le0$ and
$r_i(t)\le\lambda$, we obtain
\[
1+b_{\e,i}(t,y)\ge1-2\frac{\lambda-r_{\e,i,0}}{2\lambda-r_{\e,i,0}}
=\frac{r_{\e,i,0}}{2\lambda-r_{\e,i,0}}\ge\frac1{2\lambda^2-1}.
\]
Consequently,
\[
1+a_{\e,i}(t,y)\ge\frac1{\lambda^2},
\qquad
1+b_{\e,i}(t,y)\ge\frac1{2\lambda^2-1}.
\]
The inverse Jacobian is explicitly given by
\begin{equation}
\label{eq:jacobian_inverse}
[D\Psi_\e(t;y)]^{-1}
=
\begin{cases}
{\rm I}_3,
&
y\in\Omega_\e\setminus\mathcal U_\e^\lambda,
\\[0.6em]
\displaystyle
\frac{1}{1+a_{\e,i}(t,y)}P_{\e,i}(y)
+
\frac{1}{1+b_{\e,i}(t,y)}
\nu_{\e,i}(y)\otimes\nu_{\e,i}(y),
&
y\in\mathcal U_{\e,i}^\lambda,\quad i\in\mathcal I_\e.
\end{cases}
\end{equation}
Equivalently, on $\mathcal U_{\e,i}^\lambda$,
\[
[D\Psi_\e(t;y)]^{-1}
=
{\rm I}_3
-
\frac{a_{\e,i}(t,y)}
{1+a_{\e,i}(t,y)}
P_{\e,i}(y)
-
\frac{b_{\e,i}(t,y)}
{1+b_{\e,i}(t,y)}
\nu_{\e,i}(y)\otimes\nu_{\e,i}(y).
\]
It follows that
\[
\left|[D\Psi_\e(t;y)]^{-1}\right|
=
\begin{cases}
1,
&
y\in\Omega_\e\setminus\mathcal U_\e^\lambda,
\\[0.4em]
\displaystyle
\max\left\{
\frac1{1+a_{\e,i}(t,y)},
\frac1{1+b_{\e,i}(t,y)}
\right\},
&
y\in\mathcal U_{\e,i}^\lambda.
\end{cases}
\]
Using the lower bounds above, we conclude that
\[
|F_\e^{-1}(t,y)|=\left|[D\Psi_\e(t;y)]^{-1}\right|
\le
\max\left\{
\lambda^2,2\lambda^2-1
\right\}
=
2\lambda^2-1
\]
and
\[
|F_\e^{-1}(t,y)-{\rm I}_3|\lesssim_\lambda|r_i(t)-r_{\e,i,0}|.
\]
For the determinant,
    \[
    J_\e=\det F_\e=(1+a_{\e,i})^2(1+b_{\e,i}),
    \]
    we obtain the estimate for $J_\e-1$. Moreover,
    \[
    J_\e(t,x)
    =
    (1+a_{\e,i}(t,x))^2(1+b_{\e,i}(t,x))
    \ge
    \frac{1}{\lambda^4(2\lambda^2-1)}.
    \]
    Next, we use $P_{\e,i}\left(\nu_{\e,i}\otimes\nu_{\e,i}\right)=0$, $P_{\e,i}^2=P_{\e,i}$, and $(\nu_{\e,i}\otimes\nu_{\e,i})^2=\nu_{\e,i}\otimes\nu_{\e,i}$ to see that 
    \[
    F_\e^{-1}F_\e^{-T}
    =\left(\frac{1}{1+a_{\e,i}}P_{\e,i}+\frac{1}{1+b_{\e,i}}\nu_{\e,i}\otimes\nu_{\e,i}\right)^2
    =\frac{1}{(1+a_{\e,i})^2}P_{\e,i}+\frac{1}{(1+b_{\e,i})^2}\nu_{\e,i}\otimes\nu_{\e,i}.
    \]
With
\[
1+a_{\e,i}\le\lambda^2,
\qquad
1+b_{\e,i}\le3,
\]
we obtain
\begin{equation}
F_\e^{-1}F_\e^{-T}\xi\cdot\xi
=
\frac{|P_{\e,i}\xi|^2}{(1+a_{\e,i})^2}
+
\frac{|(\nu_{\e,i}\otimes\nu_{\e,i})\xi|^2}
{(1+b_{\e,i})^2}
\ge
\min\left\{
\frac1{\lambda^4},
\frac19
\right\}
|\xi|^2.
\end{equation}
With $\nabla J_\e^{-1}=-J_\e^{-2}\nabla J_\e$ and 
    \[
    \nabla J_\e=2(1+a_{\e,i})(1+b_{\e,i})\nabla a_{\e,i}+(1+a_{\e,i})^2\nabla b_{\e,i}
    \]
    we get the $\e^{-\alpha}$ estimate for $\nabla J_\e$.
    The corresponding estimates for $A_\e$ and $D_\e$ follow directly from
    \[
    A_\e=J_\e F_\e^{-1},
    \qquad
    D_\e=J_\e F_\e^{-1}F_\e^{-T},
    \]
    together with the uniform bounds on $F_\e^{-1}$ and $J_\e$.
    Finally, the estimate for $e_\e(\phi)$ follows via the estimates for $J_\e$ and $F_\e^{-1}$ and the estimates for the time derivatives can now be calculated directly.
\end{proof}


\begin{lemma}[Lipschitz estimates for the transformation]
\label{lemma:transform_lipschitz}
Let $\lambda>\overline r$, $\e\le\e_\lambda$, and $M>0$. For $j=1,2$, let $r^{(j)}\in\mathcal R_{\e}(S)$ satisfy
\[
\frac1\lambda\le r_i^{(j)}(t)\le \lambda,
\qquad
|\partial_t r_i^{(j)}(t)|\le M
\quad
\text{for all } i\in\indexI,\ t\in S.
\]
Then
\begin{equation*}
\e^{-\alpha}\|\delta\Psi_{\e}\|_{\infty}
+\|\delta F_\e\|_{\infty}+\|\delta F_\e^{-1}\|_{\infty}+\|\delta A_\e\|_{\infty}
+\|\delta J_\e\|_{\infty}
+\|\delta J_\e^\Gamma\|_{\infty}
+\|\delta D_\e\|_\infty
\lesssim_\lambda\|\delta r \|_{L^\infty(S;\ell^\infty(\indexI))}.
\end{equation*}
where $\delta$ denotes the difference of the corresponding functions, e.g., $\delta r=r^{(1)}-r^{(2)}$.
Moreover,
\begin{align*}
\e^{-\alpha}\|\delta\partial_t\Psi_{\e}\|_{_{\infty}}
+\|\delta\partial_tF_\e\|_{\infty}
&\lesssim_\lambda
\|\delta\partial_t r\|_{L^\infty(S;\ell^\infty(\mathcal I_\e))},
\\
\|\delta\partial_tJ_\e\|_{\infty}
&\lesssim_{\lambda}
\|\delta\partial_t r\|_{L^\infty(S;\ell^\infty(\mathcal I_\e))}
+
M\|\delta r\|_{L^\infty(S;\ell^\infty(\mathcal I_\e))}
\end{align*}
Finally,
\[
\|\delta e_\e(\phi)\|_{L^2(\Omega_\e)}
\lesssim_\lambda
\|\delta r\|_{\ell^\infty(\mathcal I_\e)}
\|\nabla\phi\|_{L^2(\Omega_\e)}\qquad (\phi\in H^1(\Omega_\e)^3,\ t\in S).
\]
Here $\delta e_\e(\phi)=e_\e^{(1)}(\phi)-e_\e^{(2)}(\phi)$.
\end{lemma}

\begin{proof}
The assertions follow from the explicit representations of
$\Psi_\e$, $F_\e$, $F_\e^{-1}$, and $J_\e$ established in the proof of \cref{lemma:transform_estimates}, together with the uniform upper and lower bounds for the eigenvalues of $F_\e$.
\end{proof}

The following corollary follows immediately as the transformation is only active on $\mathcal U_\e$ with $|\mathcal U_\e|\sim\e^{3(\alpha-1)}\lambda^3$:
\begin{corollary}\label{cor:transform_lp_estimates}
    For any $p\in[1,\infty)$ and any $r\in\mathcal{R}_\e(S)$ with $\nicefrac1\lambda\leq r_i\le \lambda$, it holds
    \[
    \|J_\e-1\|_{L^\infty(S;L^p(\Omega_\e))}^p
    +\|F_\e-\ID\|_{L^\infty(S;L^p(\Omega_\e))}^p\\
    +\|F_\e^{-1}-\ID\|_{L^\infty(S;L^p(\Omega_\e))}^p
    \lesssim_\lambda
    \e^{3(\alpha-1)}
    \|r-r_{\e,0}\|_{L^\infty(S;\ell^\infty(\indexI))}^p.
    \]
    Similarly,
    \[
    \|\partial_tJ_\e\|_{L^\infty(S;L^p(\Omega_\e))}^p
    \lesssim_\lambda
    \e^{3(\alpha-1)}
    \|\partial_tr\|_{L^\infty(S;\ell^\infty(\indexI))}^p.
    \]
    Furthermore, the corresponding difference estimates also hold. More precisely, for $r^{(1)},r^{(2)}$ as in \cref{lemma:transform_lipschitz},
    \[
    \|\delta J_\e\|_{L^\infty(S;L^p(\Omega_\e))}^p
    +\|\delta F_\e\|_{L^\infty(S;L^p(\Omega_\e))}^p
    +\|\delta F_\e^{-1}\|_{L^\infty(S;L^p(\Omega_\e))}^p
    \lesssim_\lambda
    \e^{3(\alpha-1)}
    \|\delta r\|_{L^\infty(S;\ell^\infty(\indexI))}^p,
    \]
    and
    \[
    \|\delta\partial_tJ_\e\|_{L^\infty(S;L^p(\Omega_\e))}^p
    \lesssim_{\lambda,M}
    \e^{3(\alpha-1)}
    \left(
    \|\delta\partial_t r\|_{L^\infty(S;\ell^\infty(\indexI))}^p
    +
    \|\delta r\|_{L^\infty(S;\ell^\infty(\indexI))}^p
    \right).
    \]
\end{corollary}

With the transformation $\Psi_\e$, we can freely switch between coordinates.
To fix the notation: For every function $\phi\colon \Omega_\e(t)\to\R$, we introduce the pullback $\Psi_t^*\phi\colon\Omega_\e\to\R$ defined via $\Psi_t^*\phi=\phi\circ\Psi_\e(t,\cdot)$; analogously, for $\phi\colon \Omega_\e\to\R$, we set the pushforward $(\Psi_t)_*\phi\colon  \Omega_\e(t)\to\R$ via $(\Psi_t)_*\phi=\phi\circ\Psi_t^{-1}$.
For the ease of notation, we also use the shorthand $\hat\phi=\Psi_t^*\phi$ where the hat-notation always refer to pullback to the initial coordinates.

\begin{lemma}[Pushforward and pullback]\label{lemma:pushforward_estimates}
Let $\lambda>\overline r$ and let $p\in[1,\infty]$.
   \begin{itemize}
       \item[$(i)$] Let $r\in\mathcal R_\e(S)$ with $\nicefrac1\lambda\le r_i(t)\le\lambda$.
       The pullback $\Psi_t^*\phi$ induces linear isomorphisms $L^p(\Omega_\e)\simeq L^p(\Omega_\e(t))$ and $W^{1,p}(\Omega_\e)\simeq W^{1,p}(\Omega_\e(t))$.
       In both cases the pushforward is its inverse and it holds
       \begin{align*}
        \|\Psi_t^*\phi\|^p_{L^p(\Omega_\e)}&\lesssim_{\lambda}\|\phi\|^p_{L^p(\Omega_\e(t))}\lesssim_{\lambda}\|\Psi_t^*\phi\|^p_{L^p(\Omega_\e)}\quad (\phi\in L^p(\Omega_\e(t))),\\
        \|\nabla \Psi_t^*\phi\|^p_{L^p(\Omega_\e)}&\lesssim_{\lambda}\|\nabla\phi\|^p_{L^{p}(\Omega_\e(t))}\lesssim_{\lambda}\|\nabla\Psi_t^*\phi\|^p_{L^{p}(\Omega_\e)}\quad (\phi\in W^{1,p}(\Omega_\e(t))).
        \end{align*}

        \item[$(ii)$] Let $r^{(1)},r^{(2)}\in \mathcal R_\e(S)$ with $\nicefrac1\lambda\le r^{(j)}(t)\le \lambda$ with corresponding pushforwards $(\Psi_t^*)^{(j)}$.
        Then,
        \begin{align*}
        \|(\Psi_t^*)^{(1)}\phi-(\Psi_t^*)^{(2)}\phi\|^p_{L^p(\Omega_\e)}&\lesssim_{\lambda}\|\phi\|^p_{L^p(\Omega)}\quad(\phi\in L^p(\Omega)),\\
        \|(\Psi_t^*)^{(1)}\phi-(\Psi_t^*)^{(2)}\phi\|^p_{L^p(\Omega_\e)}&\lesssim_{\lambda}\e^{\alpha p}\|\nabla \phi\|^p_{L^p(\Omega)}\|\delta r\|^p_{L^\infty(S)}\quad (\phi\in W^{1,p}(\Omega)).
        \end{align*}
        
   \end{itemize}
\end{lemma}
\begin{proof}
   $(i)$. This is clear from the regularity and uniform estimates for the transformation and can be checked directly via transformation of coordinates, e.g.,
   \begin{align*}
   \|\Psi_t^*\phi\|^p_{L^p(\Omega_\e)}=\int_{\Omega_\e}|\Psi_t^*\phi|^p\di{x}&=\int_{\Omega_\e(t)}J_\e^{-1}(t,\Psi_\e^{-1}(t,x))|\phi(x)|^p\di{x}\\
   &\le\|J_\e^{-1}\|_{L^\infty(\Omega_\e)}\|\varphi\|^p_{L^p(\Omega_\e(t))}
   \le\lambda^4(2\lambda^2-1)\|\phi\|^p_{L^p(\Omega_\e(t))}.
   \end{align*}
    
   $(ii)$. The first $L^p$ estimate follows instantly from the $L^p$ estimates in $(i)$.
   For $\theta\in[0,1]$, we set $r_\theta=(1-\theta)r^{(1)}+\theta r^{(2)}$ and note that $r_\theta\in\mathcal{R}_\e(S)$ with $\nicefrac1\lambda\le r_\theta\le \lambda$ by convexity.
   The corresponding transformation is $\Psi_\theta=(1-\theta)\Psi^{(1)}+\theta\Psi^{(2)}$ which allows us to estimate
   \[
   |(\Psi_t^*)^{(1)}\phi(y)-(\Psi_t^*)^{(2)}\phi(y)|\le |\Psi^{(1)}(y)-\Psi^{(2)}(y)|\int_0^1|\Psi_\theta^*\nabla\phi(y)|\di{\theta}.
   \]
   The claimed estimate follows via $(i)$ and the estimates for $\delta\Psi$ in \cref{lemma:transform_lipschitz}.
\end{proof}
%

\section{Poincar\'e type inequalities and other useful estimates}\label{appen:equalities}
In this section, we collect some technical inequalities and estimates.

We write here $B_s = B_s(0)$ for $s >0$.

\begin{lemma}
    Let $s < R$ and $\varphi \in H^1(B_R)$. Then
\begin{align} \label{Poincare.averages}
    \left|\fint_{B_R} \varphi - \fint_{B_s} \varphi \right| \lesssim s^{-\frac 1 2} \|\nabla \varphi \|_{L^2(B_R)}, \\
    \left|\fint_{B_R} \varphi - \fint_{\partial B_s} \varphi \right| \lesssim s^{-\frac 1 2} \|\nabla \varphi \|_{L^2(B_R)} \label{poincare.type.averages}
\end{align}
\end{lemma}
\begin{proof}
    It suffices to prove the estimate assuming in addition that $\int_{B_R} \varphi = 0$. Then, by Sobolev embedding and the Poincar\'e-Wirtinger inequality on $B_R$, we have
\begin{align*}
    \|\varphi\|_{L^6(B_R)} \lesssim \|\nabla \varphi \|_{L^2(B_R)}
\end{align*}
Hence, by H\"older's inequality
\begin{align*}
    \left|\fint_{B_s} \varphi\right| \leq s^{-3} \|\varphi\|_{L^6(B_R)} \|1\|_{L^{6/5}(B_s)} \lesssim s^{-\frac 1 2} \|\nabla \varphi \|_{L^2(B_R)}.
\end{align*}
This shows the first inequality.

For the second inequality, it suffices to combine the first inequality with
\begin{align*}
    \left|\fint_{B_s} \varphi - \fint_{\partial B_s} \varphi \right| \lesssim s^{-3/2}  \int_{B_s} \left|\varphi - \fint_{\partial B_s} \varphi \right| \lesssim s^{-\frac 1 2} \|\nabla \varphi \|_{L^2(B_R)}
\end{align*}
where the last inequality follows can be shown like the standard Poincar\'e-Wirtinger inequality and scaling considerations.
\end{proof}

\begin{lemma}[Improved Poincaré inequality in $\Omega_\e$]\label{lemma:poincare}
Let $\alpha \in [1,3]$ and assume that \eqref{ass:uniform.covering}. There exists $\e_0 > 0$ such that
\[
\|u\|_{L^2(\Omega_\e)}\lesssim \e^{\frac{3-\alpha}{2}}\|\nabla u\|_{L^2(\Omega_\e)}
\]
for all $u\in H_{\Gamma_\e}^1(\Omega_\e)$ and all $\e\le\e_0$.
\end{lemma}
\begin{remark}
    Let  $\Omega_\e(t)$ be given through \eqref{eq:omega_time} for some given $r\in\mathcal{R}_\e$ (defined in \eqref{mathcalR}) with $r(t)\in[\frac1\lambda,\infty)^{N_\e}$. 
Then, the same inequality applies with a constant and $\e_0$ that depend on $\lambda$.
\end{remark}
\begin{proof}
This inequality is well-known in the homogenization community. For periodic monodisperse holes, we refer to \cite[Lemma 3.4.1]{Allaire1991-fp}

Assumption \ref{ass:uniform.covering} guarantees that the proof from \cite{Allaire1991-fp} generalizes straightforwardly.
Fix a set $U \subset \R^3$ such that $\{x \in \R^3 : \dist(x,\Omega) \leq 1\} \subset U$.  This ensures that $B_{\overline d \e}(z_{\e,i}) \subset U$ for $\e$ small enough.
Then, we extend $u$ by $0$ inside of $\cup_i B_{\e,i}(t)$ and then to a function $E u \in H^1(U)$ with such that 
$\|\nabla E u \|_{L^2(U)} \leq C \|\nabla u\|_{L^2(\Omega_\e)}$. 
Then, as in \cite{Allaire1991-fp}, one shows 
\begin{align} \label{Poincare.allaire}
    \|E u\|_{L^2(B_{\overline d \e}(z_{\e,i})}^2 \lesssim_\lambda \e^{3-\alpha} \|\nabla E u\|_{L^2(B_{\overline d \e}(z_{\e,i})}^2
\end{align}
 and 
 summing over $i \in \mathcal I_\e$  concludes the proof.
In fact, \eqref{Poincare.allaire}  follows immediately from \eqref{Poincare.averages}: Choose $R = \overline d \e$ and $s = \e^\alpha r_{\e,i} $, and denote $\varphi = Eu$ and by $(\varphi)_s$ and $(\varphi)_R$, the averages of $\varphi$ on $B_{s}(z_{\e,i})$ and on $B_{R}(z_{\e,i})$, respectively, and note that $(\varphi)_s = 0$. Hence, by the stadard Poincar\'e-Wirtinger inequality in $B_{R}(z_{\e,i})$
\begin{align*}
    \|\varphi\|_{L^2(B_{R}(z_{\e,i})} &\leq  \|\varphi - (\varphi)_R\|_{L^2(B_{R}(z_{\e,i})} +  \|(\varphi)_R\|_{L^2(B_{R}(z_{\e,i})}\\
    &\lesssim \e \|\nabla  \varphi\|_{L^2(B_{R}(z_{\e,i})} + \e^{\frac 3 2}   \left| (\varphi)_R - (\varphi)_s \right| \\
    &\lesssim \e^{\frac{3-\alpha}2}\|\nabla  \varphi\|_{L^2(B_{R}(z_{\e,i})},
\end{align*}
where we used $\e \leq \e^{\frac{3-\alpha}2}$ in the last estimate since $\alpha > 1$.
\end{proof}

\begin{lemma}[Auxiliary estimates]
\label{lemma:trace_estimates}
\begin{itemize}
\item[$(i)$] Let $s\in(2,6)$. For every $\theta>0$, there is a constant $C_{s,\theta}>0$ such that
\[
\|\phi\|_{L^s(\Omega_\e)}^2\leq C_{s,\theta} \|\phi\|^2_{L^2(\Omega_\e)}+(\theta+\e^{2\alpha} C_{s,\theta})\|\nabla\phi\|^2_{L^2(\Omega_\e)}
\]
 for all $\phi\in H^1(\Omega_\e)$.
\item[$(ii)$] It holds 
\[
 \e^{3-2\alpha}\|\phi\|^2_{L^2(\Gamma_\e)}\lesssim\|\phi\|^2_{L^2(\Omega_\e)}+\e^{3-\alpha}\|\nabla\phi\|_{L^2(\Omega_\e)}^2.
 \]
for all $\phi\in H^1(\Omega_\e)$.
 \item[$(iii)$] There is a constant $C>0$ (independent of $\e$) such that
\[
 \|\phi-\sum_{i \in \indexI}[\phi]_{\Gamma_{\e,i}} \1_{\Gamma_{\e,i}}\|_{L^p(\Gamma_\e)}\leq C\e^{\alpha(1-\frac{1}{p})}\|\nabla\phi\|_{L^p(\Omega_\e)}
 \]
  for all $\phi\in W^{1,p}(\Omega_\e)$ where $[\phi]_{\Gamma_\e,i}$ is defined in \eqref{[]_Gamma}.
 \end{itemize}
\end{lemma}
\begin{proof}
    $(i).$ Let $s\in(2,6)$ and $\vartheta>0$. Using the extension operator in \cref{lemma:extension_operators} with $A_i=\ID$, we find $C_{s,\vartheta}>0$ such that
    \[
    \|\phi\|_{L^s(\Omega_\e)}^2\leq \|\mathcal{E}_{\e,\ID}\phi\|_{L^s(\Omega)}^2
    \leq C_{s,\vartheta} \|\mathcal{E}_{\e,\ID}\phi\|^2_{L^2(\Omega)}+\vartheta\|\nabla(\mathcal{E}_{\e,\ID}\phi)\|^2_{L^2(\Omega)}
    \]
    by making use of the compact embedding of $H^1(\Omega)$ into $L^s(\Omega)$.
    With the estimates in \cref{lemma:extension_operators}, we have
    \[
    \|\phi\|_{L^s(\Omega_\e)}^2\leq C\left(C_{s,\vartheta}\|\phi\|_{L^2(\Omega_\e)}^2+(\vartheta+\e^{2\alpha} C_{s,\vartheta})\|\nabla\phi\|^2_{L^2(\Omega_\e)}\right)
    \]
    With $\theta=C\vartheta$ and $C_{s,\theta}=CC_{s,\vartheta}$, the claim follows.
    
    $(ii).$ This follows from the trace estimate in \cite[Lemma 2.3]{allaire_homogenization_1990} for $\alpha=3$.
    Although stated there for periodically distributed holes of a common size, the proof is local and applies verbatim to our uniformly separated balls using the uniform upper and lower bounds on \(r_{\e,i,0}\).
        
    $(iii).$
    As the holes are well-separated, this trace Poincaré estimate can be shown by a standard scaling argument. More precisely, one shows by scaling that 
\[
\|\phi-[\phi]_{\Gamma_{\e,i}}\|_{L^p(\Gamma_{\e,i})}\leq C\e^{\alpha(1-\frac{1}{p})}\|\nabla\phi\|_{L^p(B_{2 \e^{\alpha} r_{\e,i}}(z_{\e,i}) \setminus B_{\e,i})}
 \]
 holds for any $i \in \mathcal I_\e$ and then sums over $i\in\indexI$.
\end{proof}

For any matrix-valued function $A\in L^\infty(\Omega_\e)^{3\times3}$ and $\phi\in W^{1,p}(\Omega)^3$, we introduce the notation 
\[
e_{\e,A}(\phi)=\frac{1}{2}\left(A\nabla\phi+(A\nabla\phi)^T\right).
\]
which simplifies to $e_{\e,A}(\phi)=e(A\phi)$ for constant $A\in\R^{3\times3}$.
Note that for $A=F_\e^{-1}$, we write $e_{\e,A}(\phi)=e_\e(\phi)$ for the ease of notation and in accordance with \cref{symmetric_gradient_reference}.

\begin{lemma}[Extension operators]
\label{lemma:extension_operators}
For any $A=(A_i)\subset(\R^{3\times3})^{N_\e}$ with $\det A_i>0$ for all $i\in\indexI$, there exists an extension operator $\mathcal{E}_{\e,A}\colon W^{1,p}(\Omega_\e)^3\to W^{1,p}(\Omega)^3$ such that
\begin{align*}
\|\mathcal{E}_{\e,A}\phi\|_{L^p(\Omega)}&\lesssim\max_i\left(|A_i||A_i^{-1}|\right)\left(\|\phi\|_{L^p(\Omega_\e)}+\e^{\alpha}\|\nabla\phi\|_{L^p(\Omega_\e)}\right),\\
\|\nabla(\mathcal{E}_{\e,A}\phi)\|_{L^p(\Omega)}&\lesssim\max_i\left(|A_i||A_i^{-1}|\right)\|\nabla\phi\|_{L^p(\Omega_\e)},\\
\|e_{\e,A}(\mathcal{E}_{\e,A}\phi)\|_{L^p(\Omega)}&\lesssim\|e_{\e,A}(\phi)\|_{L^p(\Omega_\e)}.
\end{align*}
\end{lemma}

\begin{proof}
    This proof is quite similar to \cite[Proposition 4.3]{gahn_extension_2024} (where our differences are the $\e$-scaling and the non-periodicity of the domain) and done in two steps:
\\[2mm]
    \textit{Step 1: Extending into a single hole.} Let $\overline{r}=\sup_{\e>0}\max_{i}r_{\e,i,0}$ and $\frac1{\overline r}=\inf_{\e>0}\min_{i}r_{\e,i,0}$ denote the largest and smallest possible radius in the reference configuration, respectively. For any invertible matrix $A\in\R^{3\times3}$ and any $\frac1{\overline r}\le r\le\overline{r}$ there exists an linear extension operator $E_{A}\colon W^{1,p}(B_{2\overline{r}}\setminus B_{r})\to W^{1,p}(B_{2\overline{r}})$ such that 
    \begin{align*}
    \|E_A\phi\|_{L^p(B_{2\overline{r}})}&\leq C_p|A||A^{-1}|\|\phi\|_{W^{1,p}(B_{2\overline{r}}\setminus B_{r})},\\
    \|\nabla E_A\phi\|_{L^p(B_{2\overline{r}})}&\leq C_p|A||A^{-1}|\|\nabla \phi\|_{L^p(B_{2\overline{r}}\setminus B_{r})},\\
    \|e(AE_A\phi)\|_{L^p(B_{2\overline{r}})}&\leq C_p\|e(A\phi)\|_{L^p(B_{2\overline{r}}\setminus B_{r})}
    \end{align*}
    where $C_p$ does not depend on the matrix $A$.
    This has been shown in \cite[Lemma 4.2]{gahn_extension_2024}.
    Note that the optimal constant $C_p$ decreases monotonically with the radius $r$ so we may take $C_p:=C_p(\overline{r})$
        
    Via an typical scaling argument, we find that $E_{\e,A}\colon W^{1,p}(B_{2\e^\alpha\overline{r}}\setminus B_{\e^\alpha r})\to W^{1,p}(B_{2\e^\alpha\overline{r}})$ defined via $E_{\e,A}\phi(x)=E_{A}\phi(\frac{x}{\e^\alpha})$ satisfies
    \begin{align*}
    \|E_{\e,A}\phi\|_{L^p(B_{2\e^\alpha\overline{r}})}&\le C_p|A||A^{-1}|\left(\|\phi\|_{L^{p}(B_{2\e^\alpha\overline{r}}\setminus B_{\e^\alpha r})}+\e^\alpha\|\nabla\phi\|_{L^{p}(B_{2\e^\alpha\overline{r}}\setminus B_{\e^\alpha r})}\right),\\
    \|\nabla(E_{\e,A}\phi)\|_{L^p(B_{2\e^\alpha\overline{r}})}&\le C_p|A||A^{-1}|\|\nabla \phi\|_{L^p(B_{2\e^\alpha\overline{r}}\setminus B_{\e^\alpha r})},\\
    \|e(AE_{\e,A}\phi)\|_{L^p(B_{2\e^\alpha\overline{r}})}&\le C_p\|e(A\phi)\|_{L^p(B_{2\e^\alpha\overline{r}}\setminus B_{\e^\alpha r})}.
    \end{align*}
\\[2mm]
    \textit{Step 2: Extending from $\Omega_\e$ to $\Omega$.} Let $A=(A_i)_{i\in\indexI}\subset\R^{3\times3}$ be a set of invertible matrices and set $\overline{A}:=\max_{i}|A_i|$ and $\underline{A}:=\max_{i}|A_i^{-1}|.$
    This implies $|A_i||A_i^{-1}|\le |\overline A||\underline A|$ for all $i\in\indexI$.
    Using the result from $(i)$, we can then construct a linear extension operator $\mathcal{E}_A\colon W^{1,p}(\Omega_\e)\to W^{1,p}(\Omega)$ by using the matrix $A_i$ for the $i$-th hole and the operator $E_{A_i}\colon W^{1,p}(B_{2\e^\alpha\overline{r}}(z_{\e,i})\setminus B_{\e^\alpha r_{\e,i,0}}(z_{\e,i}))\to W^{1,p}(B_{2\e^\alpha\overline{r}}(z_{\e,i}))$ via scaling with $\e^\alpha$ and shift to the center point $z_{\e,i}$.
    Gluing these together works because the holes are well-separated on the order of $\e$ as long as $\e$ is small enough such that $2\overline{r}<\e^{1-\alpha}$.
    We have
    \[
    \|\mathcal{E}_A\phi\|_{L^p(\Omega)}^p=\|\phi\|_{L^p(\Omega\setminus\bigcup_iB_{2\e^\alpha\overline{r}}(z_{\e,i}))}^p
    +\sum_{i}\|E_{A_i}\phi\|_{L^p(B_{2\e^\alpha\overline{r}}(z_{\e,i})}^p
    \]
    and use the estimates from step $(i)$ to conclude
    \begin{align*}
    \|\mathcal{E}_A\phi\|_{L^p(\Omega)}^p&\lesssim_p(\overline{A}\underline{A})^p\left(\|\phi\|_{L^p(\Omega_\e)}^p+\e^{\alpha p}\sum_{i}\|\nabla\phi\|^p_{L^{p}(B_{2\e^\alpha\overline{r}}\setminus B_{\e^\alpha})}\right),\\
    \|\nabla(\mathcal{E}_A\phi)\|^p_{L^p(\Omega))}&\lesssim_p(\overline{A}\underline{A})^p\|\nabla \phi\|^p_{L^p(\Omega_\e)},\\
    \|e_{\e,A}(\mathcal{E}_A\phi)\|^p_{L^p(\Omega))}&\lesssim_p\|e(\phi)\|^p_{L^2(\Omega_\e\setminus\bigcup B_{2\e^\alpha\overline{r}})}+\sum_i\|e(A_i\phi)\|^p_{L^p(B_{2\e^\alpha\overline{r}}(z_{\e,i})\setminus B_{\e^\alpha}(z_{\e,i}))}.
    \end{align*}
\end{proof}

\begin{lemma}[Korn's inequality for transformed symmetric gradients]
\label{lemma:Korn}
Let $\lambda\ge\overline r$ and $r\in\mathcal R_\e(S)$ with $r_i(t)\in[\lambda^{-1},\lambda]$ for all $t\in S$.
    Then,
    \[
    \| u\|_{H^1(\Omega_\e)}\lesssim_{\lambda}\|e_{\e}(u)\|_{L^2(\Omega_\e)}
    \]
    for all $u\in H^1_{\Sigma_2\cup\Gamma_\e}(\Omega_\e)^3$.
\end{lemma}
\begin{proof}
    Such generalized Korn inequalities have been shown for related but different setups in \cite[Lemma 3.6]{wiedemann_homogenisation_2024} (periodically distributed balls of size $\e$ with functions vanishing only on the internal interface) and in \cite[Theorem 3.1]{gahn_extension_2024} (periodically distributed balls of size $\e$ with functions vanishing on the exterior surface).

    In our case, we have balls with sizes of order $\e^\alpha$ which may be non-periodically distributed and functions vanishing on (part of) the exterior surface as well as the interior interfaces.
    We note that on $\Omega_\e\setminus\mathcal U_\e(t)$, it holds $e_\e(u)=e(u)$.
    We use the same approach as in \cite[Theorem 3.1]{gahn_extension_2024} where extension operators like in \cref{lemma:extension_operators} are used to proof this generalized Korn inequality.

    To that end, let $(q_{\e,i})_{i\in\indexI}$ be a collection of points with $q_{\e,i}\in \Gamma_{\e,i}$ for all $i\in\indexI$.
    We set $A_i(t):=F_{\e}^{-1}(t,q_{\e,i})$ for $t\in S$ and note that (\cref{lemma:transform_estimates}).
    \[
    |A_i(t)||A_i^{-1}(t)|\lesssim_{\lambda}1
    \]
    We can now apply \cref{lemma:extension_operators} to extend to $\Omega$ via $\mathcal{E}_{\e,A(t)}$.
    Using the same approach of approximating $F_{\e}$ by $A$ as outlined in the proof of \cite[Theorem 4.4]{gahn_extension_2024}, we get the desired result.
\end{proof}

\begin{lemma}\label{lemma:times_derivative}
    For any $r\in\mathcal R_\e(S)$ and every $\phi\in W(S;\Omega_\e(t))$, it holds
    \[
    2\int_0^T\langle\partial_t\phi,\phi\rangle_{H^1(\Omega_\e(t))^*}
    =\|\phi(T)\|^2_{L^2(\Omega_\e(T))}-\|\phi(0)\|^2_{L^2(\Omega_\e)}+\e^\alpha\int_0^T\int_{\Gamma_\e(t)}\partial_tr |\phi(t)|^2\di{\gamma}\di{t}.
    \]
    Similarly,
    \begin{multline}
        2\int_0^T\langle\partial_t\phi^+,\phi\rangle_{H^1(\Omega_\e(t))^*}
        =
    2\int_0^T\langle\partial_t\phi,\phi^+\rangle_{H^1(\Omega_\e(t))^*}\\
    =\|\phi^+(T)\|^2_{L^2(\Omega_\e(T))}-\|\phi^+(0)\|^2_{L^2(\Omega_\e)}+\e^\alpha\int_0^T\int_{\Gamma_\e(t)}\partial_tr |\phi^+(t)|^2\di{\gamma}\di{t}.
    \end{multline}
    Here, $\phi^+=\max(\phi,0)$ denotes the positive part of $\phi$.
\end{lemma}
\begin{proof}
    The first identity can easily be seen by transforming back to the initial configuration via the diffeomorphism defined in \cref{eq:diffeomorphism_psieps}.

    For the second identity, the idea is the same.
    Let $J\in W^{1,\infty}(S;L^2(\Omega_\e))$ such that $\partial_t(J\hat\phi)\in L^2(S;H^1(\Omega_\e))^*$.
    We first show that
    \[
    2\langle\partial_t(J\hat\phi),\hat\phi^+\rangle_{H^1(\Omega_\e)}=\ddt (J\hat\phi^+,\hat\phi^+)_{L^2(\Omega_\e)}+(\partial_tJ\hat\phi^+,\hat\phi^+)_{L^2(\Omega_\e)}
    \]
    holds for almost all $t\in S$.
    For $J$ constant (in space and time), this is exactly \cite[Lemma 3.3]{wachsmuth_regularity_2016} but the proof can easily be extended to the given scenario:
    First, let us assume that $\partial_t\hat\phi,\partial_t(J\hat\phi)\in L^2(S\times\Omega_\e)$. 
    In this case, we can directly calculate
    \begin{align*}
    2\langle\partial_t(J\hat\phi),\phi^+\rangle_{H^1(\Omega_\e)}
    &=2(J\partial_t\hat\phi^+,\hat\phi^+)_{L^2(\Omega_\e)}
    +2(\partial_tJ \hat\phi^+,\hat\phi^+)_{L^2(\Omega_\e)}\\
    &=\ddt(J\hat\phi^+,\hat\phi^+)_{L^2(\Omega_\e)}-(\partial_tJ \hat\phi^+,\hat\phi^+)_{L^2(\Omega_\e)}+2(\partial_tJ \hat\phi^+,\hat\phi^+)_{L^2(\Omega_\e)}.
    \end{align*}
    For $\partial_t(J\hat\phi)\in L^2(S;H^1(\Omega))^*$ this follows via density as in the proof in \cite[Lemma 3.3]{wachsmuth_regularity_2016}. The statement follows after transforming back to the moving domain and integrating in time.
    Finally,
    \[
    \int_0^T\langle\partial_t\phi^+,\phi\rangle_{H^1(\Omega_\e(t))^*}
        =
    \int_0^T\langle\partial_t\phi,\phi^+\rangle_{H^1(\Omega_\e(t))^*}
    \]
    since $\phi=\phi^+$ on $\{\phi>0\}$.
\end{proof}


\section{Proof of Lemma \ref{lem:M_eps}} \label{app:Delta.w}

We recall the definition of $w$ from \eqref{w_k}--\eqref{D_i} and  observe that  $-\Delta w^\e + \nabla q^\e$ is supported on $\bigcup_i \partial C_i^\e \cup \partial D_i^\e = \bigcup_i \partial D_i^\e \cup \Gamma_\e$. 
We define $\gamma_\e$ to be the part supported on $\Gamma_\e$ which consequently satisfies $\langle \gamma_\e, v \rangle = 0$ for all $v \in H^1_{\Gamma_\e}(\Omega_\e)$.
Then \eqref{decomposition.M.gamma} holds with $M_k^\e$, the columns of $M^\e$, being
    \begin{align} \label{M_eps^k}
        M^\e_k  = \e^{3-\alpha}\sum_i \left( m^\e_{k,i} + \dv(\1_{D_i^\e} (q_k^\e {\rm I}_3 - \nabla w_k^\e))\right) 
    \end{align}
    where 
    \begin{align} \label{m_k^eps}
         m^\e_{k,i}= \frac1{\e^{\alpha}r_{\e,i}} (q_k{\rm I}_3 - \nabla w_k)\left( \frac{x-z_{\e,i}}{\e^{\alpha} r_{\e,i}}\right) n |\partial B_{\dmin \e/4}| \delta_{\partial B_{\dmin \e/4}(z_{\e,i})}, && \delta_{\partial B_{\dmin \e/4}(z_{\e,i})} = \frac{\mathcal H^2|_{\partial B_{\dmin \e/4}(z_{\e,i})}}{|\partial B_{\dmin \e/4}(z_{\e,i})|},
    \end{align}
    where $n$ is the unit normal on $\partial B_{\dmin \e/4}(z_{\e,i})$.
    By \cite[Lemma 2.3.5]{allaire_homogenization_1991} (which follows from the fact that $w_k,q_k$ asymptotically behave as the fundamental solution of the Stokes equations) and Stokes law, i.e.,
    \begin{align*}
        \int_{\partial B_1(0)} (p {\rm I}_3-\nabla w_k) \cdot n \dd \mathcal H^2 = 6 \pi e_k,
    \end{align*}
    we have 
    \begin{align*}
        m^\e_{k,i} =  3 \pi \e^\alpha r_{\e,i} \left( e_k+ 3  n_k n +  \e^{\alpha -1} R^\e_{k,i}\right) \delta_{\partial B_{\dmin \e/4}(z_{\e,i})}, &&
        \|R^\e_{k,i}\|_{L^{\infty}(\partial B_{\dmin \e/4})} \lesssim 1.
     \end{align*}
     To conclude the proof, it suffices to show that for all  $\psi \in H^1(\R^3)$
     \begin{align}
        \langle \mathcal R_k - 6 \pi \e^3 \sum_i  r_{\e,i} e_k \delta_{B_{\dmin \e/4}(z_{\e,i})},\psi \rangle &\lesssim  \|f_\e^{app}(0) - f_0\|_{(H^1_xW_\sigma^{1,\infty})^\ast} \|\psi\|_{H^1(\R^3)} \nonumber \\
        & \qquad \qquad \qquad \qquad \qquad + (\sqrt{\eta} + \e) \|\psi\|_{L^2(\R^3)},    \label{est:M.main} \\
        \Bigl\| 3 \pi \e^3 \sum_i r_{\e,i} \Bigl(2 e_k \delta_{B_{\dmin \e/4}(z_{\e,i})} - &\left( e_k + 3  n_k n\right)\delta_{\partial B_{\dmin \e/4}(z_{\e,i})} \Bigr) \Bigr\|_{H^{-1}(\R^3)}  \lesssim  \e,  \label{est:M.remainder.0}\\
          \e^{3-\alpha} \Bigl\| \sum_i \dv(\1_{D_i^\e} (q_k^\e {\rm I}_3 - \nabla w_k^\e)) \Bigr\|_{H^{-1}(\R^3)}  &\lesssim  \e \label{est:M.remainder.2}, \\
         \Bigl\langle \e^{3} \sum_i R^\e_{k,i} \delta_{\partial B_{\dmin \e/4}(z_{\e,i})} , \psi \Bigr\rangle  &\lesssim \|\psi\|_{L^2(\R^3)} +  \e \|\nabla  \psi\|_{L^2(\R^3)}.  \label{est:M.remainder.1} 
     \end{align}
     For \eqref{est:M.main}, we observe that
     \begin{align*}
         \langle \mathcal R_k - 6 \pi \e^3 \sum_i  r_{\e,i} e_k \delta_{B_{\dmin \e/4}(z_{\e,i})}, \psi \rangle = \int_{\Omega \times (0,\infty)} 6 \pi \psi_k  \sigma \dd (f - f_\e^{app}), 
     \end{align*}
     where $f_\e^{app}$ is defined as in \eqref{f^app}.
     In particular, we can apply Lemma \ref{f^app,f}  (using that all relevant radii are upper bounded) such that we can replace $\sigma$ by $\min\{\sigma,\sigma_{
     \max}\}$ together with Remark \ref{rem:improved.f.f^app}  to deduce \eqref{est:M.main}.

     To prove \eqref{est:M.remainder.0}, we  observe
     \begin{align*}
         \fint_{\partial B_{\dmin \e/4}}  \left( e_k + 3  n_k n\right) \dd x = 2 e_k.
     \end{align*}
    Hence, the Poincar\'e type inequality \eqref{Poincare.averages.0} implies
     \begin{align*}
        & \Bigl| \fint_{B_{\dmin \e/4}(x_{\e,i})} 2 \psi_k \dd x-   \fint_{\partial B_{\dmin \e/4}(z_{\e,i})} \psi \cdot \left( e_k + 3  n_k n\right) \dd x\Bigr |  \\
         & =   \Bigl| \fint_{\partial B_{\dmin \e/4}(z_{\e,i})} \left(\psi -\fint_{B_{\dmin \e/4}(x_{\e,i})} \psi \right)  \cdot \left( e_k + 3  n_k n \right) \dd x \Bigr| \\
         &\lesssim  \e^{- \frac 1 2}  \|\nabla \psi\|_{L^2( B_{\dmin \e/4}(z_{\e,i}))}.
     \end{align*}
     Summing over $i$ and using that $r_{\e,i}$ is uniformly bounded yields \eqref{est:M.remainder.0}.
     
     \smallskip
     
          We turn to \eqref{est:M.remainder.2}. We use the pointwise estimates (see  \cite[Eq. (28)]{HoeferHuebnerRosenau26}) \begin{align*}
              |q^\e(x)| + |\nabla w^\e(x)| (x) \lesssim  \frac{\e^\alpha}{|x - z_{\e,i}|^2} \e   \qquad \text{in }   D_i^\e
          \end{align*} to bound
     \begin{align*}
          \e^{6-2\alpha} \Bigl\| \sum_i \dv(\1_{D_i^\e} (q_k^\e {\rm I}_3 - \nabla w_k^\e)) \Bigr\|^2_{H^{-1}(\R^3)} &\leq \e^{6-2\alpha}   \sum_i \|q_k^\e {\rm I}_3 - \nabla w_k^\e\|^2_{L^2(D_i^\e)}  \\
         & \lesssim \e^{6-2\alpha}     \e^{2\alpha -4}\\
         &\lesssim \e^2 .
     \end{align*}

     It remains to show \eqref{est:M.remainder.1}. Using again \eqref{Poincare.averages}
     we have
     \begin{align*}
         \Bigl| \fint_{\partial B_{\dmin \e/4}(z_{\e,i})} \psi \dd x \Bigr|& \lesssim   \fint_{\partial B_{\dmin \e/4}(z_{\e,i})} \Bigl|v  - \fint_{ B_{\dmin \e/4}(z_{\e,i})} v \dd y \Bigr| \dd x + \fint_{ B_{\dmin \e/4}(z_{\e,i})} |\psi|   \\
         &\lesssim  \e^{-\frac 1 2}   \|\nabla \psi\|_{L^2( B_{\dmin \e/4}(z_{\e,i}))} + \e^{-3/2} \|\psi\|_{L^2( B_{\dmin \e/4}(z_{\e,i}))}.
     \end{align*}
     Thus, 
     \begin{align*}
         \Bigl\langle  \e^{3} \sum_i R^\e_{k,i} \delta_{\partial B_{\dmin \e/4}(z_{\e,i})}, \psi \Bigr\rangle &\lesssim   \e^{3 } \sum_i \e^{-\frac 1 2}   \|\nabla \psi\|_{L^2( B_{\dmin \e/4}(z_{\e,i}))} + \e^{-3/2} \|\psi\|_{L^2( B_{\dmin \e/4}(z_{\e,i}))}  \\
         & \lesssim  \e \| \nabla \psi\|_{L^2(\R^3)} + \|\psi\|_{L^2(\R^3)}.
     \end{align*}
     This finishes the proof.

\bibliographystyle{abbrv}
\bibliography{references}

\end{document}